\documentclass[10pt]{amsart}
\usepackage[a4paper,top=2cm,bottom=2cm,left=2cm,right=2cm,marginparwidth=1.75cm]{geometry}
\usepackage[utf8]{inputenc} 
\usepackage[T1]{fontenc} 

\usepackage[english]{babel} 
\usepackage{graphicx} 
\usepackage[colorlinks]{hyperref} 
\usepackage{amsfonts, amsmath, amssymb, amsthm} 
\usepackage{stmaryrd}
\usepackage{bm}
\usepackage{xcolor}
\usepackage{tikz}
\usepackage{mathrsfs}

\newcounter{aah}
\newcounter{bbh}
\numberwithin{equation}{section}
\counterwithin{aah}{section}
\counterwithin{bbh}{section}
\newtheorem{thm}[aah]{Theorem}
\newtheorem{prop}[aah]{Proposition}
\newtheorem{coro}[aah]{Corollary}
\newtheorem{lem}[aah]{Lemma}

\theoremstyle{definition}

\newtheorem{defi}[bbh]{Definition}
\newtheorem{propdef}[bbh]{Definition-Proposition}

\theoremstyle{remark}

\newtheorem{exemple}{Example}[section]
\newtheorem{rem}{Remark}[section]

\title[Stochastic parallel transport on the Wasserstein space]{Stochastic parallel transport on the Wasserstein space and equivariant diffusions on the group of diffeomorphisms over a closed Riemannian manifold}
\author{Aymeric Martin}
\date{February 2026}

\newcommand{\T}{\Bar{T}}
\newcommand{\grad}{\Bar{\nabla}}
\renewcommand{\P}{\mathscr{P}}
\newcommand{\SN}{\underset{i=1}{\overset{N}{\sum}}}

\newcommand{\C}{\mathcal{C}}

\renewcommand{\L}{\bar{\mathcal{L}}}

\newcommand{\vol}{\mathrm{vol}}
\newcommand{\h}{\mathfrak{h}}
\renewcommand{\d}{\mathrm{d}}
\newcommand{\Ad}{\mathrm{Ad}}
\newcommand{\V}{\mathscr{V}}
\newcommand{\dw}{\bar{\d}}
\newcommand{\dd}{\bm \d}
\newcommand{\he}{\mathrm{H}}
\renewcommand{\H}{\mathscr{H}}
\renewcommand{\div}{\mathrm{div}}
\newcommand{\Int}{\mathrm{Int}}
\newcommand{\Grad}{\mathrm{Grad}}

\begin{document}

\maketitle

\begin{abstract}
	In this work, we establish the existence of solutions to stochastic differential equations on the Wasserstein space over a closed Riemannian manifold, under suitable regularity assumptions on the driving vector fields. Interpreting the diffeomorphism group $\mathscr{D}$ as a Riemannian submersion onto the smooth Wasserstein space $\P_\infty$, we further prove the existence and uniqueness of the stochastic parallel transport along diffusions on $\P_\infty$. Finally, we show that equivariant diffusions on $\mathscr{D}$ endowed with a principal bundle structure over $\P_\infty$ admit a unique factorization into a horizontal diffusion and a vertical component expressed as a right exponential of a process taking values in the Lie algebra $\mathfrak{g}$ of the group $G$ of volume preserving diffeomorphisms.
\end{abstract}

\tableofcontents

\section{Introduction}

Over the past few years, optimal transport has enjoyed a significant resurgence in popularity, mainly due to the diversity of application fields it covers such as data sciences or image processing. Among the main actors, we can cite Brenier in \cite{brenier1991polar} with the proof of the fact that the optimal transport map solution to the Monge-Kantorovich problem with quadratic distance in the euclidean case is given by the gradient of a convex function, solution to a Monge-Ampère equation. Years later, this result is extended by McCann to the case of compact Riemannian manifolds in \cite{mccann2001polar} with the optimal transport map given by the Riemannian exponential of a gradient of function. In \cite{benamou2000computational}, Benamou and Brenier introduce a variational formulation of the Wasserstein distance, namely,

$$
W_2^2(\mu,\nu) = \inf \left \{ \int_0^1 \int_M |V_t|^2 \d\mu_t \, : \, \partial_t \mu_t + \nabla \cdot (V_t \mu_t) = 0 \right\},
$$ 
where the infimum is taken in the set of the measures $(\mu_t)_{t \geq 0}$ solution of the above equation in the sense of distribution.
In his seminal paper \cite{Otto31012001}, Otto introduces a Riemannian structure on the Wasserstein space which will be formalized in the more general context of metric spaces by Ambrosio, Gigli and Savaré in \cite{ambrosioGradientFlowsMetric2008}. In \cite{lott2007geometriccalculationswassersteinspace}, Lott introduces a lot of geometric material, including Levi-Civita connection, Riemann curvature tensor and the parallel transport equation. The reader is refered to the reference book \cite{villani2009optimal} by Villani for results about optimal transport and Otto's calculus.

The study of stochastic differential equations on Wasserstein space can be motivated by mean field games theory. Indeed, in the reference book \cite{carmona2018probabilistic2} by Carmona and Delarue, we can see that such SDEs on Wasserstein space can be seen as a limit of an interacting particles system with common noise, also known as non linear stochastic Fokker-Planck equation. Still in \cite{carmona2018probabilistic2}, it is shown that a solution is given by taking the data of a well chosen conditionnal McKean-Vlasov equation. In the article \cite{gassiat2022long}, Gassiat, Lions and Sougadinis suggest the idea of looking at such stochastic partial differential equations with coefficients depending in a local way on the density $(\rho_t)_{t \geq 0}$ of the solution. In \cite{ding2023stochasticdifferentialequationsstochastic} the authors prove the existence of solutions to SDEs on the Wasserstein space over a closed Riemannian manifold driven by constant vector fields, while, in \cite{wang2020}, the existence of solutions to a general class of conditional McKean-Vlasov SDEs is done on the Wasserstein space over the Euclidean space.

The origin of the manifold structure on the group of diffeomorphisms $\mathscr{D}$ over a compact manifold can be traced back in the late $1960s$. In the reference book \cite{Palais66}, Palais developed a general framework for the study of Banach space-valued section functors. A fundamental example within this framework is provided by the functor assigning to a given fiber bundle the space of its $H^s$ sections. These results are used in \cite{omori_group_1970}, where Omori presents the topology of $\mathscr{D}$ as an Inverse Limit Hilbert Lie group topology, i.e 

$$
	\mathscr{D} = \cap_{s} \mathscr{D}^s
$$
where $\mathscr{D}^s$ is the set of invertible maps in the sobolev space $H^s(M,M)$ with inverse in $H^s(M,M)$. Notice that even if $\mathscr{D}$ is a Lie group, this is not the case of $\mathscr{D}^s$ since the left multiplication is not smooth, but this point of view allows us to work on Hilbert manifolds. In \cite{EbinMarsden1970}, Ebin and Marsden exploited this framework to establish well-posedness for the classical Euler equations of an ideal fluid. Moreover, they showed that, with respect the Inverse Limit Hilbert topology, the diffeomorphism group $\mathscr{D}$ is diffeomorphic to $G \times \P_\infty$, where $G$ denotes the subgroup of volume-preserving diffeomorphisms.

Even if the Wasserstein space has a quasi-Riemannian structure, the lack of local coordinates is a huge problem to construct stochastic parallel transport. Indeed, on a Riemannian manifold, it is sufficient to place oneself in local coordinates and to solve a stochastic differential equation locally, see \cite{Ito75}, \cite{IkedaWatanabe81} for instance.
The question of stochastic parallel transport on the Wasserstein space has already been adressed in \cite{ding2023stochasticdifferentialequationsstochastic} by Ding, Fang and Li. In the latter, the stochastic parallel transport is formulated by means of a covariant SDE and existence and uniqueness are proved in the case of the smooth Wasserstein space over the torus along a certain type of diffusion. In order to prove the existence of stochastic parallel transport along a diffusion on $\P_\infty$ over a closed manifold, one could consider the approach used to prove the existence of parallel transport along regular curves in the reference work \cite{gigli2011second} by Gigli. However, this method does not fit well with the stochastic framework. The idea presented in the present article is to see the group of diffeomorphisms $\mathscr{D}$ as a Riemannian submersion over $\P_\infty$. The well-posedness of the stochastic parallel transport problem is a further step towards a better understanding of Wasserstein's stochastic differential structure, in particular, it might be possible to study some processes coupling to obtain stability results about diffusions.

The fact that $\mathscr{D}$ can be seen as a principal bundle above $\P_\infty$ endowed with a connection can be used to study a specific type of diffusions on $\mathscr{D}$, namely equivariant diffusions (see Subsection \ref{subsection 2.3} and Section \ref{Section 6}). We prove that we can write any equivariant diffusions $(\Phi_t)_{t \geq 0}$ in the form

$$
\Phi_t = h_t \cdot g_t, \quad \forall t \geq 0
$$
with $(h_t)_{t \geq 0}$ a horizontal equivariant diffusion and $(g_t)_{t \geq 0}$ a right exponential on the group of volume preserving diffeomorphisms $G$. Equivariant diffusions in the finite dimensional case have been studied by Elworthy, Le Jan and Li in \cite{ElLeLi04} and \cite{ELJW}. Note that the idea of looking at such diffusions is not new, for example, Itô investigated the case of the orthonormal frame bundle over a Riemannian manifold in \cite{Ito63}. The method used in the present paper is essentially the same as in the finite dimensional case, even if Elworthy, Le Jan and Li worked with the generators of the diffusions, whereas we tend to work directly on the SDEs. The advantage of our case lies in the fact that $\mathscr{D}$ is a Lie Group and $G$ is a subgroup of $\mathscr{D}$. However, we write the proof and the results using only the fact that $\mathscr{D}$ is a principal bundle over $\P_\infty$ to stay as general as possible. Lastly, we use the structure of Riemannian submersion to study Itô's stochastic differential equation driven by vertical vector fields on $\mathscr{D}$. 

\begin{figure}[t]
    \centering
    \includegraphics[width=0.8\textwidth]{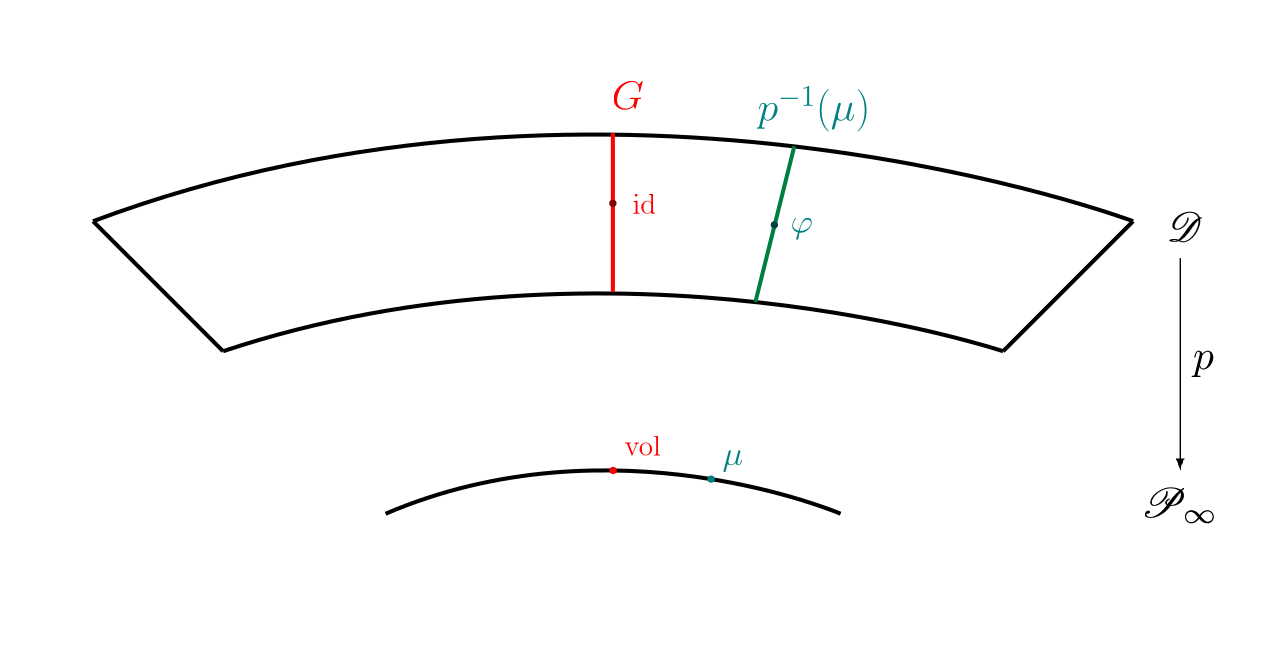}
    \caption{A way to visualize the $G$-principal bundle structure of $\mathscr{D}$ over $\P_\infty$.}
\end{figure}

The article is organized as follows: Section \ref{Section 2} is dedicated to the presentation of the arguments that will be used later in Sections \ref{Section 5} and \ref{Section 6} in the simpler case of finite-dimensional Riemannian manifolds. In Section \ref{Section 3} we present the geometry of the Wasserstein space, the one of the diffeomorphisms group as well as the Riemannian submersion structure that establish a link between these two spaces. Section \ref{Section 4} is about stochastic calculus on the two infinite-dimensional manifolds presented above, in particular, we prove the following:

\begin{thm}[Informal version of Theorem \ref{Existence Unicité}]
    Let $\bar{Z}_0, \bar{Z}_1, \dots, \bar{Z}_N$ be regular vector fields on $\P$. Then, the SDE
    
    $$
    \dw^{\grad} \mu_t = \SN \bar{Z}_i(\mu_t) \d W^i_t + \bar{Z}_0(\mu_t) \d t, \quad \mu_0 = \mu \in \P
    $$
    admits a solution on $\P$ given by $\mu_t =(X_t)_* \mu$ where $(X_t)_{t \geq 0}$ is the solution to a conditional McKean-Vlasov equation. Moreover, if $\mu \in \P_\infty$, $(\mu_t)_{t \geq 0}$ lies in $\P_\infty$.
\end{thm}
This theorem is an extension of \cite[Proposition 3.10]{ding2023stochasticdifferentialequationsstochastic} to the setting where the diffusion is not driven by constant vector fields, as well as a generalization of \cite[Theorem 2.1]{wang2020} to the context of closed manifolds. In Section \ref{Section 5} we prove the main theorem of the present paper, namely:

\begin{thm}[Informal version of Theorem \ref{Existence unicité transport parallèle}]
    Let $(\mu_t)_{t \geq 0}$ be a diffusion on $\P_\infty$ and $V \in T_{\mu_0}\P_\infty$. Then, there exists a unique stochastic parallel transport (see Definition \ref{Definition transport parallele}) on $T\P_\infty$ along $(\mu_t)_{t \geq 0}$.
\end{thm}
This theorem extends the main result of \cite{ding2023stochasticdifferentialequationsstochastic} to the setting where the diffusion is not driven by constant vector fields and the underlying manifold is an arbitrary closed manifold. The main difference with the result of \cite{ding2023stochasticdifferentialequationsstochastic} is that the stochastic parallel transport preserves the smoothness of the initial condition. To the best of our knowledge, this property is new, even in the deterministic setting. To establish the well-posedness of the stochastic parallel transport, we adapt an argument due to Ebin and Marsden and obtain the following result, which we believe to be of independent interest:

\begin{thm}(Informal version of Theorem \ref{Dérivées covariantes projection horizontale et verticale})
    The horizontal and vertical projections given by the Hodge decomposition are smooth bundle morphisms from $T\mathscr{D}$ to $T\mathscr{D}$.
\end{thm} 

Finally, in Section \ref{Section 6}, we study some decompositions of equivariant diffusions on $\mathscr{D}$.

\vspace{10pt} 

\noindent \textbf{Notations.} Throughout the paper we consider $(\Omega, \mathcal{F}, (\mathcal{F}_t)_{t \geq 0}, \mathbb{P})$ a filtered probability space which supports a $N$-dimensional Brownian motion $W=(W^1,W^2, \dots, W^N)$ in the canonical basis. Let $k \in \mathbb{N}^*$ and $\mathscr{E}_1, \mathscr{E}_2, \dots, \mathscr{E}_k$ be some sets. The \textit{$k$-th symmetric product} is defined as follows:

\begin{equation} \label{k symmetric product}
    \mathrm{SP}^k \left( \prod_{1 \leq j \leq k} \mathscr{E}_j \right)  = \bigcup_{\sigma \in \mathscr{S}_k} \left( \prod_{1 \leq j \leq k} \mathscr{E}_{\sigma(j)} \right).
\end{equation}
Let $(M, \langle \bullet, \bullet \rangle)$ be a smooth Riemannian manifold and $(X_t)_{t \geq 0}$ be solution to

$$
\circ \d X_t = \SN A_i(X_t) \circ \d W^i_t + A_0(X_t) \d t, \quad X_0 = x \in M,
$$
where $A_i$ is a smooth vector field for all $i=0,1, \dots, N$. Since Stratonovich's differentials preserves the chain rule, we can see Stratonovich's differential of $M$-valued semimartingales as elements of the tangent bundle. We will often write geometric expressions with Stratonovich's differential, for example, let $\Xi$ be a linear connection or a $(1,2)$-tensor and $B \in \Gamma(TM)$ a smooth vector field, the following expression

$$
\Xi(\circ \d X_t, B(X_t))
$$
has to be understood as

\begin{equation} \label{Notation intro}
\SN \Xi(A_i(X_t), B(X_t)) \circ \d W^i_t + \Xi(A_0, B(X_t)) \d t.
\end{equation}
For instance, let $C \in \Gamma(TM)$, if $\Xi$ is the Levi-Civita connection $\nabla$, we have:

\begin{equation} \label{Notation intro 2}
\circ \d \langle B(X_t), C(X_t) \rangle_{X_t} = \langle \nabla_{\circ \d X_t} B(X_t), C(X_t) \rangle_{X_t} + \langle B(X_t), \nabla_{\circ \d X_t} C(X_t) \rangle_{X_t},
\end{equation}
in the usual sense, namely,
\begin{align*}
    \langle B(X_t), C(X_t) \rangle - \langle B(x), C(x) \rangle & = \SN \int_0^t \left( \langle \nabla_{A_i} B(X_s), C(X_s) \rangle + \langle B(X_s), \nabla_{A_i} C(X_s) \rangle \right) \circ \d W^i_t \\
    & +\int_0^t \left( \langle \nabla_{A_0} B(X_s), C(X_s) \rangle + \langle B(X_s), \nabla_{A_0} C(X_s) \rangle \right) \d s.
\end{align*}
This is due to the fact that Stratonovich's calculus behave as usual differential calculus. We will use this convention on usual Riemannian manifolds but also on $\P$ and $\mathscr{D}$ (see Section \ref{Section 5} and \ref{Section 6}). We will also use the following notation to denote the quadratic variation:

$$
\Xi (\circ \d X_t, \circ \d X_t) = \SN \Xi(A_i(X_t), A_i(X_t)) \d t.
$$

Let $f \in \C^\infty(M,N)$ a smooth map between two manifolds. Let $x \in M$ and $V \in T_xM$, we adopt the notation $T_xf(V)$ to denote the differential of $f$ in the direction $V$. This choice is made to avoid confusions with the Stratonovich's or Itô's differentials (respectively $\circ \d$ and $\d^\nabla$) or covariant differentials (respectively $\circ D_t$ and $D^\nabla_t$).

Let $\pi : F \rightarrow M$ be a fiber bundle over $M$. We will denote by $\Gamma^k(F)$ the set of $\C^k$ sections of $F$ while $\Gamma^{(s)}(F)$ denote the set of sections of $F$ with Sobolev regularity $H^s$.

From Section \ref{Section 3}, the reader will encounter both over-lined characters and bold characters. The former are objects of the Wasserstein space when the latter belong to the Diffeomorphisms group. 
\vspace{10pt}

\noindent \textbf{Acknowledgements.} The author is grateful to his PhD advisors, Marc Arnaudon and Michel Bonnefont, for their valuable support and insightful suggestions, which have been of great help throughout this work.

\section{Finite-dimensional case} \label{Section 2}

\subsection{Stochastic differential equations on Riemannian manifolds}
In this subsection we recall some results about stochastic differential equations on Riemannian manifolds. The reader is refered to \cite{Elworthy1982}, \cite{Emery89}, \cite{hsu2002stochastic}, \cite{IkedaWatanabe81} or \cite{kunita1997stochastic} for a complete presentation of the theory of stochastic differential equations on manifolds. Let $A_0,A_1, \dots, A_N$ be smooth vector fields on a smooth manifold $M$. A continuous $M$-valued stochastic process $(X_t)_{t \geq 0}$ is said to be solution to the following Stratonovich's SDE

$$
\circ \d X_t = \SN A_i(X_t) \circ \d W^i_t + A_0(X_t) \d t, \quad X_0 = x \in M,
$$
if for all $f \in \C^\infty(M)$,

$$
f(X_t) =f(x) + \int_0^t \SN A_i f(X_s) \circ \d W^i_s + \int_0^t A_0f(X_s) \d s,
$$
or, in an usual Itô formulation,

$$
f(X_t) =f(x) + \int_0^t \SN A_i f(X_s)  \d W^i_s + \int_0^t \left( A_0f(X_s) + \frac{1}{2} \SN A_i^2 f(X_s) \right)\d s.
$$
If one wants to study Itô's stochastic differential equations on $M$, a connection on $TM$ is required. For instance, if $M$ is a Riemannian manifold, one may consider the Levi-Civita connection $\nabla$ induced by the metric. In this case, a continuous $M$-valued stochastic process $(X_t)_{t \geq 0}$ is solution to the following Itô's SDE

$$
\d^\nabla X_t = \SN A_i(X_t) \d W^i_t + A_0(X_t) \d t, \quad X_0 = x \in M,
$$
if for all $f \in \C^\infty(M)$,

$$
f(X_t) = f(x) + \SN \int_0^t A_i f(X_s) \d W^i_s + \int_0^t \left( A_0f(X_s) + \frac{1}{2} \SN \he f(A_i(X_s),A_i(X_s)) \right) \d s,
$$
where $\he f$ denotes the Hessian associated to $\nabla$, namely:

$$
\he f(A,B) = \mathcal{L}_A\mathcal{L}_B f - \nabla_A B f = \langle \nabla_A \nabla f, B \rangle_x, \quad \forall A,B  \in \Gamma(TM).
$$

\begin{prop} \label{Conversion Ito strato dimension finie}
    The Stratonovich's SDE

$$
\circ \d X_t = \SN A_i(X_t) \circ \d W^i_t + A_0(X_t) \d t, \quad X_0 = x \in M,
$$
is equivalent to the following Itô's one:

$$
\d^\nabla X_t = \SN A_i(X_t) \d W^i_t + \left( A_0(X_t) + \frac{1}{2} \SN \nabla_{A_i} A_i(X_t) \right) \d t, \quad X_0 = x \in M.
$$
\end{prop}

\begin{proof}
The proof is a direct consequence of the definition of the Hessian $\he$.
\end{proof}

\begin{thm}[Stochastic Cauchy-Lipschitz]
	Let $A_0, A_1, \dots, A_N$ be smooth vector fields on a compact manifold $M$. Then, there exists a unique solution to the Itô's SDE defined for all $t \geq 0$:
	
	$$
\circ \d X_t = \SN A_i(X_t) \circ \d W^i_t + A_0(X_t) \d t, \quad X_0 = x \in M.
$$
\end{thm}

\begin{proof}
This result is classic and a proof can be found in \cite{hsu2002stochastic} or \cite{kunita1997stochastic} for example.
\end{proof}

Let $\pi : \mathscr{F} \rightarrow M$ be a smooth vector bundle endowed with a connection $\nabla$ and $J^0, J^1, \dots, J^N$ be $\C^2$ vector bundle morphisms from $\mathscr{F}$ to itself. A $\mathscr{F}$-valued stochastic process $(A_t)_{t \geq 0}$ is said to be solution to the following covariant SDE on $\mathscr{F}$:

\begin{equation} \label{E covariant M}
    \circ D_t A_t = \SN J^i_{X_t}(A_t) \circ \d W^i_t + J^0_{X_t}(A_t) \d t
\end{equation}
if 

\begin{equation*}
    \circ \d \tau_{0,t}^{-1} A_t = \SN \tau_{0,t}^{-1}J^i_{X_t}(A_t) \circ \d W^i_t + \tau_{0,t}^{-1}J^0_{X_t}(A_t) \d t,
\end{equation*}
where $(\tau_{0,t})$ is the stochastic parallel transport along $(X_t)_{t \geq 0}$ (see Proposition-Definition \ref{Prop def transport parallele stochastique M} below for the definition in the case $\mathscr{F}= TM$).

\begin{prop}[\cite{Elworthy1982}, Theorem 13E] 
    A stochastic process $(A_t)_{t \geq 0}$ on $\mathscr{F}$ is solution to \eqref{E covariant M} if and only if for all linear form $\Phi : \mathscr{F} \rightarrow \mathbb{R}$, $(\Phi_{X_t}(A_t))_{t \geq 0}$ is solution to:
    
    \begin{equation*}
        \circ \d \Phi_{X_t}(A_t) = (\nabla \Phi_{X_t})(\circ \d X_t)(A_t) + \Phi_{X_t}(\circ D_t A_t).
    \end{equation*}
\end{prop}

\begin{coro} \label{Chain rule dim finie}
    Let $(A_t)_{t \geq 0}$ be solution to \eqref{E covariant M} and $L : \mathscr{F} \rightarrow \mathscr{F}$ be a $\C^2$ vector bundle morphism. Then, $(L_{X_t}(A_t))_{t \geq 0}$ is solution to:
    
    \begin{equation*}
        \circ D_t L_{X_t}(A_t) = (\nabla L_{X_t})(\circ \d X_t)(A_t) + L_{X_t}(\circ D_t A_t).
    \end{equation*}
\end{coro}

\subsection{Stochastic parallel transport and Riemannian submersion} \label{Subsection 2.2}

In the case of a Riemannian submersion $p : N \rightarrow M$ between two Riemannian manifolds, one can establish a link between the Levi-Civita connection on $M$ and the one on $N$ for horizontal lift of vector fields on $M$, see Definition \ref{D définition horizontal lift}. As the Levi-Civita connection is closely related to the parallel transport, it is natural to expect an analogous relation between the two parallel transports. This link is presented in this section. None of these results are new and can be found in \cite[Chapter 3]{cheeger1975comparison} or in \cite{oneill1966fundamental} for example.

    \begin{propdef}
        Let $M$ be a Riemannian manifold with $\nabla$ the associated Levi-Civita connection and $(\gamma_t)_{t \geq 0}$ a smooth curve on $M$. There exists a unique solution to the following covariant equation:
        
        \begin{equation} \label{Equation transport parallèle dim finie}
            \nabla_{\dot{\gamma_t}}A_t = 0, \quad A_0 \in T_{\gamma_0}M.
        \end{equation}
        This unique solution is called the \textit{parallel transport} along $(\gamma_t)_{t \geq 0}$ with initial condition $A_0 \in T_{\gamma_0}M$.
    \end{propdef}

        \begin{rem} \label{Remarque transport flot}
            Since the covariant equation defining the parallel transport is linear, one may also see the parallel transport along a smooth curve $(\gamma_t)_{t \geq 0}$ as a flow of linear maps $(\tau_{0,t})_{t \geq 0}$ such that, $\tau_{0,t} : T_{\gamma_0}M \rightarrow T_{\gamma_t}M$, for all $t \geq 0$, mapping each initial condition $A_0 \in T_{\gamma_0}M$ to the associated solution to \eqref{Equation transport parallèle dim finie}. We will rather adopt this point of view throughout this paper.
        \end{rem}

    Let $N,M$ be smooth Riemannian manifolds and $p : N \rightarrow M$ be a Riemannian submersion with $\V_uP$ being $\ker T_u p$, namely the \textit{vertical space} and $\H_uN = \V_uN^\perp$ the \textit{horizontal space} for all $u \in N$. 
    
    \begin{defi} \label{D définition horizontal lift}
        The \textit{horizontal lift} $\mathfrak{h}_u$ at the point $u \in P$ is defined as the inverse of $T_u p|_{\H_u N}$. Let $\bar{A} \in \Gamma(TM)$, the horizontal lift of $\bar{A}$ is defined as $\mathfrak{h}(\bar{A})(u) = \mathfrak{h}_u(\bar{A}(p(u)))$.
    \end{defi}

    \begin{defi}
        Two vector fields $A,\bar{A}$ respectively on $N$ and $M$ are said to be \textit{$p$-related} if $Tp(A(u)) = \bar{A}(p(u))$ for all $u \in N$. 
    \end{defi}
    
     For $u \in N$, let $P_{\V_u}$ and $P_{\H_u}$ denote respectively the orthogonal projection on the vertical and the horizontal space. Moreover, let $\nabla, \bar{\nabla}$ denote respectively the Levi-Civita connection on $N$ and $M$. Note that $A = \h(\bar{A})$ if and only if $A$ and $\bar{A}$ are $p$-related and $A$ is horizontal.
    
    \begin{rem}
    Notice that we have inverted the notations used in \cite{cheeger1975comparison}. The reason is that we want to apply this theory to the context of the diffeomorphism group seen as a Riemannian submersion over Wasserstein space, and overlined symbols are common in the literature of Wasserstein space.
\end{rem}
    
    \begin{lem} \label{p-related crochet de lie}
    	Let $A,B \in \Gamma(TN)$ and $\bar{A},\bar{B} \in \Gamma(TM)$ be $p$-related vector fields. Then, $[\bar{A}, \bar{B}]$ and $[A,B]$ are $p$-related.
    \end{lem}
    
    \begin{proof}
    Let $f \in \C^\infty(M)$. Observe that for all $u \in N$ we have
    
    $$
    A(f \circ p)(u) = T_up(A) f(p(u)) = (\bar{A} f)(p(u)).
    $$
    Hence, iterating it with the function $\bar{A} f$ we get
    
    $$
    B(\bar{A}f \circ p)(u) = T_up(B) (\bar{A}f)(p(u)) = (\bar{B} \bar{A} f)(p(u)).
    $$
    Thus, we obtain the result :
    
    $$
    [A,B](f\circ p)(u) = T_up ([A,B]) f(p(u)) = ([\bar{A}, \bar{B}]f)(p(u)).
    $$
    \end{proof}
    
    \begin{prop} \label{Connexion Levi-Civita Horizontale dim finie}
    	Let $\bar{A}, \bar{B} \in \Gamma(TM)$ and $A,B$ their horizontal lifts. Then, $\nabla_AB$ and $\bar{\nabla}_{\bar{A}} \bar{B}$ are $p$-related. Moreover, we have
    	
    	\begin{equation} \label{E22}
    	    \nabla_A B(u) = \h_u \left( \grad_{\bar{A}} \bar{B}(p(u)) \right) + \frac{1}{2}[A,B]^\V(u),
    	\end{equation}
    	where $[A,B]^\V(u) = P_{\V_u}([A,B](u))$ and $[A,B]^\V(u)$ depends only on the values of $A,B$ at $u$.
    \end{prop}
    
    \begin{proof}
    Let $\bar{C} \in \Gamma(TM)$ and $C$ its horizontal lift. By using Koszul formula, Lemma \eqref{p-related crochet de lie} and the fact that $p$ is a Riemannian submersion we get
    
    \begin{align*}
    2\left \langle \nabla_AB, C \right \rangle & = A \langle B,C \rangle + B\langle C,A \rangle - C\langle A,B \rangle + \left \langle [A,B], C \right \rangle - \left \langle [A,C], B \right \rangle - \left \langle [B,C],A \right \rangle \\
    &= A \left( \langle \bar{B},\bar{C} \rangle \circ p \right) + B \left(\langle \bar{C},\bar{A} \rangle \circ p \right) - C \left(\langle \bar{A},\bar{B} \rangle \circ p \right) \\
    & \quad + \left \langle Tp([A,B]), \bar{C} \right \rangle \circ p - \left \langle Tp([A,C]), \bar{B} \right \rangle \circ p - \left \langle Tp([B,C]), \bar{A} \right \rangle \circ p \\
    & = \bar{A} \langle \bar{B}, \bar{C} \rangle \circ p + \bar{B} \langle \bar{C}, \bar{A} \rangle \circ p - \bar{C} \langle \bar{A}, \bar{B} \rangle \circ p + \langle [\bar{A}, \bar{B} ], \bar{C} \rangle \circ p - \langle [\bar{A}, \bar{C}], \bar{B} \rangle \circ p - \langle [\bar{B}, \bar{C} ], \bar{A} \rangle \circ p \\
    & = 2\left \langle \grad_{\bar{A}} \bar{B}, \bar{C} \right \rangle \circ p.
    \end{align*}
    Then, we obtain
    
    \begin{equation} \label{E23}
        \nabla_A B(u) = \h_u \left( \grad_{\bar{A}} \bar{B}(p(u)) \right) + P_{\V_u}( \nabla_A B(u)).
    \end{equation}
    Let $W$ be a vertical vector field on $N$, using once again Koszul formula we have
    
    $$
    2\left \langle \nabla_AB, W \right \rangle  = A \langle B,W \rangle + B\langle W,A \rangle - W\langle A,B \rangle + \left \langle [A,B], W \right \rangle - \left \langle [A,W], B \right \rangle - \left \langle [B,W], A \right \rangle. 
    $$
    The first two terms on the right hand-side are $0$ by orthogonality. The third vanishes since $\langle A,B \rangle$ is constant in vertical directions. Since $W$ is $p$-related with the null vector field on $M$, the last two terms vanish using Lemma \ref{p-related crochet de lie}. Thus we obtain
    
    $$
    \left \langle \nabla_AB, W \right \rangle = \left \langle \frac{1}{2}[A,B], W \right \rangle.
    $$
    To see that $[A,B]^\V|_u$ depends only on the values of $A,B$ at $u$, notice that $[A,B]^\V = P_\V(\nabla_{A}B)$ which is tensorial in the first component. Using the antisymmetry of the Lie bracket we get the announced property.
    \end{proof}
    
    We are now interested in the link between parallel transports on $N$ and $M$. Let $U \in \Gamma(\H N)$ be a horizontal vector field and $B \in \Gamma(TN)$. We define the \textit{O'Neill tensor} as in \cite{oneill1966fundamental} to be:
    
    \begin{equation*}
        \mathcal{A}_U(B) = P_{\V} \left( (\nabla U) \cdot P_{\H}(B) \right).
    \end{equation*}
    One can easily check that its adjoint operator is given by:
    
    \begin{equation*}
        \mathcal{A}^T_U(B) = P_{\H} \left( (\nabla U)^T \cdot P_{\V}(B) \right).
    \end{equation*}
    Moreover, by Proposition \ref{Connexion Levi-Civita Horizontale dim finie}, if $B \in \H_u N$, then 
    $$
    \mathcal{A}_U(B) = \frac{1}{2}[U,B]^\V = -\frac{1}{2}[B,U]^\V=  -\mathcal{A}_B(U).
    $$

    \begin{defi} \label{Lift chemin}
        Let $p : N \rightarrow M$ be a Riemannian submersion and $(\bar{\gamma}_t)_{t \geq 0}$ be a smooth curve on $M$ such that $\bar{\gamma}_0 = x \in M$. Then, for $u \in p^{-1}(x)$, the \textit{horizontal lift} of $\bar{\gamma}$ in $u$ is the solution of the following ODE:
        
        $$
        \dot{\gamma}_t = \mathfrak{h}_{\gamma_t}(\dot{\bar{\gamma}}_t), \quad \gamma_0 = u.
        $$
    \end{defi}
    
    Let $x \in M$ and $u \in p^{-1}(x)$. Suppose that $(\bar{\gamma}_t)_{t \geq 0}$ is a smooth path on $M$ with $\bar{\gamma}_0 =x$ and let $(\gamma_t)_{t \geq 0}$ be its horizontal lift starting from $u$. Let $(\tau_{0,t})_{t \geq 0}$ and $(\bar{\tau}_{0,t})_{t \geq 0}$ be respectively parallel transports on $N$ and $M$ along $(\gamma_t)_{t \geq 0}$ and $(\bar{\gamma}_t)_{t \geq 0}$ (with the point of view of Remark \ref{Remarque transport flot}).  The following proposition introduces the (deterministic version of the) strategy for proving Theorem \ref{Existence unicité transport parallèle}. The idea is that, when we consider a Riemannian submersion $p : N \rightarrow M$, we can construct the parallel transport on $M$ by using only the parallel transport on $N$ and Proposition \ref{Connexion Levi-Civita Horizontale dim finie}. We will also see that the horizontal lift of the parallel transport is the parallel transport of the \textit{right-exponential of the (inverse parallel transport of the) O'Neill tensor}.
    
    \begin{prop} \label{T parallel transport dim finie}
    The parallel transport on $M$ writes
    
    $$
    \bar{\tau}_{0,t}(\bar{U}) = T_{\gamma_t}p(\tau_{0,t}^h(\mathfrak{h}_u(\bar{U})), \quad \forall t \geq 0,
    $$
    where $(\tau_{0,t}^h)_{t \geq 0}$ is the so-called horizontal transport, namely the horizontal lift of $(\bar{\tau}_{0,t})_{t \geq 0}$,
    
    $$
    \tau_{0,t}^h(U) = \mathfrak{h}_{\gamma_t(u)}(\bar{\tau}_{0,t}(T_up(U))),\quad \forall U \in \H_u N,
    $$
    which is also the unique solution to 
    
    \begin{equation}\label{Equation transport parallele horizontal dim finie}
        D_t \tau^h_{0,t}(U) = -\mathcal{A}_{\dot{\gamma}_s}( \tau_{0,t}^h(U)), \quad \tau_{0,0}^h(U) = U, \, \forall U \in \H_u N. 
    \end{equation}
    Moreover, $\tau_{0,t}^h(U) = \tau_{0,t}(Q_t(U))$ where 
    
    \begin{equation*}
        Q_t(U) = \mathcal{E}\left( \int_0^\cdot \tau_{0,s}^{-1} \left(-\mathcal{A}_{\dot{\gamma}_s}( \tau_{0,s}(U)) \right) \right)_t,
    \end{equation*}
    with $\mathcal{E}$ the right-exponential map on $GL(T_uN)$.
    \end{prop} 
    \begin{proof}
    Let $(\tau^h_{0,t})_{t \geq 0}$ be the horizontal lift of $(\bar{\tau}_{0,t})_{t \geq 0}$ starting from $u$. By Proposition \ref{Connexion Levi-Civita Horizontale dim finie}, we have
    
    $$
    D_t \tau^h_{0,t}(U) = \frac{1}{2} \left[ \dot{\gamma}_t, \tau^h_{0,t}(U) \right ]^\V =  -\mathcal{A}_{\dot{\gamma}_s}( \tau_{0,t}^h(U)), \quad \forall U \in \H_u N.
    $$
    Let $L_t(\cdot)$ denote the linear map $\tau_{0,t}^{-1}\left(\frac{1}{2}[\dot{\gamma}_t,\tau_{0,t}(\cdot)]^\V \right)$ for all $t \geq 0$. Then, letting $(Q_t)_{t \geq 0} = (\tau_{0,t}^{-1} \tau_{0,t}^h)_{t \geq 0}$, the above equation is equivalent to the following $GL(T_uN)$-valued equation
    
    $$
	\dot{Q}_t = L_t(Q_t), \quad Q_0 = id.
    $$
    Indeed, 
    
    \begin{equation}\label{Trick finite dimension}
        \dot{Q}_t = \frac{\d}{\d t} \tau_{0,t}^{-1} \tau_{0,t}^h = \tau_{0,t}^{-1} \nabla_{\dot{\gamma}_t} \tau_{0,t}^h = \frac{1}{2}\tau_{0,t}^{-1}\left[ \dot{\gamma}_t, \tau_{0,t} (\tau_{0,t}^{-1} \tau^h_{0,t}) \right ]^\V =  L_t(Q_t).
    \end{equation}
    There exists a unique solution given by the associated exponential map on $GL(T_uN)$, namely, $Q_t = \mathcal{E}\left(\int_0^\cdot L_s(\cdot)ds \right)_t$ for all $t \geq 0$. Thus, the solution to \eqref{Equation transport parallele horizontal dim finie} is unique and is given by $(\tau_{0,t}^h)_{t \geq 0} = (\tau_{0,t} Q_t)_{t \geq 0}$.
    \end{proof}

        \begin{rem} \label{Remarque caractérosation du transport}
        The first part of this theorem is trivial, it says that the projection of the horizontal lift of the parallel transport is the parallel transport. The important point in this result is the fact that $(\tau_{0,t}^h)_{t \geq 0}$ is entirely characterized by equation \eqref{Equation transport parallele horizontal dim finie}. Thus, in order to obtain the parallel transport on $M$, one can solve \eqref{Equation transport parallele horizontal dim finie} on $N$ and project onto $M$ in order to obtain the solution.
    \end{rem}

    Let us apply this reasoning to the stochastic setting.

    \begin{propdef}[\cite{Ito75}] \label{Prop def transport parallele stochastique M}
        Let $M$ be a Riemannian manifold with $\nabla$ the associated Levi-Civita connection, $A_0,A_1, \dots, A_N \in \Gamma(TM)$. Let $(X_t)_{t \geq 0}$ be the unique solution to:
        
            $$
    \circ \d X_t = \SN A_i(X_t) \circ \d W^i_t + A_0(X_t) \d t, \quad X_0 \in M.
    $$
    There exists a unique solution to the following stochastic covariant equation
    
    \begin{equation*}
        \nabla_{\circ \d X_t}A_t = 0, \quad A_0 \in T_{X_0}M,
    \end{equation*}
    in the following sense:
    
    $$
    \circ \d \langle A_t, B(X_t) \rangle_{X_t} = \SN \langle A_t, \nabla_{A_i(X_t)} B(X_t) \rangle_{X_t} \circ \d W^i_t + \langle A_t, \nabla_{A_0(X_t)}B(X_t) \rangle_{X_t} \d t, \quad \forall B \in \Gamma(TM),
    $$
    consistently with \eqref{Notation intro 2}. This unique solution is called the \textit{stochastic parallel transport} along $(X_t)_{t \geq 0}$ with initial condition $A_0 \in T_{X_0}M$.
    \end{propdef}

    \begin{propdef}
        A $TM$-valued stochastic process $(U_t)_{t \geq 0}$ along a semimartingale $(X_t)_{t \geq 0}$ is said to be a semimartingale if $(\tau_{0,t}^{-1}U_t)_{t \geq 0}$ is a semimartingale on $T_{U_0}M$, where $(\tau_{0,t})_{t \geq 0}$ denotes the stochastic parallel transport along $(X_t)_{t \geq 0}$. The Stratonovich covariant derivative of $U_t$ is defined as the following formal element of $T_{U_t}M$:
        
        \begin{equation*}
            \circ D_t U_t = \tau_{0,t} \circ \d \tau_{0,t}^{-1} U_t.
        \end{equation*}
        Equivalently, $\circ D_t U_t$ can be defined as the unique formal element of $\T_{X_t} M$ such that:
        
        \begin{equation*}
            \circ \d \langle U_t, \tau_{0,t}(B) \rangle_{X_t} = \langle \circ D_t U_t, \tau_{0,t}(B) \rangle_{X_t},
        \end{equation*}
        for all $B \in T_{X_0}M$.
    \end{propdef}
    
    \begin{prop}[Theorem 2.4, \cite{KunitaHiroshi1981}] \label{Bonne forme dérivée covariante}
        Let $B \in \Gamma(TM)$ and $(X_t)_{t \geq 0}$ be a semimartingale solution to 
        
    $$
    \circ \d X_t = \SN A_i(X_t) \circ \d W^i_t + A_0(X_t) \d t, \quad X_0 \in M.
    $$
    Then, 
    
    \begin{equation*}
        \circ \d \tau_{0,t}^{-1} B(X_t) = \SN \tau_{0,t}^{-1}(\nabla_{A_i(X_t)} B(X_t)) \circ \d W^i_t + \tau_{0,t}^{-1}(\nabla_{A_0(X_t)} B(X_t)) \d t.
    \end{equation*}
    \end{prop}

    From now on, we will write either $\circ D_t B(X_t)$ or $ \nabla_{\circ \d X_t} B(X_t)$ to denote the Stratonovich covariant derivative of $ B$ along $(X_t)_{t \geq 0}$, consistently with \eqref{Notation intro 2}.
    
    Let $\bar{A}_0, \bar{A}_1, \dots, \bar{A}_N \in \Gamma(TM)$ and $A_0,A_1, \dots, A_N \in \Gamma(TN)$ their horizontal lifts. Consider a diffusion $(\bar{X}_t)_{t \geq 0}$ on $M$ starting from $x \in M$ solution to
    
    \begin{equation} \label{E 152904}
        \circ \d \bar{X}_t = \SN \bar{A}_i(\bar{X}_t) \circ \d W^i_t + \bar{A}_0(\bar{X}_t) \d t, \quad \bar{X}_0 = x.
    \end{equation}
    Let us define the \textit{horizontal lift} of $(\bar{X}_t)_{t \geq 0}$ starting from $u \in p^{-1}(x)$ similarly as in the deterministic case (see Definition \ref{Lift chemin}) as the process $(X_t)_{t \geq 0}$ given by:
    
    \begin{equation*}
    \circ \d X_t = \mathfrak{h}_{X_t}(\circ \d_t \bar{X}_t) = \SN A_i(X_t) \circ \d W^i_t + A_0(X_t) \d t, \quad X_0 = u
    \end{equation*} 
    Let $(\bar{\tau}_{0,t})_{t \geq 0}$ and $(\tau_{0,t})_{t \geq 0}$ be the stochastic parallel transports respectively on $M$ and $N$ along $(\bar{X}_t)_{t \geq 0},(X_t)_{t \geq 0}$ (once again with the point of view of Remark \ref{Remarque transport flot} but with stochastic flows) and $(\tau^h_{0,t})_{t \geq 0}$ the horizontal lift of $(\bar{\tau}_{0,t})_{t \geq 0}$. As in the deterministic case, the horizontal lift of the stochastic parallel transport is the stochastic parallel transport of the \textit{right-exponential of the (inverse stochastic parallel transport of the) O'Neill tensor}.

    \begin{thm} \label{Transport parallel en dim finie}
    The stochastic parallel transport on $M$ writes
    
    $$
    \bar{\tau}_{0,t}(\bar{U}) = T_{X_t}p(\tau_{0,t}^h(\mathfrak{h}_u(\bar{U})), \quad \forall t \geq 0,
    $$
    with $(\tau_{0,t}^h)_{t \geq 0}$ the so-called horizontal stochastic transport, namely the horizontal lift of $(\bar{\tau}_{0,t})_{t \geq 0}$, 
    
    $$
    \tau_{0,t}^h(U) = \mathfrak{h}_{X_t(u)}(\bar{\tau}_{0,t}(T_up(U))), \quad \forall U \in \H_u N,
    $$
    which is also the unique solution to 
    
        \begin{equation} \label{Horizontal parallel transport finite dimension}
        \circ D_t \tau^h_{0,t}(U) = - \mathcal{A}_{\circ \d X_t(u)}(\tau_{0,t}^h(U)), \quad \tau_{0,0}^h(U) = U, \, \forall U \in \H_u N.
    \end{equation}
    Moreover, $\tau_{0,t}^h = \tau_{0,t}(Q_t)$ where
    
    \begin{equation*}
        Q_t(U) = \mathcal{E} \left(- \SN \int^\cdot_0 \tau_{0,s}^{-1}\left( \mathcal{A}_{A_i(X_s(u))}(\tau_{0,t}(U)) \right) \circ \d W^i_s + \int_0^\cdot \tau_{0,s}^{-1}\left(\mathcal{A}_{A_0(X_s(u))}(\tau_{0,t}(U))  \right) \d s \right)_t,
    \end{equation*}
    with $\mathcal{E}$ the so-called stochastic right-exponential map on $GL(T_uN)$ (see \cite{Hakim1986}). 
    \end{thm}
    
    \begin{proof}
       Let $(\tau^h_{0,t})_{t \geq 0}$ be the horizontal lift of $(\bar{\tau}_{0,t})_{t \geq 0}$ starting from $u$. By Proposition \ref{Connexion Levi-Civita Horizontale dim finie}, we have
       
       $$
       \nabla_{\circ \d X_t} \tau^h_{0,t}(U) = \frac{1}{2} \left[ \circ \d X_t, \tau^h_{0,t}(U) \right ]^\V, \quad \forall U \in \H_u N.
       $$
       Let $L^i_t$ denote the linear map $\tau_{0,t}^{-1}\left(\frac{1}{2}[A_i( X_t),\tau_{0,t}(\cdot)]^\V \right)$ for all $t \geq 0$ and $i=0,1,\dots,N$. Proceeding as in the proof of Proposition \ref{T parallel transport dim finie}, we study the following equivalent $GL(T_uN)$-valued stochastic differential equation
    
    $$
    \circ \d Q_t = -\SN \mathcal{A}_{A_i(X_t(u))}( \tau_{0,t}Q_t(\cdot)) \circ \d W^i_t - \mathcal{A}_{A_0(X_t(u))}(\tau_{0,t}Q_t(\cdot)) \d t.
    $$
    There exists a unique solution, this is the so-called stochastic exponential map on $GL(T_uN)$, 
    $$
    Q_t = \mathcal{E} \left( -\SN \mathcal{A}_{A_i(X_s(u))}( \tau_{0,s}(\cdot)) \circ \d W^i_s - \int_0^\cdot \mathcal{A}_{A_0(X_s(u))}(\tau_{0,s}(\cdot)) \d s \right)_s.
    $$ 
    
    Finally $(\tau_{0,t}^h)_{t \geq 0}$ is the unique solution to \eqref{Horizontal parallel transport finite dimension} and is given by $\tau_{0,t}^h = \tau_{0,t} Q_t$ for all $t \geq 0$ and we get the stochastic parallel transport on $M$ by taking the projection.
    \end{proof} 
    As already mentioned in Remark \ref{Remarque caractérosation du transport}, the important point of this proposition lies in the fact that $(\tau_{0,t}^h)_{t \geq 0}$ is characterized by Equation \eqref{Horizontal parallel transport finite dimension}. We will use this reasoning in Section \ref{Section 5} to prove the existence and uniqueness of the stochastic parallel transport on $\P_\infty$. 
    
    \begin{rem}[The stochastic horizontal transport is horizontal]
        Let $U_t = \tau_{0,t}^h(U)$ denote the solution to \eqref{Horizontal parallel transport finite dimension} with initial condition $U \in \H_u N$. By construction, the horizontal transport preserves horizontality. Nevertheless, this property can in fact be deduced directly from \eqref{Horizontal parallel transport finite dimension}. Indeed, by Proposition \ref{Chain rule dim finie}, one can show that $(U_t^\V)_{t \geq 0}$ is the unique solution to:
        
        \begin{equation} \label{E 12002}
            \circ D_t U_t^\V = \mathcal{A}_{\circ \d X_t(u)}^T(U_t^\V) - \mathcal{A}_{\circ \d X_t(u)}(U_t^\V), \quad U_0^\V = 0.
        \end{equation}
        Thus, the horizontality of $(U_t)_{t \geq 0}$ follows from the uniqueness of the solution and the fact that $0$ is solution to this covariant SDE. 
        
    \end{rem}
    
    \begin{rem}
        In order to highlight the arguments of the proofs in our infinite-dimensional setting, we have stated the results within the (more restrictive) framework of semimartingales solutions to SDEs of the form \eqref{E 152904}. All of these results remains true for general semimartingales.
    \end{rem}

    \subsection{Equivariant diffusions on principal bundles} \label{subsection 2.3}
	In this subsection we leave the context of Riemannian submersions to study principal bundles. Note that these notions are not decorrelated, see Remark \ref{Lien entre fibré principaux et submerisons R} and Subsection \ref{Subsection 2.4} below. For instance, we will see in Section \ref{Section 3} that the group of diffeomorphisms $\mathscr{D}$ can be seen as a Riemannian submersion over $\P_\infty$ but also as a principal bundle over the latter. 
	    
    Let $G$ be a Lie group with identity element $e$. In this subsection we first introduce some definitions and basic results about $G$-principal bundles, and then we show that any equivariant diffusion on a principal bundle admits a decomposition as a product of a horizontal equivariant diffusion $(h_t)_{t \geq 0}$ and a vertical process $(g_t)_{t \geq 0}$ given by a right exponential of a stochastic process in the Lie algebra $\mathfrak{g}$ of $G$.
    
    Let $p : P \rightarrow M$ be a $G$-principal bundle over a smooth manifold $M$. Let $g \in G$ and let us denote $R_g : P \rightarrow P$ the \textit{right action} by $g$ on $P$ and also the \textit{right multiplication} by $g$ in $G$. The notation $L_g$ will be used for the \textit{left multiplication} by $g$ in $G$. As before, we define $\V_uP = \ker{T_up}$. A vector field $A$ on $P$ is said to be \textit{$G$-right invariant} if for all $u \in P$ and for all $g \in G$ we have
    
    $$
    A(u \cdot g) = TR_g( A(u)).
    $$
    Let $u \in G$, we define the map $\iota_u : G \rightarrow P$ to be $\iota_u(g) = u \cdot g$. A vector field $B$ on $P$ is said to be a \textit{vertical left invariant vector field} if there exists $\tilde{B} \in \mathfrak{g}$ such that for all $u \in P$
    
    $$
    B(u) = T\iota_u \tilde{B}.
    $$
    One can easily prove that a vertical left invariant vector field is vertical. The reader is referred to \cite{Sontz2015} for a detailed presentation of principal bundles. 
    
    \begin{defi}
        Let $g \in G$, the \textit{adjoint representation} $\mathrm{Ad}(g) : \mathfrak{g} \rightarrow \mathfrak{g}$ is the differential of the inner automorphism $ h \in G \mapsto g \cdot h \cdot g^{-1}$ in $h=e$.
    \end{defi}
    
    \begin{defi} \label{D Ehresmann connection}
        An \textit{Ehresmann connection} on $P$ is a subbundle $\H P$ of $TP$ such that
    
    \begin{enumerate}
    \item For all $u \in P$, $\H_uP \oplus \V_uP = T_uP$.
    \item For all $u \in P$, $T_up$ is an isomorphism between $\H_uP$ and $T_p(u)M$, the inverse is the horizontal lift $\h_u : T_{p(u)}M \rightarrow \H_uP$.
    \item For all $g \in G$ and $u \in P$, $\H P$ is $G$-invariant i.e
    
    $$
    TR_g(\H_uP) = \H_{u \cdot g}P.
    $$
    \end{enumerate}
    \end{defi}

    \begin{rem} \label{R invariance à droite relevé horizontal}
    Note that points $(2)$ and $(3)$ in Definition \ref{D Ehresmann connection} gives that
    
    $$
    \mathfrak{h}_{u \cdot g}(v) = TR_g( \mathfrak{h}_u(v)), \quad \forall v \in T_{p(u)}M.
    $$
    Moreover, one can check that $TR_g(\V_uP) = \V_{u \cdot g} P$. In particular, if a vector $A \in T_uP$ admits the following unique decomposition $A = A^\H + A^\V \in \H_u P \oplus \V_u P$, then, $TR_g(A)$ admits the following one
    
    $$
    TR_g(A) = TR_g(A^\H) + TR_g(A^\V) \in \H_{u \cdot g}P \oplus \V_{u \cdot g} P.
    $$
    \end{rem}
    
    \begin{rem} \label{Lien entre fibré principaux et submerisons R}
        Suppose that $P$ is endowed with a Riemannian metric which is right invariant for the action of $G$. Then, $P$ admits an Ehresmann connection given by the horizontal space $\H P = \V P^\perp$. In this case, $p : P \rightarrow M$ defines a Riemannian submersion. Conversely, let $p : N \rightarrow M$ be a Riemannian submersion, if there exists a Lie group $G$ acting freely and transitively on the fibers, $N$ defines a principal bundle over $M$. 
    \end{rem}
    Suppose that $P$ is provided with an Ehresmann connection $\H P$, and let $A \in \Gamma(TP)$. Since for all $u \in P$, $T_uP = \H_uP \oplus \V_uP$, we have a unique decomposition $A = A^\H \oplus A^\V \in HP \oplus VP$. We say that $A^\H$ (resp. $A^\V$) is the horizontal (resp. vertical) part of $A$.

    In this context, we have a one-to-one correspondence between vertical left invariant vector fields and elements of the lie algebra $\mathfrak{g}$ being given by
    
    $$
    A \in \mathfrak{g} \longmapsto T \iota_u (A) \in \V_uP.
    $$
    The inverse is given by the so called connection form $\varpi$ which is characterized by the following conditions.
    
    \begin{enumerate}
    \item For all $g \in G$, $u \in P$, $\varpi(T \iota_u (U)) = U$ for all $U \in \mathfrak{g}$.
    \item The connection form $\varpi$ vanishes on $\H P$.
    \end{enumerate}
    
    \begin{prop} \label{P varpi}
        Let $A \in \Gamma(TP)$. Then, for all $g \in G$, $\varpi(TR_g A) = \Ad (g^{-1}) \varpi(A)$.
    \end{prop}

    \begin{rem}
    	In the case where $P$ is a Lie Group and $G$ is a sub-Lie group of $P$ one can check that $\iota_u = TL_u$ and the connection form is just given by $\varpi(u) = TL_{u^{-1}} \circ P_{\V_u}$.
    \end{rem}
    
    \begin{defi} \label{Définition équivariance}
        A map $U : P \rightarrow P$ is said to be \textit{equivariant} if
        
        $$
        U(u \cdot g) = U(u) \cdot g, \quad \forall u \in P, \forall g \in G.
        $$
        A flow of maps $(U_t)_{t \geq 0}$ such that $U_t : P \rightarrow P$ for all $t\geq 0$ is said to be equivariant if $U_t$ is equivariant for all $t \geq 0$.
    \end{defi}
    
    \begin{defi}
        A smooth path $(\gamma_t)_{t \geq 0}$ on $P$ is said to be \textit{horizontal} (resp. \textit{vertical}) if for all $t \geq 0$, $\dot{\gamma}_t \in \H_{\gamma_t}P$ (resp. $ \dot{\gamma}_t \in \V_{\gamma_t} P$).
    \end{defi}
    
    We define the \textit{horizontal lift} of a path $(\bar{\gamma}_t)_{t \geq 0}$ on $M$ as in Definition \ref{Lift chemin} with $\mathfrak{h}$ the horizontal lift associated to the considered Ehresmann connection. 
    
    \begin{exemple}[The horizontal lift is equivariant] \label{Ex Horizontal lift equivariant}
    Let $\bar{A} \in \Gamma(TM)$ and $(\bar{\gamma}_t)_{t \geq 0}$ be solution to
    
    $$
    \dot{\bar{\gamma}}_t(x) = \bar{A}(\bar{\gamma}_t(x)), \quad \bar{\gamma}_0(x) \in M.
    $$
    Then, for all $u \in P_{\bar{\gamma}_0(x)}$, the horizontal lift $(\gamma_t)_{t \geq 0}$ is, by definition, solution to
    
    $$
    \dot{\gamma}_t(u) = A(\gamma_t(u)), \quad \gamma_0(u)= u,
    $$
    where $A = \mathfrak{h}(\bar{A})$. In particular, this is the simplest example of equivariant diffusion that one could think of. To see that $(\gamma_t)_{t \geq 0}$ is equivariant, let $g \in G$. We have
    
    $$
    \frac{\d}{\d t} \gamma_t(u) \cdot g =\frac{\d}{\d t} R_g(\gamma_t(u)) = TR_g( \dot{\gamma}_t(u)) = TR_g( A(\gamma_t(u))) = A(\gamma_t(u) \cdot g),
    $$
    where the last equality is due to the fact that $A$ is $G$-right invariant (see Remark \ref{R invariance à droite relevé horizontal}). In particular, $(\gamma_t(u) \cdot g)_{t \geq 0}$ is solution to
    
    $$
    \dot{\gamma}_t(u \cdot g) = A (\gamma_t(u) \cdot g), \quad \gamma_0(u \cdot g) = u \cdot g,
    $$
    and, by uniqueness of the solution, $\gamma_t(u \cdot g) = \gamma_t(u) \cdot g$.
    \end{exemple}
    
    \begin{rem} \label{Remarque identification fibre et groupe}
        Notice that a vertical path $(\gamma_t)_{t \geq 0}$ starting in $\gamma_0 =u$ stays in the fiber $p^{-1}(p(u))$. In particular, since $G$ acts freely and transitively on the fibers, $(\gamma_t)_{t \geq 0}$ can always be written as $\gamma_t = \gamma_0 \cdot g_t$, and then it may be identified as the path $(g_t)_{t \geq 0}$ in $G$. In the remainder of this subsection we will identify $P_{\gamma_0}$ and $G$ and consequently vertical paths starting at $\gamma_0$ with paths in $G$.
    \end{rem}
    
    Let $(\gamma_t)_{t \geq 0}$ be a smooth path on $P$, $(\bar{\gamma}_t)_{t \geq 0} = (p(\gamma_t))_{t \geq 0}$ the projected path and $(h_t)_{t \geq 0}$ the horizontal lift of $\bar{\gamma}$ defined as in Definition \ref{Lift chemin} with $\mathfrak{h}$ the horizontal lift associated to the Ehresmann connection. Since $G$ acts freely and transitively on the fibers of $P$, it is clear that $(\gamma_t)_{t \geq 0}$ decomposes itself in a unique way as a product
    
    \begin{equation} \label{decomposition}
        \gamma_t = h_t \cdot g_t,
    \end{equation}
    where $(g_t)_{t \geq 0}$ is a path on $G$ with $g_0 = e$. In the case of equivariant flow of maps, this decomposition has more properties. 
    
    \begin{prop} \label{P D1}
    	Let $A$ be a smooth right invariant vector field on $P$ with $A = A^\H + A^\V \in \H P \oplus \V P$. Let $(U_t)_{t \geq 0}$ be the solution to
    	
    	    \begin{equation}\label{EDO principal bundle}
        \dot{U}_t(u) = A(U_t(u)), \quad U_0(u) \in P,
    \end{equation}
    with $U_0$ an equivariant map and let $(X_t)_{t \geq 0} = (p(U_t))_{t \geq 0}$. Then, $(X_t)_{t \geq 0}$ is solution to the following autonomous equation
    
    \begin{equation} \label{Equation X_t}
        \dot{X}_t(x) = \bar{A}^\H(X_t(x)), \quad X_0(x) = p(U_0(u)), \,  u \in p^{-1}(x),
    \end{equation}
    where $\bar{A}^\H = Tp(A^\H)$. Moreover, the flow $(U_t)_{t \geq 0}$ is equivariant and the decomposition \eqref{decomposition} takes the form $U_t = h_t \cdot g_t$ where $(h_t)_{t \geq 0}$ is the horizontal lift of $(X_t)_{t \geq 0}$ and $(g_t)_{t \geq 0}$ is a $G$-valued path such that
    		
    	\begin{equation*} 
    		\dot{h}_t(u) = \h_{h_t(u)}(\dot{X}_t(p(u))) = A^H (h_t(u)), \quad h_0(u) = U_0(u),
    	\end{equation*}
    	and
    	
    \begin{equation} \label{E g_t dim finie}
    	\dot{g}_t(u) = TR_{g_t(u)} \varpi\left (A^\V(h_t(u)) \right) , \quad g_0(u) =e.
    \end{equation}
    	In particular, $(h_t)_{t \geq 0}$ is a horizontal equivariant autonomous flow  and $g_t(u \cdot g) = g^{-1} \cdot g_t(u) \cdot g$ and is a right exponential of a path in the Lie algebra $\mathfrak{g}$ of $G$, namely:
    	
    	\begin{equation}\label{E exponentielle droite dim finie}
    		g_t(u) = \mathcal{E}^R\left( \int_0^\cdot \varpi \left( A^\V(h_s(u)) \right) ds  \right)_t.
    	\end{equation}
    \end{prop}
    
    \begin{proof}
    The fact that $(X_t)_{t \geq 0}$ is solution to \eqref{Equation X_t} is due to the fact that
    
	$$
	\frac{\d}{\d t}p(U_t(u)) = T_{U_t(u)}p \left(A^\H(U_t(u)) + A^\V(U_t(u)) \right) = \bar{A}^\H(p(U_t(u))),
	$$    
    where the first equality is an application of the chain rule and the second is due to the fact that $\ker Tp = \V P$. The fact that Equation \eqref{Equation X_t} is autonomous is due to the fact that $A^\H$ is $G$-right invariant since $A$ is $G$-right invariant. 
    
     Let $u \in P$, $g \in G$. We have:

$$
\frac{\d}{\d t} (U_t(u) \cdot g) =TR_g(\dot{U}_t(u)) = TR_g(A(U_t(u))) = A(U_t(u) \cdot g).
$$
In particular, $(U_t(u) \cdot g)_{t \geq 0}$ is solution to

$$
	\dot{U}_t(u \cdot g) = A(U_t(u \cdot g)), \quad U_0(u \cdot g) \in P,	
$$
by uniqueness of the solution to \eqref{EDO principal bundle}, we obtain that $U_t(u) \cdot g = U_t(u \cdot g)$ which is the equivariance property.
     
Let $(h_t)_{t \geq 0}$ be the unique solution to
    
    	$$
    	\dot{h}_t(u) = \h_{h_t(u)}(\dot{X}_t(p(u))), \quad h_0(u) = U_0(u).
    	$$
    	As noticed before, the projected flow $(X_t)_{t \geq 0}$ is autonomous, consequently, $(h_t)_{t \geq 0}$ is a horizontal equivariant autonomous flow as a horizontal lift of an autonomous flow (see Example \ref{Ex Horizontal lift equivariant}). Set $g_t$ such that $U_t = h_t \cdot g_t$ (it exists and is unique since $G$ acts freely and transitively on the fibers of $P$). Since $(U_t)_{t \geq 0}$ is equivariant, we have $U_t(u \cdot g) ) = U_t(u) \cdot g$ and, by construction, it is clear that $(h_t)_{ \geq 0}$ is equivariant. In particular:
    
    $$
    U_t(u \cdot g) = h_t(u \cdot g ) \cdot g_t( u \cdot g) = h_t(u) \cdot g \cdot g_t(u \cdot g) \qquad \text{and} \qquad U_t( u \cdot g) = U_t(u) \cdot g = h_t(u) \cdot g_t(u) \cdot g.
    $$
    Then, since $G$ acts freely on the fibers of $P$ we get $g_t(u \cdot g) = g^{-1} \cdot g_t(u) \cdot g$. To get the equation of $(g_t)_{t \geq 0}$, let $A = A^\H \oplus A^\V \in \H P \oplus \V P$ be the direct sum decomposition. By uniqueness of this decomposition we have
    
    \begin{equation*}
    \dot{U}_t = A^\H(U_t) + A^\V(U_t) = TR_{g_t}(\dot{h}_t) + T\iota_{h_t}(\dot{g}_t).
    \end{equation*}
    By identifying the horizontal and vertical parts, it is clear that $TR_{g_t}(\dot{h}_t) = A^\H(U_t)$ and $T\iota_{h_t}(\dot{g}_t) = A^\V(U_t)$. Moreover, we have
    
    $$
    A^\V(U_t) = T\iota_{h_t \cdot g_t} \left( \varpi \left( A^\V (h_t \cdot g_t)\right) \right) = T\iota_{h_t} TL_{g_t} Ad(g_t^{-1}) \left( \varpi \left( A^\V (h_t)\right) \right) = T\iota_{h_t} \left ( TR_{g_t} \varpi\left( A^\V (h_t)\right)\right).
    $$
    By identification we obtain
    
    $$
    	\dot{g}_t = TR_{g_t} \varpi \left( A^\V (h_t)\right).
    $$
    Formula \eqref{E exponentielle droite dim finie} is a direct consequence of Equation \eqref{E g_t dim finie}.
    
    \end{proof}
    
 Let us introduce the notion of equivariant diffusions, the stochastic analogous of equivariant flows defined in Definition \ref{Définition équivariance}.

\begin{defi}
	A solution to a Stratonovich's SDE of the following form:
	
	$$
	\circ \d U_t = \SN A_i(U_t) \circ \d W^i_t + A_0(U_t)dt, \quad U_0(u) \in P,
	$$
	is said to be an \textit{equivariant diffusion} if for all $g \in G$, $U_t(u \cdot g) = U_t(u) \cdot g$ for all $t \geq 0$.
\end{defi}

 The following theorem is the stochastic analogous of Proposition \ref{P D1}. This result is proved in \cite{ELJW}, \cite{ElLeLi04} by considering the associated infinitesimal generators. 
    
    \begin{prop} \label{P D2} 
    		Let $A_0,A_1, \dots, A_N$ be smooth right invariant vector fields on $P$ with $A_i = A_i^\H + A_i^\V \in \H P \oplus \V P$ for all $i=0,1,\dots, N$. Let $(U_t)_{t \geq 0}$ be the diffusion solution to 
    		
    		\begin{equation}\label{Equation equivariante sur N}
        \circ \d U_t = \SN A_i(U_t) \circ \d W^i_t + A_0(U_t)dt, \quad U_0(u) \in P,
    \end{equation}
    with $U_0$ an equivariant map and let $(X_t)_{t \geq 0} = (p(U_t))_{t \geq 0}$. Then, $(X_t)_{t \geq 0}$ is a diffusion solution to the following Stratonovich's SDE
    
    \begin{equation*} 
        \circ \d X_t(x) = \SN \bar{A}^\H_i(X_t(x)) \circ \d W^i_t + \bar{A}^\H_0(X_t(x)) \d t, \quad X_0(x) = p(U_0(u)), \, u \in p^{-1}(x),
    \end{equation*}    
    where $\bar{A}_i^\H =Tp(A_i^\H)$ for all $i=0,1,\dots,N$. Moreover, the stochastic flow $(U_t)_{t \geq 0}$ is equivariant and the decomposition \eqref{decomposition} takes the form $U_t = h_t \cdot g_t$ where $(h_t)_{t \geq 0}$ is the horizontal lift of $(X_t)_{t \geq 0}$ and $(g_t)_{t \geq 0}$ is a $G$-valued path such that for all $u \in P$ :
	
	\begin{equation*}
	\circ \d h_t(u) = \h_{h_t(u)}(\circ \d X_t(p(u))) = \SN A_i^\H(h_t(u)) \circ \d W^i_t + A_0^\H(h_t(u)) \d t, \quad h_0(u) = U_0(u)
	\end{equation*}
	and 
	
	\begin{equation*}
	\circ \d g_t(u) = \SN TR_{g_t(u)} \varpi\left (A^\V_i(h_t(u)) \right) \circ \d W^i_t + TR_{g_t(u)} \varpi\left (A^\V_0(h_t(u)) \right) \d t, \quad g_0(u) = e.
	\end{equation*}
	In particular, $(h_t)_{t \geq 0}$ is a horizontal equivariant diffusion and $g_t(u \cdot g) = g^{-1} \cdot g_t(u) \cdot g$ for all $g \in G$ and is a right stochastic exponential of a stochastic process in the Lie algebra $\mathfrak{g}$ of $G$, namely:
	
	\begin{equation*}
	    g_t(u) = \mathcal{E}^R \left( \int_0^\cdot \SN \varpi\left (A^\V_i(h_s(u)) \right) \circ \d W^i_s + \varpi\left (A^\V_0(h_s(u)) \right) \d s \right)_t.
	\end{equation*}
    \end{prop} 
    
    \begin{proof}
    The proof can be performed exactly in the same way as the proof of Proposition \ref{P D1} by replacing time derivative by Stratonovich's derivative. See also \cite{ELJW}, \cite{ElLeLi04} for another proof.
    \end{proof}
    
    \begin{rem} \label{R Horizontal lift equivariant}
	As seen in Example \ref{Ex Horizontal lift equivariant}, the horizontal lift of a diffusion on $M$ is an equivariant diffusion. In particular, the vertical part in the decomposition \eqref{decomposition} is given by $g_t = e$.
\end{rem}

\subsection{Principal bundle with a Riemannian submersion structure} \label{Subsection 2.4}
As mentioned in Remark \ref{Lien entre fibré principaux et submerisons R}, a $G$-principal bundle $p : N \rightarrow M$ endowed with a Riemannian metric which is right invariant for the action of $G$ on $N$ has a canonical Ehresmann connection given by the horizontal space $\H N = \V N^\perp$. This is not uncommon for these two structures to coexist, the (infinite-dimensional) example of the group of diffeomorphisms will be studied in Section \ref{Section 3}. We still denote by $\nabla$ the Levi-Civita connection on $N$ and by $\bar{\nabla}$ the one on $M$. In this case, the results presented in Subsections \ref{Subsection 2.2} and \ref{subsection 2.3} obviously apply, but we also obtain specific results to this context.

This subsection is divided into two paragraphs: We show in the first one that the horizontal stochastic transport is right invariant (see Proposition \ref{P Equivariance transport paralelle horizontal dim finie}). We will see that this property may be interpreted as an equivariance property on a well chosen principal bundle, see Remark \ref{R Equivariance of the horizontal stochastic parallel transport dim finie}. In Proposition \ref{P Processus equivariant TD}, we will show that this lift extends to an equivariant one on the $G \rtimes \mathfrak{g}$-principal bundle $TN$. The second paragraph is the continuation of Subsection \ref{subsection 2.3}. More precisely, we study decomposition \eqref{decomposition} of Itô's equivariant diffusions driven by respectively horizontal (Proposition \ref{P D 4}) and vertical (Proposition \ref{P D3}) vector fields.  

\vspace{10pt}

\noindent \textbf{Right invariance of the horizontal stochastic transport.}
    
    \begin{lem} \label{Invariance à droite dérivée covariante dim finie}
        Let $A, B \in \Gamma(TN)$, then
        
        \begin{equation*}
            TR_g(\nabla_B A) = \nabla_{TR_g(B)}TR_g(A), \quad \forall g \in G.
        \end{equation*}
		In particular, if $A, B \in \Gamma(TN)$ are $G$-right invariant vector fields, $\nabla_{A}B$ is a $G$-right invariant vector field.
	\end{lem}
	
	\begin{proof}
	One can check this result by using Koszul's formula and the fact that $R_g$ is an isometry from $N$ to $N$. 
	\end{proof}

The next proposition is about the right invariance property of the horizontal stochastic transport. Note that this proposition is trivial since, by Theorem \ref{Transport parallel en dim finie}, we know that $\tau_{0,t}^h = \mathfrak{h}(\tau_{0,t})$. Here, our goal is to recover this right invariance property directly from \eqref{E 229}. The same arguments will be used in the proof of Proposition \ref{P Equivariance transport horizontal D}.

\begin{prop} \label{P Equivariance transport paralelle horizontal dim finie}
	Let $\bar{A}_0, \bar{A}_1, \dots, \bar{A}_N \in \Gamma(TM)$, and $(\bar{X}_t)_{t \geq 0}$ be a stochastic process on $M$ solution to
    
    \begin{equation*}
    \circ \d \bar{X}_t = \SN \bar{A}_i(\bar{X}_t) \circ \d W^i_t + \bar{A}_0(\bar{X}_t) \d t, \quad \bar{X}_0 = x
    \end{equation*}
    and consider its horizontal lift:
    
    \begin{equation*}
    \circ \d X_t(u) = \mathfrak{h}_{X_t(u)}(\circ \d_t \bar{X}_t(u)) = \SN A_i(X_t(u)) \circ \d W^i_t + A_0(X_t(u)) \d t, \quad X_0(u) = u,
    \end{equation*}
    with $A_i = \mathfrak{h}(\bar{A}_i)$ for all $i=0,1, \dots, N$. Then, for $U \in \H_uN$, the process $(\tau_{0,t}^h(u, U))_{t \geq 0}$ solution to
    
    \begin{equation}\label{E 229}
    \circ D_t \tau^h_{0,t}(u,U) = \frac{1}{2} \left[ \circ \d X_t(u), \tau^h_{0,t}(u,U) \right ]^\V, \quad \tau_{0,0}^h(u,U) = U
\end{equation}
	is $G$-right invariant on $N$, namely,
	
	\begin{equation*} 
	TR_g (\tau_{0,t}^h(u,U)) = \tau_{0,t}^h(u \cdot g, TR_g(U)), \quad \forall g \in G.
	\end{equation*}
    
\end{prop}

\begin{proof}
Let $u \in N$, $U \in T_u N$ and $U_t = \tau_{0,t}^h(u,U)$. By applying Lemma \ref{Invariance à droite dérivée covariante dim finie} and Corollary \ref{Chain rule dim finie}, for any semimartingale $(A_t)_{t \geq 0}$ along the equivariant diffusion $(X_t(u))_{t \geq 0}$ (see Remark \ref{R Horizontal lift equivariant}) we have:

\begin{equation} \label{E 231}
\circ D_t TR_g(A_t) = TR_g (\circ D_t A_t) = \frac{1}{2} TR_g \left(  \left[\circ \d X_t(u \cdot g), TR_g(A_t) \right]^\V \right).
\end{equation}
By Lemma \ref{Invariance à droite dérivée covariante dim finie}, the fact that

$$
\frac{1}{2} \left[ \circ \d X_t(u), U_t \right ]^\V =- P_{\V_{X_t(u)}} \left( \nabla_{U_t } \circ \d X_t(u)  \right),
$$
and since $TR_g$ commutes with the vertical projection, we obtain:

$$
\frac{1}{2} TR_g \left(  \left[\circ \d X_t(u), A_t \right]^\V \right) = \frac{1}{2} [TR_g(\circ \d X_t(u)), TR_g(A_t)]^\V = \frac{1}{2} \left[ \circ dX_t(u\cdot g), TR_g(A_t) \right]^\V.
$$
By uniqueness of the solution to the horizontal transport equation (see Theorem \ref{Transport parallel en dim finie}), we obtain that

$$
\tau_{0,t}^h(u \cdot g, TR_g(U)) = TR_g (\tau_{0,t}^h(u,U)).
$$
\end{proof}

\begin{rem}[Equivariance of the horizontal stochastic parallel transport] \label{R Equivariance of the horizontal stochastic parallel transport dim finie}
	The map $Tp|_{\H N} : \H N \rightarrow TM $ can be seen as a $G$-principal bundle where the action of $G$ on $\H N$ is given by the differential of the right multiplication $TR_g$ for $g \in G$. Indeed, it can be easily seen that this action is free and transitive on the fibers $T_u p|_{\H N}^{-1}(A)$ where $ A \in T_{p(u)}M$. In particular, the result of Proposition \ref{P Equivariance transport paralelle horizontal dim finie} can be interpreted as follows: the process $(\tau_{0,t}^h)_{t \geq 0}$ is equivariant on the principal bundle $\H N$.
	\end{rem}
	
	Note that there exists a natural group acting on the whole tangent space $TN $, this is the semidirect product group $G \rtimes \mathfrak{g}$. The action of $G \rtimes \mathfrak{g}$ on $T N$ is given by the differential of the action of $G$ on $N$, namely:
	
	\begin{equation} \label{E Action de groupe}
	\left( u, A \right) \cdot \left( g, Y \right) = \left( u \cdot g, TR_g(A) + TR_g T\iota_{u}  (Y) \right),
	\end{equation}
	for $u \in N$, $A \in T_uN$, $g \in G$ and $Y \in \mathfrak{g}$. This action provides the tangent bundle $TN$ with a principal bundle structure above $TM$. One can construct an equivariant process on this principal bundle which coincides with $\tau_{0,t}^h$ on $\H P$, this is the subject of the next proposition.
	
	\begin{prop}[Equivariant lift of the stochastic parallel transport] \label{P Processus equivariant TD}
	    Let $(U_t)_{t \geq 0}$ be the horizontal lift of a $M$-valued diffusion $(X_t)_{t \geq 0}$. For all $u \in P_{X_0}$, we define $\tilde{\tau}_{0,t}(u, \bullet)$ to be the stochastic flow of linear maps $\tilde{\tau}_{0,t}(u, \bullet) : T_{U_0(u)} P \rightarrow T_{U_t(u)} P$ such that
	    
	    \begin{equation*}
	        \tilde{\tau}_{0,t}(u,A^\H) = \tau_{0,t}^h(u,A^\H) \text{ and } \tilde{\tau}_{0,t}(u,A^\V) = T \iota_{U_t(u)} (\varpi(A^\V)), \quad \forall A^\H \in \H_u P, \, \forall A^\V \in \V_u P, \, \tilde{\tau}_{0,0}(u, \cdot) = \mathrm{id}_{T_{U_0(u)}P}.
	    \end{equation*}
	    This process is equivariant on $TP$ seen as a $G \rtimes \mathfrak{g}$ principal bundle, namely,
	    
	    \begin{equation*}
	        \tilde{\tau}_{0,t}(u,A) \cdot (g, Y) = \tilde{\tau}_{0,t}( (u, A) \cdot (g,Y)), \quad \forall A \in T_{U_0(u)}P, \, \forall (g, Y) \in G \rtimes \mathfrak{g}.
	    \end{equation*}
	    Moreover, $(\tilde{\tau}_{0,t})_{t \geq 0}$ is a lift of $(\bar{\tau}_{0,t})_{t \geq 0}$, i.e. for any $\bar{A} \in T_{X_0}M$:
        
        \begin{equation} \label{E236}
             Tp (\tilde{\tau}_{0,t}(u,A))= \bar{\tau}_{0,t}(\bar{A}) , \quad  \forall A \in Tp^{-1}(\bar{A})
        \end{equation}
	    
	\end{prop}
	
	\begin{proof}
	Let $u \in P_{X_0}$, $A \in T_uP$ and $(g,Y) \in G \times \mathfrak{g}$. We have to show that
	
		    \begin{equation*}
	        \tilde{\tau}_{0,t}(u \cdot g, TR_g(A) + TR_g(T\iota_u(Y))) = TR_g(\tilde{\tau}_{0,t}(u, A)) + TR_g T \iota_{U_t(u)}(Y).
	    \end{equation*}
	    By linearity of $\tilde{\tau}_{0,t}(u \cdot g, \bullet)$ we have:
	    
	    \begin{align}
	        \tilde{\tau}_{0,t}(u \cdot g, TR_g(A) + TR_g(T\iota_u(Y))) & = \tilde{\tau}_{0,t}(u \cdot g, TR_g(A)) + \tilde{\tau}_{0,t}(u \cdot g, TR_g ( T\iota_{u} (Y))) \nonumber \\
	        & = \tilde{\tau}_{0,t}(u \cdot g, TR_g(A)) + T \iota_{U_t(u \cdot g)}( \varpi(TR_g(T \iota_u Y))) \nonumber \\
	        & =\tilde{\tau}_{0,t}(u \cdot g, TR_g(A)) + T \iota_{U_t(u \cdot g)}( \mathrm{Ad}(g^{-1})(Y)) \label{E238} \\
	        & =\tilde{\tau}_{0,t}(u \cdot g, TR_g(A)) + T \iota_{U_t(u)}TL_{g}( \mathrm{Ad}(g^{-1})(Y)) \label{E239} \\
	        & =\tilde{\tau}_{0,t}(u \cdot g, TR_g(A)) + TR_g T \iota_{U_t(u)} (Y), \nonumber
	    \end{align}
	    where \eqref{E238} comes from Proposition \ref{P varpi} and \eqref{E239} follows from the equivariance property of $(U_t)_{t \geq 0}$. Let $A^\H$ and $A^\V$ be respectively the horizontal and the vertical part of $A$. The fact that $(\tilde{\tau}_{0,t})_{t \geq 0}$ is a lift of $(\bar{\tau}_{0,t})_{t \geq 0}$ comes from:
	    
	    \begin{align}
	        Tp( \tilde{\tau}_{0,t}(u,A)) & = Tp \left( \tilde{\tau}_{0,t}(u, A^\H) \right) + Tp \left( \tilde{\tau}_{0,t}(u, A^\V) \right) \label{E240} \\
	        & = Tp \left( \tilde{\tau}_{0,t}(u, A^\H) \right) \label{E241} \\
	        &= Tp (\tau_{0,t}^h(u, \mathfrak{h}_{U_0(u)}(\bar{A}))) = \bar{\tau}_{0,t}(\bar{A}), \label{E242}
	    \end{align}
	    where \eqref{E240} is due to the linearity of $\tilde{\tau}_{0,t}$, \eqref{E241} is due to the fact that $\tilde{\tau}_{0,t}(u,A^\V)$ is a vertical vector and \eqref{E242} comes by definition of $(\tau^h_{0,t})_{t \geq 0}$.
	\end{proof}

\vspace{10pt}

\noindent \textbf{Decompositions of Itô's equivariant diffusions.}

    \begin{prop} \label{P D 4}
    Let $(U_t)_{t \geq 0}$ be a diffusion on $N$ such that
	
	\begin{equation*}
	\d^{\nabla}U_t = \SN A_i^{\H}(U_t) \d W^i_t + A_0^{\H}(U_t) \d t, \quad U_0(u) \in N,
	\end{equation*}
	where $A_i^{\H}$ is horizontal and right invariant for all $i=0,1, \dots, N$ with $U_0$ an equivariant map and let $(X_t)_{t \geq 0} = (p(U_t))_{t \geq 0}$. Then, $(X_t)_{t \geq 0}$ is a diffusion solution to the following Itô's SDE
    
    \begin{equation*} 
        \d^{\bar{\nabla}} X_t(x) = \SN \bar{A}^\H_i(X_t(x)) \d W^i_t + \bar{A}^\H_0(X_t(x)) \d t, \quad X_0(x) = p(U_0(u)), \, u \in p^{-1}(x),
    \end{equation*}    
    where $\bar{A}_i^\H =Tp(A_i^\H)$ for all $i=0,1,\dots,N$. Moreover, the stochastic flow $(U_t)_{t \geq 0}$ is an equivariant horizontal diffusion. In particular, decomposition \eqref{decomposition} writes $U_t = h_t$.
    \end{prop}
    
    \begin{proof}
    By applying Proposition \ref{Conversion Ito strato dimension finie} we have
    
    $$
    \circ \d U_t = \SN A_i^{\H}(U_t) \circ \d W^i_t + A_0^{\H}(U_t) \d t - \frac{1}{2} \SN \nabla_{A_i^{\H}}A_i^{\H}(U_t) \d t.
    $$
    Since $A_i^{\H}$ is $G$-right invariant for all $i=0,1, \dots, N$, this is also the case for $\nabla_{A_i^{\H}}A_i^{\H}$ by Lemma \ref{Invariance à droite dérivée covariante dim finie}. Thus, we can apply Proposition \ref{P D2}. By using the fact that $\varpi(\H N) = \{0 \}$ and the fact that $P_{\V}(\nabla_{A_i^{\H}} A_i^{\H}) = \frac{1}{2}[A_i^{\H}, A_i^{\H}] = 0$ (see Proposition \ref{Connexion Levi-Civita Horizontale dim finie}) we obtain that the vertical process $(g_t)_{t \geq 0}$ in the decomposition \eqref{decomposition} is $g_t= e$ for all $t \geq 0$. We obtained the result.
    \end{proof}
    
    As seen in Proposition \ref{P D 4}, the decomposition \ref{decomposition} in the case of Itô's horizontal diffusions is trivial. This is not the case for Itô's vertical diffusions and this is the subject of the next proposition.
    
    We denote by $\nabla^{P_x}(u) = P_{\V_u}(\nabla(u))$ the induced Levi-Civita connection at the point $u$ of the fiber $P_x$. We also define the second fundamental form as:
    
    $$
    \mathrm{I\!I}^{P_x}(A_1,A_2)(u) = P_{\H_uP}\left( \nabla_{A_1}A_2(u) \right), \quad \forall u \in P_x, \, A_1, A_2 \in \Gamma(TP_x).
    $$
    Note that the second fundamental form is actually a tensor, see for example \cite{Lee1997}.
    
    \begin{prop} \label{P D3}
    Let $(U_t)_{t \geq 0}$ be a diffusion on $N$ such that
	
	\begin{equation*}
	\d^{\nabla}U_t = \SN A_i^{\V}(U_t) \d W^i_t + A_0^{\V}(U_t) \d t, \quad U_0(u) \in P,
	\end{equation*}
	where $A_i^{\V}$ is vertical and $G$-right invariant for all $i=0,1, \dots, N$ with $U_0$ an equivariant map. Then, $(U_t)_{t \geq 0}$ is an equivariant diffusion. Moreover, the decomposition \eqref{decomposition} is given by a finite variation process $(h_t)_{t \geq 0}$ solution to:
	
	\begin{equation*}
	\d h_t = -\frac{1}{2} \SN \mathrm{I\!I}^{P_{p(h_t)}}(A_i^{\V}, A_i^{\V})(h_t) \d t, \quad h_0(u) = U_0(u) \in N,
	\end{equation*}
	and a $G$-valued process process $(g_t)_{t \geq 0}$ solution to:
	
	\begin{equation*}
	\circ \d g_t  = \SN TR_{g_t}(\varpi (A_i^{\V}(h_t))) \circ \d W^i_t + TR_{g_t}(\varpi( A_0^{\V}(\Phi_t))) \d t - \frac{1}{2} \SN TR_{g_t}\left( \varpi \left( \nabla^{N_{p(h_t)}}_{A_i}A_i(h_t) \right) \right) \d t, \quad g_0(u) =e.
	\end{equation*}
	
\end{prop}

\begin{proof}
By using Itô-Stratonovich equivalence on $N$ (see Proposition \ref{Conversion Ito strato dimension finie}), $(U_t)_{t \geq 0}$ is solution to the following Itô's SDE:

$$
\circ \d U_t = \SN A_i^\V(U_t) \circ \d W^i_t + A_0^\V(U_t) \d t - \frac{1}{2} \SN \nabla_{A_i^\V}A_i^\V(U_t) \d t.
$$
Since $A_i^\V$ is right invariant for all $i=0,1, \dots, N$, it is also the case for $\nabla_{A_i^\V} A_i^\V$ by Lemma \ref{Invariance à droite dérivée covariante dim finie}. In particular, we can apply Proposition \ref{P D2} and we deduce directly the equation for $(h_t)_{t \geq 0}$ since $A_i^\V$ is vertical for all $i=0,1, \dots, N$. For $(g_t)_{t \geq 0}$ we obtain the following Stratonovich's SDE:

$$
\circ \d g_t = \SN TR_{g_t}(\varpi(A_i^\V(h_t))) \d W^i_t + TR_{g_t}(\varpi(A_0^\V(h_t))) \d t - \frac{1}{2} \SN TR_{g_t}\left( \varpi \left( \nabla^{N_{p(h_t)}}_{A_i^\V}A_i^\V(h_t) \right) \right) \d t.
$$
The proof is done.
\end{proof}

\begin{rem}\label{Remarque fibres totalement géodésiques dim finie}
    Proposition \ref{P D3} tells us that an Itô's diffusion on $N$ driven by vertical vector fields is vertical if the second fundamental form vanishes on each fiber $N_x$. Note that, for $A_1,A_2$ vertical vector fields and $u \in N_x$:
    
    $$
    \nabla_{A_1}A_2(u) = \nabla^{N_x}_{A_1}A_2(u) + \mathrm{I\!I}^{N_x}(A_1, A_2)(u).
    $$
    In particular, the second fundamental form vanishes if and only if $\nabla^{N_x} = \nabla |_{N_x}$. This is the case if and only if $N_x$ is a totally geodesic submanifold of $N$.
\end{rem} 
    
\section{Infinite-dimensional geometric setting} \label{Section 3}

Let $M$ be a connected closed Riemannian manifold with dimension $m$, $\pi : TM \rightarrow M$ its tangent bundle with a Riemannian metric denoted by $\langle \bullet, \bullet \rangle$. Let $d_M$ denotes the Riemannian distance and $\vol$ the renormalized volume measure.  This section is devoted to the study of the infinite-dimensional manifolds that will play a central role in this work.
In Subsection \ref{Subsection 3.1}, we recall some geometric tools on the Wasserstein space $\P$.
Subsection \ref{Subsection 3.2} is concerned with the geometric structure of the diffeomorphism group $\mathscr{D}$.
Finally, in Subsection \ref{Subsection 3.3}, we show that $\mathscr{D}$ can be viewed as a Riemannian submersion onto $\P_\infty$, and we derive a few simple consequences of this observation.

\subsection{The Wasserstein space} \label{Subsection 3.1}

In this subsection we present the geometric structure on $\P$ introduced by Otto in \cite{Otto31012001}; see \cite{lott2007geometriccalculationswassersteinspace} and \cite{gigli2011second} for further details. Let $\P$ denote the \textit{Wasserstein space} over $M$ i.e. the space of probability measures endowed with the Wasserstein distance $W_2$:

\begin{equation*}
    W_2^2(\mu, \nu) = \inf \left \{ \int_{M \times M} d_M^2(x,y) \d\pi(x,y) \, : \, \pi \in \mathcal{C}(\mu,\nu) \right \},
\end{equation*}
where $\mathcal{C}(\mu,\nu)$ denotes the set of \textit{transport plans} between $\mu$ and $\nu$. Let $s>m/2$ and let:

$$
\P^s := \left\{ \mu \in \P \, : \, \frac{\d \mu}{\d \vol} = \rho \in H^s(M,\mathbb{R}),\, \rho > 0 \right\}.
$$
Moreover, let $\P_\infty \subset \P$ denote the \textit{smooth Wasserstein space} over $M$, namely:

$$
\P_\infty = \left\{ \mu \in \P \, : \, \frac{\d \mu}{\d \vol} = \rho \in \C^\infty(M),\, \rho >0 \right \}
$$
equipped with the same distance. Additionally, let us mention the following fact which is proved using topological arguments in probability spaces (see for example \cite[Remark 6.19]{villani2009optimal}).

\begin{prop}
	If $M$ is a compact metric space, the associated Wasserstein space is compact too.
\end{prop}

The \textit{tangent space} of $\P$ in $\mu \in \P$ is defined to be:

$$
\T_\mu \P = \overline{ \left \{ \nabla \phi \, : \, \phi \in \C^\infty(M) \right \} }^{L^2(\mu)},
$$
and the \textit{tangent bundle} is $ \T \P = \underset{\mu \in \P}{\Pi}\T_\mu \P$. We also introduce

$$
T_\mu \P^s = \left \{ \nabla \phi \, : \, \phi \in H^{s+2}(M,\mathbb{R}) \right \}, \quad \mu \in \P^s
$$
and

$$
T_\mu \P_\infty = \left \{ \nabla \phi \, : \, \phi \in \C^\infty(M) \right \}, \quad \mu \in \P_\infty.
$$
We denote by $T\P^s$ and $T\P_\infty$ the associated tangent bundles. The \textit{Otto's metric} on $\T\P$ is defined to be

$$
\langle Z_1, Z_2 \rangle_\mu = \int_M \langle Z_1(x), Z_2(x) \rangle \d\mu,
$$
for all $Z_1, Z_2 \in \T_\mu \P$.

\begin{defi}
	A \textit{vector field} on $\P$ is a map $\bar{Z} : \P \rightarrow \T \P$ such that $\bar{Z}(\mu) \in \T_\mu \P$. The notion of vector fields on $\P^s$ and $\P_\infty$ are defined analogously. 
\end{defi}

\begin{exemple}
	Let $f \in \C^\infty(M)$. Let $V_f$ denotes the vector field such that for all $\mu \in \P$, $V_f(\mu) = \nabla f$ seen as an element of $\T_{\mu} \P$. Such vector fields are often called constant vector fields even if, as we will see later, their covariant derivatives do not vanish in general.
\end{exemple}

\vspace{10pt}

\noindent \textbf{Lie derivative.} Let $\mu \in \P$, we define the \textit{Lie derivative} of a functional $F$ in the direction $V_\phi \in T_\mu \P$ to be

$$
\bar{\mathcal{L}}_{Z} F(\mu) = \left. \frac{\d}{\d t} \right|_{t=0} F(\mu_t),
$$
where $(\mu_t)_{t \geq 0} =((U_t)_*\mu)_{t \geq 0}$ with

\begin{equation*}
    \partial_t U_t(x) = \nabla \phi(U_t(x)), \quad U_0(x) = x.
\end{equation*}
In particular, $(\mu_t)_{t \geq 0}$ satisfies

\begin{equation} \label{Equation dérivée mesure}
\left. \frac{\d}{\d t} \right|_{t=0} \mu_t = V_\phi, \quad \mu_0 = \mu,
\end{equation}
in the sense of distribution, i.e. 

$$
\left. \frac{\d}{\d t} \right|_{t=0} \int_M f(x) \d\mu_t(x) = \int_M \langle \nabla f(x), \nabla \phi(x) \rangle \d\mu(x), \quad \forall f \in \C^\infty(M).
$$
A functional $F$ is said to be \textit{differentiable} in $\mu$ if the map 

\begin{align*}
    T_\mu \P &\longrightarrow \mathbb{R} \\
    V_\phi &\longmapsto \bar{\mathcal{L}}_{V_\phi} F(\mu)
\end{align*}
is a bounded linear functional. In this case, this map extends by density to a unique bounded linear functional:

\begin{align*}
    \T_\mu \P &\longrightarrow \mathbb{R} \\
    Z &\longmapsto \bar{\mathcal{L}}_{Z} F(\mu)
\end{align*}
The \textit{gradient} of a differentiable functional in $\mu \in \P$ is the unique vector $\grad F(\mu)$ (the existence and uniqueness being guaranteed by Riesz representation theorem) such that for all $\bar{Z} \in \T_\mu \P$:

$$
\bar{\mathcal{L}}_{Z}F(\mu) = \langle \grad F(\mu), Z \rangle_\mu.
$$
We say that $F \in \C^1_w(\P)$ (weakly-$\C^1$) whenever $F$ is differentiable, the map

$$
\mu \in \P \longmapsto \bar{\mathcal{L}}_{\bar{Z}}F (\mu),
$$
is continuous for all $\bar{Z}=\nabla \phi$ with $\phi \in \C^{\infty,0}(M \times \P)$ and $\sup_{\mu \in \P}|\grad F(\mu)|_\mu < \infty$ (consistently with the compactness of $\P$). These notions extend immediately to the case $F : \P \rightarrow E$ with $E$ a real vector space.

A map $F \in \C^0(M \times \P)$ is said to belong to  $C_w^{\infty,1}(M \times \P)$ if: 

\begin{enumerate}
    \item The map $F$ belongs to $\C^{\infty,0}(M \times \P)$. Moreover, for any $x \in M$, $F(x,\cdot)$ is of class $\C^{1}_w$, and the map 
    
    $$
    (x,\mu) \longmapsto \grad F(x,\mu)
    $$ 
    is continuous.
    \item For all $x \in M$, $T^k_x F(\cdot,\cdot) : \P \rightarrow \mathrm{Sym}^k(T_x^*M)$ is of class $\C^1_w$ for any $k \geq 1$.
    \item For all $\bar{Z}=\nabla \phi$ with $\phi \in \C^{\infty,0}(M \times \P)$, the map $(x,\mu) \mapsto \bar{\mathcal{L}}_{\bar{Z}}F (x,\mu)$ is $\C^{\infty,0}(M\times \P)$ and
    
    $$
\bar{\mathcal{L}}_{\bar{Z}} T^k F(x,\mu) = T^k \bar{\mathcal{L}}_{\bar{Z}}F(x,\mu), \quad \forall k \geq 1.
$$
\end{enumerate}

\vspace{10pt}

\noindent \textbf{Observables.} Let $f \in \C^\infty(M)$, and let $F_f : \P  \rightarrow \mathbb{R}$ be the so called \textit{potential energy functional}, namely

\begin{equation} \label{E energie potentielle}
    F_f(\mu) = \int_M f(x) \d\mu(x).
\end{equation}
One can easily check that $F_f \in \C^1_w(\P)$.
In this article, we will consider only such functionals, but we present two more functionals which are intensively studied. The first one is the $H$-functional of Boltzmann:

\begin{equation*}
    H(\mu) = 
\begin{cases}
\int_M \rho \log(\rho) \d\vol & \text{if } \d\mu = \rho \d\vol \\
+ \infty & \text{else},
\end{cases}
\end{equation*}
and the second one is the interaction energy functional:

\begin{equation} \label{E Energie interaction}
    \mathcal{W}(\mu) = \int_{M \times M} E(x,y) \d\mu(x) \d\mu(y),
\end{equation}
where $E$ is $\C^\infty(M \times M)$. Note that the latter is also of class $\C^1_w$ since 

$$
\bar{\mathcal{L}}_{V_\phi} \mathcal{W}(\mu) = \int_M \langle \nabla e(x,\mu), \nabla \phi(x) \rangle \d \mu(x), \quad \forall \phi \in \C^\infty(M),
$$
with

$$
e(x,\mu) = \int_M E(x,y) + E(y,x) \d \mu(y).
$$

One of the conclusion of the Otto's formalism is the fact that, by considering the density $\rho_t = \d \mu_t / \d \vol$ of the measure $\mu_t$ evolving according to the $H$-gradient flow on the Wasserstein space, namely

$$
\frac{\d}{\d t} \mu_t = \grad H(\mu_t), \quad \mu_0 = \mu \in \P_\infty,
$$
one recovers the heat equation

$$
\frac{\d}{\d t} \rho_t = - \Delta \rho_t.
$$
This result follows directly from the identity $\grad H(\mu_t) = V_{\log \rho_t}$. Otto extended this result in \cite{Otto31012001} by solving the porous medium equation 

$$
\frac{\d}{\d t}\rho_t = - \Delta \rho_t^m
$$
with gradient flows on Wasserstein. For more details about gradient flows in the Wasserstein space, the reader is refered to \cite{villani2009optimal}, \cite{ambrosioGradientFlowsMetric2008} and references therein.

\begin{exemple}
	Let $f \in \C^\infty(M)$, we have for all $\phi \in \C^\infty(M)$:
	
	$$
	\bar{\mathcal{L}}_{V_\phi} F_f(\mu) = \int_M \langle \nabla f, \nabla \phi \rangle \d\mu.
	$$
	In particular,
	
	$$
	\grad F_f = V_f.
	$$
\end{exemple}

We define the set of $\C^1_w$ vector fields on $\P$ to be:
	
	\begin{equation*}
	    \Gamma^1_w(T\P) = \left \{ \nabla \phi(\cdot) \, : \, \phi \in \C_w^{\infty,1}(M \times \P) \right\}.
	\end{equation*}
	Throughout this article we will consider such vector fields rather than general vector fields. This choice makes simpler the introduction of essential geometric objects such as Levi-Civita connection or Lie bracket. We will also consider $\mathcal{C}^1$ vector fields on $\P_\infty$, namely,
	
	\begin{equation*}
	    \Gamma^1_w(T\P_\infty) = \left \{ \nabla \phi(\cdot) \, : \, \phi \in \C_w^{\infty,1}(M \times \P_\infty) \right\}.
	\end{equation*}
	
    \vspace{10pt}
\noindent \textbf{Levi-Civita connection.} Let $\bar{Z} = \nabla \phi(\cdot) \in \Gamma^1_w(T\P)$ and $v \in \T_\mu \P$. The \textit{Levi-Civita connection} $\grad$ on $\P$ is defined as in \cite{lott2007geometriccalculationswassersteinspace} in the following way:

\begin{equation*}
    \grad_{v}\bar{Z}(\mu) = \Pi_\mu \left( \nabla_{v}\nabla \phi(\mu) \right) + \nabla \L_{v} \phi(\mu),
\end{equation*}
where $\Pi_\mu$ is the orthogonal projection from $L^2(M,TM,\mu)$ to $\T_\mu \P$.
The Levi-Civita connection is obviously tensor in the first component and it defines a metric connection, i.e. for $\bar{Z}_1 = \nabla \phi_1(\cdot), \bar{Z}_2 = \nabla \phi_2(\cdot) \in \Gamma^1_w(TP)$:

$$
\L_{v} \langle \bar{Z}_1(\mu), \bar{Z}_2(\mu) \rangle_\mu = \langle \grad_{v} \bar{Z}_1(\mu), \bar{Z}_2(\mu) \rangle_\mu + \langle  \bar{Z}_1(\mu), \grad_{v} \bar{Z}_2(\mu) \rangle_\mu.
$$
Indeed, let $f \in \C^\infty(M)$, $\dot{U}_t = \nabla f(U_t), U_0 = \mathrm{id}$ and $\mu_t = (U_t)_* \mu$, we have

\begin{align*}
\left. \frac{d}{dt} \right|_{t=0} \int_M \left\langle \nabla \phi_1(\mu_t)(x), \nabla \phi_2(\mu_t)(x) \right\rangle \d\mu_t(x) & = \int_M \left. \frac{d}{dt} \right|_{t=0} \left\langle \nabla \phi_1(\mu_t)(U_t(x)), \nabla \phi_2(\mu_t)(U_t(x)) \right \rangle \d\mu(x) \\
& = \int_M \left \langle \nabla_{\nabla f} \nabla \phi_1(\mu)(x) + \nabla \L_{V_f} \phi_1(\mu)(x), \nabla \phi_2(\mu)(x) \right \rangle \d\mu(x) \\
& + \int_M \left \langle \nabla \phi_1(\mu)(x), \nabla_{\nabla f} \nabla \phi_2(\mu)(x) + \nabla \L_{V_f} \phi_2(\mu)(x) \right \rangle \d\mu(x)\\
& = \int_M \left \langle \Pi_\mu \left(\nabla_{\nabla f} \nabla \phi_1(\mu)\right)(x) + \nabla \L_{V_f} \phi_1(\mu)(x), \nabla \phi_2(\mu)(x) \right \rangle \d\mu(x) \\
& + \int_M \left \langle \nabla \phi_1(\mu)(x),  \Pi_\mu \left(\nabla_{\nabla f} \nabla \phi_2(\mu)\right)(x) + \nabla \L_{V_f} \phi_2(\mu)(x) \right \rangle \d\mu(x).
\end{align*}
Since this property is true for all $f \in \C^\infty(M)$, we can extend it to general vectors $v \in \T_\mu \P$ by density.

\begin{defi}
    A vector field $\bar{Z}$ is said to belong to $\Gamma_w^2(T\P)$ if, for every $\bar{Z}_1 \in \Gamma_w^1(T\P)$, the covariant derivative $\grad_{\bar{Z}_1} \bar{Z}$ exists and lies in $\Gamma_w^1(T\P)$. Higher regularity classes $\Gamma_w^k(T\P)$ are defined recursively according to the same principle.
\end{defi}

\begin{rem}[Hodge decomposition] \label{Remarque décompo hodge} 
Note that for $\mu = \rho \vol \in \P_\infty$, by Hodge decomposition (see Appendix \ref{Annexe Hodge}), any smooth vector field $A$ on $M$ can be written in a unique way as $A = \nabla \phi(\mu) + Y(\mu)$, where $\phi(\mu) \in \C^\infty(M)$ and $Y(\mu)$ is a free $\mu$-divergence smooth vector field for all $\mu \in \P_\infty$. In particular, for $\bar{Z}_1, \bar{Z}_2 \in \Gamma_w^1(TP)$, the covariant derivative $\grad_{\bar{Z}_1} \bar{Z}_2(\mu)$ is in $T_\mu \P$. This decomposition also holds for $L^2$ vector fields, see for example \cite{Schwarz1995}.
\end{rem} 

\begin{defi}
    A functional $F : \P \rightarrow \mathbb{R}$ is said to belong to $\C^2_w(\P)$ if for all $\bar{Z} \in \Gamma^1_w(T\P)$, $\bar{\mathcal{L}}_{\bar{Z}}F \in \C^1_w(\P)$. Higher regularity classes $\C^k_w(\P)$ are defined recursively following the same principle.
    
    A map $F : M \times \P \rightarrow \mathbb{R}$ is said to belong to  $C_w^{\infty,2}(M \times \P)$ if: 

\begin{enumerate}
    \item The map $F$ belongs to $\C^{\infty,1}_w(M \times \P)$. Moreover, for all $x \in M$, $F(x,\cdot)$ is of class $\C^{2}_w$.
    \item For all $x \in M$, $T^k_x F(\cdot,\cdot) : \P \rightarrow \mathrm{Sym}^k(T_x^*M)$ is of class $\C^2_w$ for any $k \geq 1$.
    \item For all $\bar{Z} \in \Gamma^1_w(T\P)$, the map $(x,\mu) \mapsto \bar{\mathcal{L}}_{\bar{Z}}F (x,\mu)$ is $\C^{\infty,1}_w(M\times \P)$ and
    
    $$
\bar{\mathcal{L}}_{\bar{Z}} T^k F(x,\mu) = T^k \bar{\mathcal{L}}_{\bar{Z}}F(x,\mu), \quad \forall k \geq 1.
$$
Higher regularity classes $\C^{\infty,k}_w(M \times \P)$ are defined recursively following the same principle.
\end{enumerate}
\end{defi}

\vspace{10pt}

\noindent \textbf{Lie Bracket.} Let $\bar{Z}_1 = \nabla \phi_1(\cdot),\bar{Z}_2 = \nabla \phi_2(\cdot) \in \Gamma_w^1(T\P)$. We then have

$$
\L_{\bar{Z}_1} \L_{\bar{Z}_2} F_f(\mu) = \L_{\bar{Z}_1} \langle \grad F_f(\mu), \bar{Z}_2 \rangle_\mu.
$$
By setting $\dot{U}_t = \nabla \phi_1(\mu)(U_t)$ with $U_0 = id$, we have:

\begin{align*}
\L_{\bar{Z}_1} \langle \grad F_f(\mu), \bar{Z}_2 \rangle_\mu &= \left. \frac{\d}{\d t}\right|_{t=0} \int_M \langle \nabla f(x), \nabla \phi_2(\mu_t)(x) \rangle \d\mu_t(x) \\
& = \int_M \left. \frac{\d}{\d t} \right|_{t=0} \langle \nabla f (U_t(x)), \nabla \phi_2(\mu_t)(U_t(x)) \rangle \d\mu(x) \\
& =  \int_M \left \langle \nabla f(x), \nabla \L_{\bar{Z}_1} \phi_2(\mu)(x) \right \rangle + \left \langle  \nabla f(x),\nabla_{\nabla \phi_1(\mu)(x)} \nabla \phi_2(\mu)(x) \right \rangle \d\mu(x) \\
& + \int_M \he f(\nabla \phi_1(\mu), \nabla \phi_2(\mu)) \d\mu(x),
\end{align*}
 where $\he f$ is the Riemannian Hessian of $f$. Thus, using the fact that $\he f$ is symmetric, we obtain

$$
[\bar{Z}_1, \bar{Z}_2]F_f(\mu) = \left(\L_{\bar{Z}_1} \L_{\bar{Z}_2} - \L_{\bar{Z}_2} \L_{\bar{Z}_1} \right) F_f(\mu) = \left\langle \grad_{\bar{Z}_1} \bar{Z}_2(\mu) -\grad_{\bar{Z}_2} \bar{Z}_1(\mu), \grad F_f(\mu) \right\rangle_\mu.
$$
In particular, $\grad$ is torsion-free. We then obtain the Koszul formula for the Levi-Civita connection on $\P$:

\begin{align*}
    2 \langle \grad_{\bar{Z}_1}\bar{Z}_2 , \bar{Z}_3 \rangle &= \L_{\bar{Z}_1}\langle \bar{Z}_2, \bar{Z}_3 \rangle + \L_{\bar{Z}_2}\langle \bar{Z}_3, \bar{Z}_1 \rangle - \L_{\bar{Z}_3}\langle \bar{Z}_1, \bar{Z}_2 \rangle \\
    & + \langle \bar{Z}_3, [\bar{Z}_1, \bar{Z}_2] \rangle - \langle \bar{Z}_2, [\bar{Z}_1, \bar{Z}_3] \rangle - \langle \bar{Z}_1, [\bar{Z}_2, \bar{Z}_3] \rangle.
\end{align*}

\vspace{10pt}

\noindent \textbf{Hessian.} Let $F \in \C^2_w(\P)$ and let $\bar{Z}_1, \bar{Z}_2 \in \Gamma^1(T\P)$, the quantity
    
    $$
    \bar{\he } F(\bar{Z}_1, \bar{Z}_2) = \L_{\bar{Z}_1} \L_{\bar{Z}_2} F -  \grad_{\bar{Z}_1} \bar{Z}_2 F
    $$
    is well defined. We call this quantity the \textit{Hessian} of $F$. Note that, as in the finite-dimensional case, we have the following identity:
    
    $$
    \bar{\he } F(\mu)(\bar{Z}_1, \bar{Z}_2) = \left\langle \grad_{\bar{Z}_1(\mu)}\grad F(\mu), \bar{Z}_2(\mu) \right \rangle_\mu.
    $$
    Moreover, since
    
    $$
    \bar{\he } F(\bar{Z}_1, \bar{Z}_2) - \bar{\he } F(\bar{Z}_2, \bar{Z}_1) = [\bar{Z}_1, \bar{Z}_2] - [\bar{Z}_1,\bar{Z}_2] = 0,
    $$
    $\bar{\he }F$ is cleary symmetric. Furthermore, since $\bar{\he }F$ is tensorial in the first component and symmetric, we deduce that it is tensorial in both components. In particular, since $\grad F_f = V_f$ for $f \in \C^\infty(M)$, it is easy to see that
    
    \begin{equation} \label{Equation hessienne constant}
        \bar{\he }F_f(\mu)(\bar{Z}_1(\mu),\bar{Z}_2(\mu)) = \left \langle \grad_{\bar{Z}_1(\mu)}V_f, \bar{Z}_2(\mu) \right \rangle_\mu = \int_M \left \langle \nabla_{Z_1} \nabla f, Z_2 \right \rangle \d\mu(x) = \int_M \he f(Z_1(\mu)(x),Z_2(\mu)(x)) \d\mu(x).
    \end{equation}
    
    \vspace{10pt}

\noindent \textbf{Optimal transport maps.} The Monge problem concerns the existence an optimal transport map between two probability measure $\mu,\nu \in \P$. On a compact manifold, McCann showed in \cite{mccann2001polar} that such a map exists and is unique provided that $\mu$ is absolutely continuous with respect to the volume measure. In this case, the transport map is given by $\exp(\nabla \phi)$ for a suitable function $\phi$. When both $\mu$ and $\nu$ are absolutely continuous with respect to the volume measure, the function $\phi$ is the unique solution to the so called Monge-Ampère equation. Moreover, the curve $t \in [0,1] \mapsto \left(\exp(t\nabla \phi)\right)_* \mu$ defines the unique minimizing geodesic joining $\mu$ and $\nu$. Since both $M$ and $\P$ are compact spaces, we get easily the following result:

      	\begin{prop} \label{Proposition fonctions lipschitz}
  	Let $\Phi \in \C_w^1(\P)$. Then, the functional $\Phi$ is Lipschitz:
  	
  	$$
  	|\Phi(\mu)- \Phi(\nu)| \leq | \grad \Phi|_\infty \, W_2(\mu,\nu), \quad \forall \mu, \nu \in \P,
  	$$
  	where $|\grad \Phi|_\infty = \sup_{\mu \in \P}|\grad \Phi(\mu)|_\mu$.
  	\end{prop}
  	
  	\begin{proof}
  	Let $\mu, \nu \in \P_\infty$ and $(\mu_t)_t$ be the unique minimizing geodesic from $\mu$ to $\nu$. By the Benamou-Brenier formulation of the Wasserstein distance we have the equality
  	
  	$$
  	W_2^2(\mu,\nu) = \int_0^1 |\dot{\mu}_t|^2 \d t.
  	$$
  	Then, by Cauchy-Schwarz inequality
  	
  	$$
  	|\Phi(\mu) -\Phi(\nu) | = \left| \int_0^1 \langle \grad \Phi(\mu_t), \dot{\mu}_t \rangle_{\mu_t} \d t \right| \leq |\grad \Phi|_\infty \int_0^1 |\dot{\mu}_t|_{\mu_t} \d t  \leq |\grad \Phi|_\infty \, \left(\int_0^1 |\dot{\mu}_t|_{\mu_t}^2 \d t \right)^{1/2} = | \grad \Phi|_\infty W_2(\mu,\nu).
  	$$
  	Since $\P_\infty$ is dense in $\P$, we obtain the result on $\P$.
  	\end{proof}
  	
  	\begin{coro} \label{PropLipschitzGeneralVectorField}

  		Let $A \in \C_w^{\infty,1}(M \times \P, TM)$ such that for all $\mu \in \P$, $A(\cdot, \mu)$ is a vector field. Then $A$ is globally Lipschitz in $\mu$, i.e. there exist a constant $L \geq 0$, independent of $x \in M$ such that:

  	$$
  	|A(x,\mu)- A(x,\nu)|_x \leq L \, W_2(\mu,\nu), \quad \forall \mu,\nu \in \P, \quad \forall x \in M.
  	$$
  	
  	\end{coro}
  	
  	\begin{proof}
  		Let $(\mu_t)_t$ be the minimizing geodesic from $\mu$ to $\nu$. Let $(e_1, \dots, e_n)$ be an orthonormal basis of $T_xM$, we denote $A^i$ the $i$-th coordinate in this basis. Then we have by Proposition \ref{Proposition fonctions lipschitz}:
 
  	$$
  	|A^i(x,\mu) - A^i(x,\nu) |_x \leq  |\grad A^i(x,\cdot)|_\infty \, W_2(\mu,\nu).
  	$$
  	Then,
  	
  	$$
  	|A(x,\mu) - A(x,\nu)|^2_x = \sum_{i=1}^n |A^i(x,\mu) - A^i(x,\nu)|^2_x \leq W_2^2(\mu,\nu)  \sum_{i=1}^n |\grad A^i(x,\cdot)|_\infty^2.
  	$$
  	We conclude by taking the supremum over all $x \in M$ in the right-hand side term.
  	\end{proof}
  	
  	\begin{rem}[On the smoothness of geodesics]
  	It is important to be careful when handling geodesics on $\P$. Indeed, if we take $\mu, \nu \in \P_\infty$, we are not guaranteed that the minimizing geodesic $(\mu_t)_t$ joining $\mu$ to $\nu$ stays in $\P_\infty$. A necessary condition for the smoothness of optimal transport map is the so called MTW (standing for Ma-Trudinger-Wang) condition introduced in \cite{MaTrudingerWang2005}. The reader is referred for instance to \cite{Fig2010}, \cite{KimMcCann2008}, \cite{Loep2006}, \cite{LoeperVillani2010} for sufficient conditions to the smoothness of the optimal transport map.
  	\end{rem}
    
	\vspace{10pt}
	
	\noindent \textbf{Ordinary differential equations.} The notion of ordinary differential equation is well understood on $\P$, indeed, we have the following result:
	
	\begin{prop} \label{ODE on P}
		Let $\bar{Z} = \nabla \phi (\cdot) \in \Gamma^1_w(T\P)$ and $\mu \in \P$. Then, the ordinary differential equation
		
		\begin{equation*}
		\dot{M}_t(\mu) = \bar{Z}(M_t(\mu)), \quad M_0(\mu) = \mu
		\end{equation*}
		admits a solution given by $M_t(\mu) = (U_t)_* \mu$, where $(U_t)_{t \geq 0}$ is the solution flow to:
		
		$$
		\dot{U}_t(x) = \nabla \phi(M_t(\mu), U_t(x)), \quad U_0(x) = x, \text{ and } M_t(\mu) = (U_t)_* \mu.
		$$
	\end{prop}
	
	\begin{proof}
	See for example Theorem 2.2 in \cite{ding2021geometrywassersteinspacecompact} or adapt the proof of Theorem \ref{Existence Unicité} to the framework of ODEs.
	\end{proof}

\subsection{The group of diffeomorphisms} \label{Subsection 3.2}

Let $(\mathscr{M}^s)_{s \in \mathbb{N}}$ be a Hilbert manifold such that $\mathscr{M}^{s+1} \subset \mathscr{M}^s$ for all $s \in \mathbb{N}$ and consider the following Inverse Limit Hilbert (ILH for short) manifold:

\begin{equation*}
	\mathscr{M} = \bigcap_{s \geq 0} \mathscr{M}^s.
\end{equation*}
Let $\mathscr{N} = \cap_{s \geq 0} \mathscr{N}^s$ be another ILH manifold.
A map $f : \mathscr{M} \rightarrow \mathscr{N}$ is said to be of class $\C^k$ for the ILH topology if there exists $s_0 \in \mathbb{N}$ such that for all $s>s_0$ there exists $j(s)$ such that $f$ admits a $\C^k$ extension $f^s : \mathscr{M}^{j(s)} \rightarrow \mathscr{N}^s$. A map $f : \mathscr{M} \rightarrow \mathscr{M}$ is said to be a $\C^k$ stratified map if for all $s$ there exists a $\C^k$ extension $f^s : \mathscr{M}^s \rightarrow \mathscr{M}^s$. We denote by $\C^k_S(\mathscr{M},\mathscr{M})$ the set of $\C^k$ stratified maps.

Let $M$ be a closed manifold and let $\pi : E$ be a fiber bundle over $M$. 
The set $\Gamma^{(s)}(E)$ of $H^s$ sections of $E$ has a natural structure of Hilbert manifold, see \cite{Palais66} for instance.
This is the case for the bundle $M \times N$ where $N$ is another Riemannian manifold.
Note that if $E$ is a vector bundle, $\Gamma^{(s)}(E)$ is a Hilbert space. 
Classical examples are given by $E= M \times \mathbb{R}$ or $E= TM$.
The group of (orientation-preserving) diffeomorphisms $\mathscr{D} := \mathrm{Diff}(M,M)$ is an ILH manifold. 
More precisely, for each $s> m/2$, we consider the following Hilbert manifold:

\begin{equation*}
	\mathscr{D}^s = \left\{ \varphi \in H^s(M,M) \, : \, \varphi \text{ is bijective and } \varphi^{-1} \in H^s(M,M) \right\},
\end{equation*}
and the inverse limit description is given by:

\begin{equation*}
	\mathscr{D} = \bigcap_{s >m/2} \mathscr{D}^s.
\end{equation*}
 The reader is referred to \cite{Angst2020}, \cite{Ebin70}, \cite{EbinMarsden1970}, \cite{omori_group_1970} and references therein for a detailed presentation of the ILH structure of $\mathscr{D}$. We can also endow $\mathscr{D}$ with a ILH Lie group structure. Indeed, let $\varphi \in \mathscr{D}$, the following maps
 
 \begin{align*}
    & R_\varphi : \psi \longmapsto \psi \cdot \varphi \\
    & L_{\varphi} : \psi \longmapsto \varphi \cdot \psi \\
    & \mathrm{Inv} : \psi \longmapsto \psi^{-1},
 \end{align*}
 are smooths maps from $\mathscr{D}$ to $\mathscr{D}$ (see \cite[\S 3]{Ebin1970}.

We start this section by establishing the surjection property of the projection

\begin{align*}
    p : \mathscr{D} &\longrightarrow \P_\infty \\
     \varphi &\longmapsto (\varphi)_* \vol,
\end{align*}
namely the map which sends a diffeomorphism to the pushforward of the volume measure by this diffeomorphism.

\begin{prop}
	Let $s>m/2$. The projection $p : \mathscr{D}^{s+1} \rightarrow \P^s$ is surjective and continuous. As a consequence, $p : \mathscr{D} \rightarrow \P_\infty$ is surjective and continuous.
\end{prop}

\begin{proof}
The surjectivity is done in \cite[Lemma 5.2]{EbinMarsden1970}. To prove the continuity, let $(\varphi_n)_{n \geq 0}$ be a sequence of elements of $\mathscr{D}^{s+1}$ converging toward $\varphi \in \mathscr{D}^{s+1}$. Then, we have:

\begin{equation} \label{E 3.6}
W_2^2(p(\varphi_n), p(\varphi)) = W_2^2((\varphi_n)_* \vol, (\varphi)_* \vol) \leq \int_M d_M^2(\varphi_n(x), \varphi(x)) \d \vol(x) \leq \sup_{x \in M} d_M^2(\varphi_n(x), \varphi(x)),
\end{equation}
and, since $(\varphi_n)_n$ converges toward $\varphi$ for the $H^{s+1}$ topology, by Sobolev injection, $(\varphi_n)_n$ converges toward $\varphi$ in $\C^0$ topology. Consequently, the right hand-side of \eqref{E 3.6} goes to $0$ when $n \to + \infty$. Since this result is true for all $s > m/2$, the projection $p : \mathscr{D} \rightarrow \P$ is continuous.
\end{proof}

Throughout the rest of the article we will let $\mathscr{D}_\mu$ denotes the \textit{fiber} above $\mu$, i.e $\mathscr{D}_\mu = p^{-1}(\mu)$.

\vspace{10pt}
    
\noindent \textbf{Riemannian structure on $\mathscr{D}$.} The \textit{tangent space} of $\mathscr{D}$ in $\varphi \in \mathscr{D}$ is defined to be:
    
    $$
    T_{\varphi} \mathscr{D} := \Gamma(\varphi^*TM),
    $$
    and the tangent space of $\mathscr{D}^s$ in $\varphi \in \mathscr{D}^s$ is defined to be:
    
    $$
    T_{\varphi} \mathscr{D}^s := \Gamma^{(s)}(\varphi ^*TM),
    $$
    In particular, for $\varphi \in \mathscr{D}$, $T_\varphi \mathscr{D} = \cap_{s > m/2} T_\varphi \mathscr{D}^s$. Moreover, $T_{id}\mathscr{D}^s = \Gamma^{(s)}(TM)$ and $T_\varphi \mathscr{D}^s \cong T_{id} \mathscr{D}^s$ for all $\varphi \in \mathscr{D}^s$, the isomorphism being given by $\bm T R_\varphi : A \mapsto A \circ \varphi$ or $\bm T L_{\varphi}$. The set $T \mathscr{D} =  \underset{\varphi \in \mathscr{D}}{\sqcup} T_{\varphi}\mathscr{D}$ is the tangent bundle of $\mathscr{D}$ and it admits a structure of smooth ILH vector bundle over $\mathscr{D}$.
    
    \begin{defi}
    A \textit{vector field} on $\mathscr{D}$ is a map $\bm A : \mathscr{D} \rightarrow T \mathscr{D}$ such that $\bm A(\varphi) = A(\varphi) \circ \varphi$ with $A(\varphi) \in \Gamma(TM)$ for all $\varphi \in \mathscr{D}$.
    \end{defi}  
    
    We can build \textit{right invariant vector fields} on $\mathscr{D}$ similarly as on Lie groups by setting $ \check{A}(\varphi) = A \circ \varphi$ for $A \in T_{\mathrm{id}}\mathscr{D}$.
    
    \begin{defi}\begin{enumerate}
        \item A vector field $\bm A$ on $\mathscr{D}$ is said to be of class $\C^k$ if $\bm A : \mathscr{D} \rightarrow T\mathscr{D}$ is $\C^k$ as a ILH manifold map. We denote by $\Gamma^k(T\mathscr{D})$ the set of $\C^k$ vector fields on $\mathscr{D}$.
        \item A vector field $\bm A \in \Gamma^k(T\mathscr{D})$ on $\mathscr{D}$ is said to be a $\C^k$ stratified vector field if there exists $s_0$ such that for all $s > s_0$ there exists a $\C^k$ extension $A^s : \mathscr{D}^s \rightarrow T\mathscr{D}^s$. The set of $\C^k$ stratified vector fields is denoted by $\Gamma^k_S(T \mathscr{D})$. 
    \end{enumerate}
        
    \end{defi}
    
    Let $s \geq s' >m/2$, $\bm A \in \Gamma^k_{S}(T\mathscr{D})$, $\bm A^s \in \Gamma^k(T\mathscr{D}^s)$ and $\bm A^{s'} \in \Gamma^k(T\mathscr{D}^{s'})$ denote the $\C^k$ extensions of $\bm A$. Let also $\iota : \mathscr{D}^s \rightarrow \mathscr{D}^{s'}$ denote the smooth injection. Then, the definition of $\Gamma^k_S(T\mathscr{D})$ tells us that the following identity holds:
        
        \begin{equation} \label{Identite iota}
            \iota_*(\bm A^{s})(\varphi) = \bm A^{s'}(\varphi), \quad \forall \varphi \in \mathscr{D}^s.
        \end{equation}
    
    In the following we will only consider stratified vector fields. Let $s > m/2$, we consider the following $L^2$ \textit{Riemannian metric} on $T\mathscr{D}^s$:
    
    \begin{equation} \label{E Métrique Riemanienne}
        \langle A \circ \varphi, B \circ \varphi \rangle_{\varphi} = \int_M \langle A(\varphi(x)), B(\varphi(x)) \rangle \d\vol(x), \quad A \circ \varphi, B \circ \varphi \in T_\varphi \mathscr{D}^s.
    \end{equation}
    We also endow $T\mathscr{D}$ with this Riemannian metric.
    
    \begin{rem}[On the different structures on $\mathscr{D}$]
        Note that the topology induced by the Riemannian metric on $\mathscr{D}$ is strictly weaker than the ILH Lie group topology mentioned in the introduction of this subsection. In the literature we talk about weak Riemannian metric, see \cite{Angst2020}, \cite[Section 9]{EbinMarsden1970} for instance. On the one hand, the ILH Lie group structure is useful to prove results of existence and uniqueness. In particular, this point of view is especially powerful since it allows us to check that solutions to both ordinary differential equations (Proposition \ref{ODE D}) and to stochastic differential equations (Proposition \ref{Existence unicité Difféos}) lie in $\mathscr{D}$. We will use this structure mainly in Sections \ref{Section 3} and \ref{Section 4}. On the other hand, the Riemannian structure, although weaker, is useful since it allows us to consider $\mathscr{D}$ as a (quasi) Riemannian manifold and as a Riemannian submersion onto $\P_\infty$, see Subsection \ref{Subsection 3.3}. This structure will be central in Sections \ref{Section 5} and \ref{Section 6}.
    \end{rem}
    
\vspace{10pt}
    
\noindent    \textbf{Structure of principal bundle.} Let us define
    
    $$
    G := \left \{ g \in \mathscr{D} \, : \, p(g) = \vol \right \} = \mathscr{D}_{\vol}.
    $$
    The group $\mathscr{D}$ has a structure of $G$-principal bundle over $\P_\infty$. Indeed, since for all $g \in G$ and $\varphi \in \mathscr{D}$ we have $p(\varphi \circ g) = p(\varphi)$ it is clear that the right action of $G$ preserves the fiber. The fact that $G$ acts freely on the fibers is clear. To see that $G$ acts transitively, let $\varphi_1, \varphi_2 \in \mathscr{D}$ be such that $p(\varphi_1) = p(\varphi_2) = \mu$. The diffeomorphism given by $\varphi_1^{-1} \cdot \varphi_2$ is in $G$ since
    
    $$
	p(\varphi_1^{-1} \cdot \varphi_2) = (\varphi_1^{-1})_* \mu = \vol
    $$ 
    and we clearly have $\varphi_2 = \varphi_1 \cdot (\varphi_1^{-1} \cdot \varphi_2)$. Let $s>m/2$, we also define the following:
    
    \begin{equation*}
	G^s = \left\{ g \in H^s(M,M) \, : \, g \text{ is bijective, } g^{-1} \in H^s(M,M), \, p(g) = \vol \right\},
\end{equation*}
    
    With this principal bundle structure, we may define the counterpart of Definition \ref{Définition équivariance} in the setting of diffeomorphisms.
    
        \begin{defi} \label{Définition équivariance sur D}
        A map $\Phi : \mathscr{D} \rightarrow \mathscr{D}$ is said to be \textit{equivariant} if
        
        $$
        \Phi(\varphi \cdot g) = \Phi(\varphi) \cdot g, \quad \forall \varphi \in \mathscr{D}, \, \forall g \in G.
        $$
        A flow of maps $(\Phi_t)_{t \geq 0}$ such that $\Phi_t : \mathscr{D} \rightarrow \mathscr{D}$ for all $t\geq 0$ is said to be equivariant if $\Phi_t$ is equivariant for all $t \geq 0$.
    \end{defi} 
    
    \begin{rem} \label{R G right invariance métrique R}
        Let $s>m/2$. Note that the Riemannian metric \eqref{E Métrique Riemanienne} is $G$-right invariant, namely, for all $g \in G^s$ and all $\varphi \in \mathscr{D}^s$,
        
        \begin{equation*}
            \langle \bm T R_g(A \circ \varphi), \bm T R_g(B \circ \varphi) \rangle_{\varphi \cdot g} = \langle A \circ \varphi, B \circ \varphi \rangle_\varphi, \quad \forall A,B \in T_\varphi \mathscr{D}^s.
        \end{equation*}
        Indeed, let $A \circ \varphi, B \circ \varphi \in T_\varphi \mathscr{D}^s$, we have:
        
        \begin{align}
            \langle \bm T R_g(A \circ \varphi), \bm T R_g(B \circ \varphi) \rangle_{\varphi \cdot g} &= \int_M \langle A(\varphi(g(x))), B(\varphi(g(x))) \rangle \d \vol(x) \nonumber \\
            & = \int_M \langle A(\varphi(x)), B(\varphi(x)) \rangle \d \vol(x) \label{E 3.8} \\
            & = \langle A \circ \varphi, B \circ \varphi \rangle_\varphi \nonumber,
        \end{align}
        where \eqref{E 3.8} is due to the fact that $(g)_* \vol = \vol$.
    \end{rem}

    \vspace{10pt} 
    
    \noindent \textbf{Lie derivative.} Let $s>m/2+1$, $\bm F \in \C^1(\mathscr{D}^s)$, $\varphi \in \mathscr{D}^s$ and $A \circ \varphi \in T_\varphi \mathscr{D}^s$. We say that a path $(\varphi_t)_t$ on $\mathscr{D}^s$ with $\varphi_0 = \varphi$ satisfies $\partial_t |_{t=0} \varphi_t =A \circ \varphi$ if for all $x \in M$, $\partial_t|_{t=0} \varphi_t(x) = A(\varphi(x))$. Then, we can define the \textit{Lie derivative} of $\bm F$ in the direction $A \circ \varphi$ to be
    
    $$
    \mathcal{L}^{\mathscr{D}^s}_{A\circ \varphi} \bm F(\varphi) = \left.\frac{\d}{\d t} \right|_{t=0} \bm F(\varphi_t).
    $$
    The \textit{gradient} of $\bm F$ in $\varphi \in \mathscr{D}^s$ is the unique vector $\bm \nabla \bm F(\varphi) \in T_\varphi \mathscr{D}^s$ (the existence and uniqueness being guaranteed by Riesz's representation theorem) such that
    
    $$
    \mathcal{L}^{\mathscr{D}^s}_{A \circ \varphi}\bm F(\varphi) = \langle \bm \nabla \bm F(\varphi), A \circ \varphi \rangle_\varphi, \quad \forall A \circ \varphi \in T_\varphi \mathscr{D}^s.
    $$
    Moreover, for $\bm F \in \C^k(\mathscr{D}^s)$, $\bm \nabla \bm F \in \Gamma^{k-1}(T\mathscr{D}^s)$.
    
    Let $\bm F \in \C^1(\mathscr{D})$. The Lie derivative of $\bm F$ in the direction $A \circ \varphi$ is defined in the same way as in the $\mathscr{D}^s$ case. The gradient of $\bm F$ is obtained as the limit object of the one on $T \mathscr{D}^s$ and is also denoted by $\bm \nabla$. Moreover, if $\bm F \in \C^k(\mathscr{D})$, then $\bm \nabla \bm F \in \Gamma^{k-1}(T\mathscr{D})$.

    \begin{exemple}
    Let $f \in \C^\infty(M \times M)$. Let $\bm F_f$ be the functional on $\mathscr{D}$ defined by
    
    $$
    \bm F_f(\varphi) = \int_M f(\varphi(x),x) \d\vol(x).
    $$
    Let $A \circ \varphi \in T_\varphi \mathscr{D}$, we have
    
    $$
    \mathcal{L}^\mathscr{D}_{A \circ \varphi} \bm F_f(\varphi) = \int_M \langle \nabla_1 f(\varphi(x),x), A(\varphi(x)) \rangle \d\vol(x).
    $$
    In particular, we have:
    
    $$
    \bm \nabla \bm F_f (\varphi) = \nabla_1 f( \varphi(\cdot), \cdot).
    $$
\end{exemple}
    We also define the \textit{Lie bracket} between two vector fields $\bm A, \bm B \in \Gamma^k_S(T\mathscr{D})$ to be the unique vector field satisfying:
    
    $$
    [\bm A, \bm B] \bm F(\varphi) = \mathcal{L}^\mathscr{D}_{\bm A} \mathcal{L}^\mathscr{D}_{\bm B} \bm F(\varphi) - \mathcal{L}^\mathscr{D}_{\bm B} \mathcal{L}^\mathscr{D}_{\bm A} \bm F(\varphi).
    $$
    Note that $[\bm A,\bm B] \in \Gamma^{k-1}_S(T\mathscr{D})$. 
    
        \vspace{10pt}
    
    \noindent \textbf{Levi-Civita connection.} Let $s >m/2$, $\bm A_1, \bm A_2, \bm A_3 \in \Gamma^1(T \mathscr{D}^s)$. The Levi-Civita connection $\bm \nabla^s$ on $T\mathscr{D}^s$ is defined to be the unique torsion free linear connection, namely,
    
    \begin{equation} \label{Torsion free D}
        \bm \nabla^s_{\bm A_1} \bm A_2 - \bm \nabla^s_{\bm A_2} \bm A_1 = [\bm A_1, \bm A_2],
    \end{equation}
    and that is compatible with the Riemannian metric, i.e.
    
    \begin{equation} \label{Compatible with metric D}
        \mathcal{L}^{\mathscr{D}^s}_{\bm C \circ \varphi} \langle \bm A_1(\varphi), \bm A_2 (\varphi) \rangle_\varphi = \langle \bm \nabla^s_{C \circ \varphi} \bm A_1 (\varphi), \bm A_2 (\varphi) \rangle_\varphi + \langle \bm A_1(\varphi), \bm \nabla^s_{C \circ \varphi} \bm B_2(\varphi) \rangle_\varphi.
    \end{equation}
    The Levi-Civita connection is defined in a unique way by the Koszul formula:
    
        \begin{align}
    2 \langle \bm \nabla^s_{\bm A_1}\bm A_2 , \bm A_3 \rangle &= \mathcal{L}^{\mathscr{D}^s}_{\bm A_1}\langle \bm A_2, \bm A_3 \rangle + \mathcal{L}^{\mathscr{D}^s}_{\bm A_2}\langle \bm A_3, \bm A_1 \rangle - \mathcal{L}^{\mathscr{D}^s}_{\bm A_3}\langle \bm A_1, \bm A_2 \rangle \label{Koszul D} \\
    & + \langle \bm A_3, [\bm A_1, \bm A_2] \rangle - \langle \bm A_2, [\bm A_1, \bm A_3] \rangle - \langle \bm A_1, [\bm A_2, \bm A_3] \rangle. \nonumber
\end{align}

We will sometimes forget the exponent $s$ for the sake of readability. In a general weak Riemannian framework, the existence and uniqueness of such a connection need not hold, in this setting the result is due to Ebin and Marsden, see \cite[Theorem 9.1]{EbinMarsden1970}. The reader is also referred to \cite{Eliasson1967} for a general construction of connections on manifolds of maps. The Levi-Civita connection on $T\mathscr{D}$ is obtained as the limit object of the Levi-Civita connections on $T \mathscr{D}^s$, is also denoted by $\bm \nabla$ and acts on elements of $\Gamma^1_S(T\mathscr{D})$.

Let $\check{A}$ be a right invariant vector field on $\mathscr{D}$, i.e. a vector field defined as $\check{A}(\varphi) = A \circ \varphi$ with $A \in \Gamma(TM)$. One can easily check by applying \eqref{Compatible with metric D} that
    
    \begin{equation*}
        \bm \nabla_{C \circ \varphi}\check{A}(\varphi) = \left( \nabla_C A\right) \circ \varphi.
    \end{equation*}
Let $\bm A(\varphi) = \nabla \phi(\varphi) \circ \varphi$ for all $\varphi \in \mathscr{D}$. One can, by applying \eqref{Compatible with metric D}, show that

\begin{equation*}
    \bm \nabla_{C \circ \varphi} \bm A(\varphi) = \left( \nabla_C \nabla \phi(\varphi) \right) \circ \varphi + \left(\nabla \mathcal{L}^\mathscr{D}_{C \circ \varphi} \phi(\varphi) \right) \circ \varphi.
\end{equation*}

\begin{rem}\label{R derivee covariante stratifiee}
    If $\bm A, \bm B \in \Gamma^k_S(T\mathscr{D})$ for $k \geq 1$, then $\bm \nabla_{\bm B} \bm A \in \Gamma^{k-1}_S(T\mathscr{D})$, and $\C^{k-1}$ extensions are given by $\bm \nabla^s_{\bm B^s} \bm A^s$ where $\bm A^s, \bm B^s \in \Gamma^k(T\mathscr{D}^s)$ are the $\C^k$ extensions of $\bm A$ and $\bm B$.
\end{rem}

\vspace{10pt}

    \noindent \textbf{Parallel transport.} Let $s >m/2$ and $(\varphi_t)_{t \geq 0}$ be a smooth path on $\mathscr{D}^s$ where $\varphi_0 = \varphi$ and $A \circ \varphi \in T_\varphi \mathscr{D}^s$, we define the map
    
    $$
    \tau_{0,t} : A \circ \varphi \mapsto \sslash_{0,t}(A \circ \varphi),
    $$
    where for all $x \in M$, $\sslash_{0,t}(x)$ is the parallel transport along $(\varphi_s(x))_{0 \leq s \leq t}$. One can check that it is an isometry between $T_{\varphi} \mathscr{D}^s$ and $T_{\varphi_t}\mathscr{D}^s$. The parallel transport on $T\mathscr{D}$ along a smooth path $(\varphi_t)_t$ on $\mathscr{D}$ is obtained as the limit object of the parallel transports on $T\mathscr{D}^s$.
    
    \begin{defi}
        Let $(\varphi_t)_{t \geq 0}$ be a smooth curve on $\mathscr{D}^s$ and $(\tau_{0,t})_{t \geq 0}$ the parallel transport along $(\varphi_t)_{t \geq 0}$. A vector field $(\bm A_t)_{t \geq 0}$ along $(\varphi_t)_{t \geq 0}$ is said to be derivable if $t \mapsto \tau_{0,t}^{-1} \bm A_t$ is derivable on $T_{\varphi_0} \mathscr{D}^s$. The covariant derivative of $(\bm A_t)_{t \geq 0}$ is defined to be
        
        \begin{equation*}
            \bm D_t \bm A_t = \tau_{0,t} \frac{\d}{\d t} \tau_{0,t}^{-1}(\bm A_t).
        \end{equation*}
    \end{defi}

    \begin{prop} \label{Prop 3.7}
        Let $\bm A \in \Gamma^1(T\mathscr{D}^s)$ and $C \circ \varphi \in T_\varphi \mathscr{D}^s$, the following holds:
    
    $$
    \bm \nabla_{C \circ \varphi} \bm A(\varphi) = \left. \frac{\d}{\d t}\right|_{t=0} \tau_{0,t}^{-1} \bm A(\varphi_t),
    $$ 
    with $(\tau_{0,t})_{t \geq 0}$ the parallel transport along a smooth path $(\varphi_t)_{t \geq 0}$ such that $\partial_t|_{t =0} \varphi_t = C \circ \varphi$ and $\varphi_0 = \varphi$.
    \end{prop}
    
    \begin{proof}
    Let $\bm A(\varphi) = A(\varphi) \circ \varphi$ for all $\varphi \in \mathscr{D}^s$, $(\varphi_t)_{t \geq 0}$ as in the statement of the proposition and $x \in M$. Then,
    
    \begin{align*}
        \left.\frac{\d}{\d t} \right|_{t =0} (\tau_{0,t}^{-1}\bm A(\varphi_t))(x) &= \left.\frac{\d}{\d t} \right|_{t =0} (\sslash_{0,t}^{-1}  A(\varphi_t)(\varphi_t(x))) \\
        &= \left.\sslash_{0,t}^{-1}\left( (\partial_t A(\varphi_t))(\varphi_t(x) \right) +  \sslash_{0,t}^{-1} D_t A(\varphi_t)(\varphi_t(x))\right|_{t=0} \\
        & = \left( (\mathcal{L}^{\mathscr{D}}_{C \circ \varphi} A(\varphi))(\varphi(x)) \right)  + \left( (\nabla_{C}A)(\varphi(x)) \right) \\
        & = \left( \bm \nabla_{C \circ \varphi} \bm A (\varphi) \right)(x).
    \end{align*}
    The desired identity now follows.
    \end{proof}
    
    In the remainder of the article, we will write either $\bm D_t \bm A(\varphi_t)$ or $\bm \nabla_{\dot{\varphi}_t}\bm A(\varphi_t)$ to denote the covariant derivative of $\bm A$ along $(\varphi_t)_{t \geq 0}$.

    \vspace{10pt}
    
    \noindent \textbf{Hessian.} Let $\bm F \in \C^2(\mathscr{D}^s)$ and $\bm A,\bm B \in \Gamma^1(T\mathscr{D}^s)$. The quantity
    
    $$
    \bm \he \bm F(\bm A, \bm B) = \mathcal{L}^{\mathscr{D}^s}_{\bm A} \mathcal{L}^{\mathscr{D}^s}_{\bm B} \bm F - \bm \nabla_{\bm A} \bm B \bm F
    $$
    is called the \textit{Hessian} of $\bm F$ and it is easy to see that, as in the finite dimensional case, the following identity holds:
    
    $$
    \bm \he \bm F(\varphi)(\bm A(\varphi), \bm B(\varphi)) = \langle \bm \nabla_{\bm A(\varphi)} \bm \nabla \bm F(\varphi), \bm B(\varphi) \rangle_\varphi.
    $$
    The operator $\bm \he $ is a symmetric tensor since, by \eqref{Torsion free D}, $\bm \nabla_{\bm A} \bm B - \bm \nabla_{\bm B} \bm A = [\bm A, \bm B]$.
    
    The Hessian of $\bm F \in \C^2(\mathscr{D})$ is obtained as the limit object of the one on $\mathscr{D}^s$, is denoted by $\bm \he \bm F$.
    
    	\vspace{10pt}
	
	\noindent \textbf{Ordinary differential equations.}
    
    	\begin{prop} \label{ODE D}
		Let $\bm A \in \Gamma^1_S(T\mathscr{D})$ and $\varphi \in \mathscr{D}$. Then, the ordinary differential equation
		
		\begin{equation*}
		\dot{\Phi}_t(\varphi) = \bm A(\Phi_t(\varphi)), \quad \Phi_0(\varphi) = \varphi
		\end{equation*}
		admits a unique solution.
	\end{prop}
	
	\begin{proof}
		This result is well known, it is sufficient to apply Cauchy-Lipschitz theorem on each $\mathscr{D}^s$ for $s$ large enough.
	\end{proof}

    \subsection{\texorpdfstring{$\mathscr{D}$}{D} as a Riemannian submersion onto \texorpdfstring{$\P_\infty$}{P}}\label{Subsection 3.3} This subsection is devoted to the presentation of the Riemannian submersion structure of $\mathscr{D}$ onto $\P_\infty$. The tools introduced here will be used intensively in Sections \ref{Section 5} and \ref{Section 6}.
    
    \noindent \textbf{Submersion.} The projection $p : \mathscr{D}^{s+1} \rightarrow \P^s$ is a \textit{submersion} in the following sense: for all $\varphi \in \mathscr{D}^{s+1}$, the map $\bm T p : T_{\varphi} \mathscr{D}^{s+1} \rightarrow T_{p(\varphi)} \P^s$ is surjective. Indeed, let $\mu \in \P^s$ and $V_\phi \in T_{\mu} \P^s$, it is clear that for all $ \varphi \in \mathscr{D}_\mu^{s+1}$, $\nabla \phi \circ \varphi$ is in $T_\varphi \mathscr{D}^s$. Let $\dot{\varphi}_t = \nabla \phi \circ \varphi_t$ with $\varphi_0 = \varphi$, we get (in the sense of distribution):
    
    $$
    T_{\varphi}p(\nabla \phi \circ \varphi) = \partial_t|_{t=0} p(\varphi_t) = V_\phi
    $$
    since for all $f \in \C^\infty(M)$:
    
    $$
    \left. \frac{\d}{\d t}\right|_{t=0} \left( \int_M f(\varphi_t(x)) \d \vol (x) \right) = \int_M \langle \nabla f(\varphi(x)), \nabla \phi(\varphi(x)) \rangle \d\vol(x) = \int_M \langle \nabla f, \nabla \phi \rangle \d\mu = \left \langle V_f, V_\phi \right \rangle_\mu.
    $$
    As a consequence, for $\varphi \in \mathscr{D}$, $\bm T p : T_\varphi\mathscr{D} \rightarrow T_{p(\varphi)}\P_\infty$ is surjective.
    
    \vspace{10pt}
    
	\noindent \textbf{Horizontal and vertical bundles.} Let us define the \textit{vertical tangent space} in $\varphi \in \mathscr{D}$ to be
	
	\begin{equation*}
	    \V_\varphi \mathscr{D} = \ker T_{\varphi} p.
	\end{equation*}
    Since $p$ is a submersion, the fiber $\mathscr{D}_{p(\varphi)}$ can be seen as a submanifold of $\mathscr{D}$ where $\mathscr{V}_\varphi T \mathscr{D}$ is the tangent space at $\varphi$. Given that $\mathscr{D}$ is equipped with a Riemannian metric we can define the \textit{horizontal tangent space} in $\varphi$ to be 
	
	\begin{equation*}
	    \H_\varphi \mathscr{D} = \V_\varphi \mathscr{D}^\perp \cap T_\varphi \mathscr{D}.
	\end{equation*}
	This construction defines a right inverse to $T_{\varphi}p : \H_\varphi \mathscr{D} \rightarrow T_{p(\varphi)} \P$, called the horizontal lift, we denote it by $\h_\varphi$. In particular, $\h_\varphi (V_\phi) =  \nabla \phi \circ \varphi$. To see that $p$ is a \textit{Riemannian submersion}, just notice that
    
    $$
    |V_\phi|_{\mu}^2 = \int_M | \nabla \phi(x) |^2 \d\mu(x) = \int_M | \nabla \phi(\varphi(x)) |^2 \d\vol = |\h_\varphi(V_\phi)|_\varphi^2.
    $$
    Note that the vertical component of the tangent space $T_\varphi \mathscr{D}$ is
    
    $$
    \V_\varphi \mathscr{D} = \left \{ Y \in \Gamma(\varphi^*TM), \, \text{div}_{p(\varphi)}(Y) = 0 \right \}
    $$
    while the horizontal one is
    
    $$
    \H_{\varphi} \mathscr{D} = \left \{ \nabla \phi \circ \varphi \, : \, \phi \in \C^\infty(M) \right \}.
    $$
    The horizontal (resp. vertical) bundle is denoted by $\mathscr{H}\mathscr{D}$ (resp. $\mathscr{V} \mathscr{D}$).

    \vspace{10pt}
    
    \noindent \textbf{Extensions of horizontal and vertical bundles.} Let $s > m/2 + 3$ and $\varphi \in \mathscr{D}^s$. By Appendix \ref{Annexe Hodge}, $T_{\varphi} \mathscr{D}^s = \V_{\varphi} \mathscr{D}^s \oplus \H_{\varphi} \mathscr{D}^s$ with
    
    $$
    \V_\varphi \mathscr{D}^s = \left \{  Y \in \Gamma^{(s)}(\varphi^*TM), \, \text{div}_{p(\varphi)}(Y\circ \varphi^{-1}) = 0 \right \}
    $$
    and
    
    $$
    \H_{\varphi} \mathscr{D}^s =\left \{ \nabla \phi \circ \varphi \, : \, \phi \in H^{s+1}(M,\mathbb{R}) \right \},
    $$
    where the orthogonal sum holds for the $L^2$ weak Riemannian metric. Moreover, the horizontal lift map is a linear continuous map $\h_\varphi : T_{p(\varphi)} \P^s \rightarrow \H_{\varphi} \mathscr{D}^{s+1}$.
    
        \begin{prop} \label{P Propriétés invariances à gauche et à droite des fibrés} 
    Let $s >m/2+3$. Then,
    
    \begin{enumerate}
    \item The map $\bm T R_\psi$ is an isomorphism between $\H_{\varphi} \mathscr{D}^s$ and $\H_{\varphi \cdot \psi} \mathscr{D}^s$ for all $\varphi, \psi \in \mathscr{D}^{s}$.
    \item The map $\bm T R_g$ is an isomorphism between $\V_\varphi \mathscr{D}^s$ and $\V_{\varphi \cdot g} \mathscr{D}^s$ for all $\varphi \in \mathscr{D}^{s}$ and $g \in G^{s}$.
    \item The map $\bm T R_g$ is an isometric isomorphism between $T_\varphi \mathscr{D}^s$ and $T_{\varphi \cdot g} \mathscr{D}^s$ for all $\varphi \in \mathscr{D}^{s}$ and $g \in G^{s}$.
    \item The map $\bm T L_\psi$ is an isomorphism between $\V_{\varphi} \mathscr{D}^s$ and $\V_{\psi \cdot \varphi}\mathscr{D}^s$ for all $\varphi \in \mathscr{D}^s, \psi \in \mathscr{D}^{s+1}$.
   
    \end{enumerate}
    \end{prop} 
    
    \begin{proof}
    \begin{enumerate}
        \item Let $\varphi, \psi \in \mathscr{D}$ and $Z \circ \varphi \in \H_\varphi \mathscr{D}^s$. We have $\bm T R_\psi (Z \circ \varphi) = Z \circ( \varphi \cdot \psi) \in \H_{\varphi \cdot \psi} \mathscr{D}^s$. The inverse map is given by $\bm T R_{\psi^{-1}}$.
        
    \item Let $\varphi \in \mathscr{D}$, $g \in G$, $Y \circ \varphi \in \V_\varphi \mathscr{D}^s$.  We have $\bm T R_g (Z \circ \varphi) = Z \circ (\varphi \cdot g) \in \V_{\varphi \cdot g} \mathscr{D}^s$. The inverse map is given by $\bm T R_{g^{-1}}$.
    
    \item This is a direct consequence of the fact that the Riemannian metric is $G$-right invariant, see Remark \ref{R G right invariance métrique R}.
    
    \item Let $Y \circ \varphi \in \V_\varphi \mathscr{D}^s$. Let $(\varphi_t)_{t \geq 0}$ such that $\varphi_0 = \varphi$ and $\dot{\varphi_t} = \check{Y}(\varphi_t)$, where $\check{Y}$ is the right invariant vector field associated to $Y$. Since $\mathrm{div}_{p(\varphi)}(Y)=0$ and the flow associated to a divergence free vector field is volume preserving, we have
    
    $$
    \bm T p (\bm T L_\psi(Y \circ \varphi)) = \left. \frac{\d}{\d t} \right|_{t=0}p(\psi \cdot \varphi_t) = \left. \frac{\d}{\d t} \right|_{t=0} p(\psi) = 0.
    $$
    Consequently, $\bm T L_\psi(Y \circ \varphi) \in \V_{\psi \cdot \varphi}  \mathscr{D}^s$. The inverse map is given by $\bm T L_{\psi^{-1}}$.
    \end{enumerate}

    \end{proof} 
    
    \begin{defi}
        Let $s >m/2+3$ and $\varphi \in \mathscr{D}^{j(s)}$. The \textit{vertical projection} (resp. \textit{horizontal projection}) is defined to be the orthogonal projection on the vertical space $P_{\V_\varphi} : T_{\varphi} \mathscr{D}^s \rightarrow \V_{\varphi}\mathscr{D}^s$ (resp. on the horizontal space $P_{\H_\varphi} : T_{\varphi} \mathscr{D}^s \rightarrow \H_{\varphi}\mathscr{D}^s$).
    \end{defi}
    
    \vspace{10pt}

    \vspace{10pt}
    \noindent \textbf{Stronger notions of regularity on $\P$.} Let us introduce additional notions of regularity on the Wasserstein space. Let $k \in \mathbb{N}$.
    	    \begin{enumerate}
    	        \item A functional $F : \P_\infty \rightarrow \mathbb{R}$ is said to be $\C^k(\P_\infty)$ if $F \circ p \in \C^k(\mathscr{D})$.
    	    \item A functional $F \in \C^k_w(\P)$ is said to be $\C^k(\P)$ if $F|_{\P_\infty} \in \C^k(\P_\infty)$.
    	    \item A vector field $\bar{Z} = \nabla \phi(\cdot) \in \Gamma_w^k(T\P)$ is said to belong to $\Gamma^k(T\P)$ if $\bm Z = \mathfrak{h}(\bar{Z}) \in \Gamma^k(T\mathscr{D})$.
    	    \item A vector field $\bar{Z} \in \Gamma^k(T\P)$ is said to belong to $\Gamma^k_S(T\P)$ if $\bm Z = \mathfrak{h}(\bar{Z}) \in \Gamma^k_S(T\mathscr{D})$.
    	    \end{enumerate}
    	    
    	    \begin{rem}
    	        One can check that both the potential energy functional \eqref{E energie potentielle} and the interaction energy functional \eqref{E Energie interaction} are of class $\C^2(\P)$.
    	    \end{rem}
    	    
    	    \begin{exemple}
    	    A representative example of a subspace of $\Gamma^\infty_S(\P)$ is the following:
    	    
    	    \begin{equation*}
    	        \left\{ \sum_{i=1}^l F_i V_{\phi_i}, \quad F_i \in \mathscr{A}, \, \phi_i \in \C^\infty(M), \, l \in \mathbb{N}^* \right\},
    	    \end{equation*}
    	     where $\mathscr{A}$ is the algebra of polynomial functions on $\P$, namely,
    	    
\begin{equation*}
\mathscr{A} := \left \{ F = \prod_{i=1}^k F_{f_i} \, : \, f_i \in \C^\infty(M), \, k \in \mathbb{N}^* \right \}.
\end{equation*}
    	    \end{exemple}
    	    
    	    \begin{rem}
    	        It should be emphasized that the present definitions of $\C^k$ functions and vector fields are weaker than the standard ones on a Riemannian manifold. This difference arises from the fact that equipping the tangent bundle $\T \P$ with a topology suitably compatible with the Riemannian metric $\langle \bullet, \bullet \rangle$ is a delicate problem.
    	    \end{rem}

    \vspace{10 pt}

    \noindent \textbf{Second fundamental form on G.} As mentioned before, $G$ can be seen as a Riemannian submanifold of $\mathscr{D}$, we denote by $\mathcal{L}^G$ the Lie derivative on $G$. Let $\bm Y_1,  \bm Y_2$ be vertical vector fields on $\mathscr{D}$. Then, using orthogonal decomposition, we have
    
    $$
    \bm \nabla_{\bm Y_1} \bm Y_2(g) = P_{\H_g}(\bm \nabla_{\bm Y_1} \bm Y_2(g)) + P_{\V_g}( \bm \nabla_{\bm Y_1} \bm Y_2(g)), \quad \forall g \in G.
    $$
    We define the map $(\bm Y_1, \bm Y_2) \in (\Gamma^1_S(TG))^2 \mapsto \mathrm{I\!I}^G(\bm Y_1,\bm Y_2) = P_{\H_g}(\nabla_{\bm Y_1} \bm Y_2) \in \Gamma^0_S(\H\mathscr{D})$ to be the \textit{second fundamental form on $G$} (as a submanifold of $\mathscr{D}$). Notice that for all $\varphi \in \mathscr{D}$, the set $\varphi \cdot G$ is the fiber above $p(\varphi)$. We define in a similar way the second fundamental form on any fiber, namely $\mathrm{I\!I}^{\varphi \cdot G}(\bm Y_1,\bm Y_2)$. 
    
    We also define the induced \textit{Levi-Civita connection} on $G$ as $\bm \nabla^G_{\bm Y_1} \bm Y_2 = P_{\V_g}( \bm \nabla_{\bm Y_1} \bm Y_2(g))$. It can be easily checked that this connection is actually the Levi-Civita connection on $G$ for the Riemannian metric on $\mathscr{D}$ restricted to $G$. In particular $\bm \nabla^G$ is torsion free, compatible with the metric and the Koszul's formula holds. The \textit{Hessian} of $\bm F \in \C^\infty(G)$ in $\varphi \in G$ is defined as
    
    $$
    \bm \he^G \bm F ( \bm Y_1(\varphi),\bm Y_2(\varphi)) = \mathcal{L}^G_{\bm Y_1} \mathcal{L}^G_{\bm Y_2} \bm F(\varphi) - \bm \nabla^G_{\bm Y_1}\bm Y_2 \bm F(\varphi),
    $$
    for $\bm Y_1, \bm Y_2$ vector fields on $G$. 
    \begin{prop}
    The second fundamental form $\mathrm{I\!I}^{\varphi \cdot G}$ is a symmetric tensor on $TG$ for all $\varphi \in \mathscr{D}$.
    \end{prop}
    
    \begin{proof}
    We show the result for $\varphi = \mathrm{id}$, the proof is the same for any fiber $\varphi \cdot G$. Let $\bm Y_1, \bm Y_2$ be vertical vector fields on $\mathscr{D}$. The fact that $\mathrm{I\!I}^G(\bm Y_1, \bm Y_2) = \mathrm{I\!I}^G(\bm Y_2, \bm Y_1)$ is equivalent to the fact that $[\bm Y_1, \bm Y_2]$ is vertical. Let $g \in G$ and $f \in \C^\infty(M)$, we have
    
    \begin{align*}
    \left \langle [\bm Y_1, \bm Y_2](g), \nabla f \circ g \right \rangle_g &= \left \langle \bm \nabla_{\bm Y_1} \bm Y_2(g), \nabla f \circ g \right \rangle_g - \left \langle \bm \nabla_{\bm Y_2} \bm Y_1(g), \nabla f \circ g \right \rangle_g \\
    & =  \mathcal{L}^D_{\bm Y_1(g)} \langle \bm Y_2(g), \nabla f \circ g \rangle_g - \mathcal{L}^D_{\bm Y_2(g)} \langle \bm Y_1(g), \nabla f \circ g \rangle_g  \\
    &- \left \langle  \bm Y_2(g), \bm \nabla_{\bm Y_1} \nabla f \circ g \right \rangle_g + \left \langle  \bm Y_1(g), \bm \nabla_{\bm Y_2} \nabla f \circ g \right \rangle_g.
    \end{align*}
    The first two terms in the last equality are $0$ since $Y_i(g) \perp \nabla f \circ g$ for $i=1,2$. Using the fact that
    
    $$
    \left \langle  \bm Y_2(g), \bm \nabla_{\bm Y_1} \nabla f \circ g \right \rangle_g = \left \langle  \bm Y_1(g), \bm \nabla_{\bm Y_2} \nabla f \circ g \right \rangle_g = \bm \he F_f(g)(Y_1,Y_2)
    $$
    and the fact that $\bm \he F_f(g)$ is a symmetric tensor we get
    
    $$
    \left \langle [\bm Y_1, \bm Y_2](g), \nabla f \circ g \right \rangle_g = 0,
    $$
    and thus, $[\bm Y_1, \bm Y_2]$ is vertical. Since the second fundamental form is tensorial in the first coordinate and symmetric, it is tensorial in both components.
    
    \end{proof} 
    
    \noindent \textbf{$G$-Right invariant vector fields.} A vector field $\bm A \in \Gamma^0(T\mathscr{D}^s)$ is said to be \textit{$G^s$-right invariant} if for all $\varphi \in \mathscr{D}^s$ and all $g \in G^s$, 
    
    \begin{equation*}
    \bm T R_g(\bm A(\varphi)) = \bm A(\varphi \cdot g).
    \end{equation*}
    A vector field $\bm A \in \Gamma_S^0(T\mathscr{D})$ is said to be $G$-right invariant if its extensions $\bm A^s \in \Gamma^0_S(T\mathscr{D})$ are $G^s$-right invariant vector fields.
    
    \begin{rem}\label{G-invariance des lifts horizontaux}
        Let $\bar{Z} \in \Gamma^k_S(T\P)$. Then, $\bm Z = \mathfrak{h}(\bar{Z}) \in \Gamma^k_S(T\mathscr{D})$ is a $G$-right invariant vector field on $\mathscr{D}$.
    \end{rem}

    \begin{prop} \label{Invariance à droite connexion levi-civita difféos}
	\begin{enumerate}
	\item Let $\bm A, \bm B \in \Gamma_S^1(T\mathscr{D})$, then
        
        \begin{equation} \label{G invariance connexion levi civita}
            \bm T R_g(\bm \nabla_{\bm B} \bm A) = \bm \nabla_{\bm T R_g( \bm B)}\bm T R_g(\bm A), \quad \forall g \in G.
        \end{equation}
		In particular, if $\bm A, \bm B \in \Gamma(T\mathscr{D})$ are $G$-right invariant vector fields, $\bm \nabla_{\bm A}\bm B$ is a $G$-right invariant vector field. Moreover, as a straightforward consequence of \eqref{G invariance connexion levi civita}, the following holds:
        
        $$
        \bm \nabla \bm T R_g = 0.
        $$
	\item A vector field $\bm A$ on $\mathscr{D}$ is $G$-right invariant vector field if and only if $P_\H(\bm A)$ and $P_\V(\bm A)$ are $G$-right invariant vector fields.
	\item Let $\bm A \in \Gamma^1_S(T\mathscr{D})$ be a right invariant vector field for the action of $G$. Then we have
    
    $$
    \bm \nabla_{Y\circ \varphi} \bm A (\varphi) = \left( \nabla_{Y} A(p(\varphi)) \right) \circ \varphi, \quad \forall Y \circ \varphi \in \V_\varphi\mathscr{D}.
    $$
    
    \item Assume that $\bm Y_1, \bm Y_2$ are $G$-right invariant vertical vector fields, then, $\mathrm{I\!I}^G(\bm Y_1, \bm Y_2)$ is a $G$-right invariant vector field.
\end{enumerate}

    \end{prop}
    
    \begin{proof}
    \begin{enumerate}
    \item This result can be proved by using the same arguments as in Lemma \ref{Invariance à droite dérivée covariante dim finie}, namely, by using Koszul's formula and the fact $R_g$ is an isometry from $\mathscr{D}$ to $\mathscr{D}$.
    \item This is a direct consequence of Proposition \ref{P Propriétés invariances à gauche et à droite des fibrés} together with the uniqueness of the decomposition in $\H \mathscr{D} \oplus \V \mathscr{D}$.
    \item Let $(\varphi_t)_{t \geq 0}$ be such that $\partial_t|_{t=0} \varphi_t = Y \circ \varphi$. Since $\mathrm{div}_{p(\varphi)}(Y) =0$, we have $p(\varphi_t) = p(\varphi)$. Then, by Proposition \ref{Prop 3.7}, we have
    
    $$
    \bm \nabla_{Y \circ \varphi} \bm A (\varphi)(x) = \left.  \frac{\d}{\d t} \right|_{t=0} \tau_{t,0} A(p(\varphi))(\varphi_t)(x) = \left. \frac{\d}{\d t} \right|_{t=0} \sslash_{t,0} A(p(\varphi))(\varphi_t(x)) = \left( \nabla_Y A(p(\varphi) \right) \circ \varphi(x).
    $$
    \item This is a direct consequence of the first two points together.
    \end{enumerate}

    \end{proof}

\section{Infinite-dimensional stochastic calculus} \label{Section 4}

    \subsection{Stochastic differential equations on $\mathscr{D}$} \label{Subsection 4.1}
    
    Let $\bm A_0, \bm A_1, \dots, \bm A_N \in \Gamma^k_S(\mathscr{D})$ for $k \geq 0$. A $\mathscr{D}^s$-valued stochastic process $(\varphi_t)_{t \geq 0}$ adapted to $(\mathcal{F}_t)_{t \geq 0}$ is a solution to the SDE 

\begin{equation} \label{E17}
\dd^{\bm \nabla}\varphi_t = \SN \bm A_i(\varphi_t) \d W^i_t + \bm A_0(\varphi_t) \d t, \quad \varphi_0 = \varphi\in \mathscr{D}
\end{equation}
if for all $\bm F \in \mathcal{C}^\infty(\mathscr{D}^s)$, $(\varphi_t)_{t \geq 0}$ satisfies

$$
\bm F(\varphi_t) - \bm F(\varphi) = \SN \int_0^t \mathcal{L}^{\mathscr{D}^s}_{\bm A_i} \bm F(\varphi_s) \d W^i_s + \int_0^t \mathcal{L}^{\mathscr{D}^s}_{\bm A_0} \bm F(\varphi_s) \d s + \frac{1}{2} \SN \int_0^t \bm \he \bm F(\bm A_i, \bm A_i) \d s.
$$
In the same way, we say that $(\varphi_t)_{t \geq 0}$ is a solution to the Stratonovich's SDE in $\mathscr{D}^s$

\begin{equation} \label{EDS strato D}
\circ \dd\varphi_t = \SN \bm A_i(\varphi_t) \circ \d W^i_t + \bm A_0(\varphi_t) \d t, \quad \varphi_0 = \varphi
\end{equation}
if for all $\bm F \in \mathcal{C}^\infty(\mathscr{D}^s)$, $(\varphi_t)_{t \geq 0}$ satisfies

$$
\bm F(\varphi_t) - \bm F(\varphi) =\SN \int_0^t \mathcal{L}^{\mathscr{D}^s}_{\bm A_i} \bm F(\varphi_s) \d W^i_s + \int_0^t \mathcal{L}^\mathscr{D}_{\bm A_0} \bm F(\varphi_s) \d s + \frac{1}{2} \SN \int_0^t \mathcal{L}^{\mathscr{D}^s}_{\bm A_i} \mathcal{L}^\mathscr{D}_{\bm A_i} \bm F(\varphi_s) \d s.
$$
A process $(\varphi_t)_{t \geq 0}$ is said to be a solution in $\mathscr{D}$ of \eqref{E17} or of \eqref{EDS strato D} if there exists $s_0 \in \mathbb{N}$ such that it is a solution in all $\mathscr{D}^s$ for $s> s_0$.

A consequence of the definition is that if $(\varphi_t)_{t \geq 0}$ is solution to \eqref{EDS strato D}, then for any $\Psi \in \C^2_S(\mathscr{D},\mathscr{D})$, $(\Psi(\varphi_t))_{t \geq 0}$ is solution to:

\begin{equation*}
    \circ \dd \Psi(\varphi_t) = \SN (\Psi_*(\bm A_i))(\Psi(\varphi_t)) \circ \d W^i_t + (\Psi_*(\bm A_0))(\Psi(\varphi_t)) \d t, \quad \Psi(\varphi_0) = \Psi(\varphi).
\end{equation*}

\begin{rem}
    Note that this notion of solution is meaningful since, by Proposition \ref{P Evaluation map lisse}, the evaluation map $E_{x_0} : \varphi \mapsto \varphi(x_0)$ is smooth for all $x_0 \in M$. As a consequence, if $\bm A_i(\psi) = A_i\circ \psi$ for all $\psi \in \mathscr{D}$ and for all $i=0,1,\dots,N$, then:
    
    $$
    \circ \d \varphi_t(x_0) = \SN A_i(\varphi_t(x_0)) \circ \d W^i_t + A_0(\varphi_t(x_0)) \d t.
    $$
    This is in fact the approach used by Elworthy to show that a stochastic differential system on a compact manifold with smooth driving vector fields admits a unique solution flow, and that this flow is indeed a stochastic flow of diffeomorphisms, see \cite[Chapter VIII, Corollary 1C.2]{Elworthy1982}.
\end{rem} 

\begin{prop} \label{Equivalence Ito Strato difféos}
    Let $\bm A_0, \bm A_1, \dots, \bm A_N \in \Gamma^2_S(T\mathscr{D})$. The Stratonovich's SDE

$$
\circ \dd\varphi_t = \SN \bm A_i(\varphi_t) \circ \d W^i_t + \bm A_0(\varphi_t) \d t, \quad \varphi_0 = \varphi
$$
is equivalent to the following Itô's one:

$$
\dd^{\bm \nabla} \varphi_t = \SN \bm A_i(\varphi_t) \d W^i_t + \bm A_0(\varphi_t)\d t + \frac{1}{2}\SN \bm \nabla_{\bm A_i}\bm A_i(\varphi_t) \d t, \quad \varphi_0 = \varphi.
$$
\end{prop}

\begin{proof}
The proof is a direct consequence of the definition of the Hessian $\bm \he$.
\end{proof}

\begin{defi}
	A continuous $\mathscr{D}$-valued process $(\varphi_t)_{t \geq 0}$ is said to be a \textit{semimartingale} if for all $\bm F \in \C^\infty(\mathscr{D})$, the process $(\bm F(\varphi_t))_{t \geq 0}$ is a real semimartingale.
\end{defi}

\begin{thm}[Elworthy, \cite{Elworthy1982}] \label{Existence unicité Difféos}
	Let $\bm A_0, \bm A_1, \dots, \bm A_N \in \Gamma^1_S(T\mathscr{D})$. The Itô's SDE
	
	$$
	\dd^{\bm \nabla} \Phi_t(\varphi) = \SN \bm A_i(\Phi_t(\varphi)) \d W^i_t + \bm A_0(\Phi_t(\varphi)) \d t, \quad \Phi_0(\varphi) = \varphi,
	$$
	admits a unique solution in $\mathscr{D}$.

\end{thm}

\begin{proof}
    It is sufficient to prove the result on $\mathscr{D}^s$ for all $s$ large enough since $\mathscr{D} = \cap_{s > m/2} \mathscr{D}^s$. This result can be found in \cite[Chapter VII, Theorem 2E and Corollary 6.1]{Elworthy1982}. More precisely, the existence and uniqueness is given by the first result and the completeness of the solution is given by the second one, since, as detailed in \cite[Chapter VIII, (C)]{Elworthy1982}, $\mathscr{D}^s$ admits a uniform cover for all $s > m/2$.
\end{proof}

\begin{rem}[SDEs on $G$] \label{Equivalence Ito Strato sur G}
    The notion of SDE on $G$ can be formalized exactly in the same way as on $\mathscr{D}$ by replacing $\mathcal{L}^\mathscr{D}$ by $\mathcal{L}^G$, $\bm \nabla$ by $\bm \nabla^G$ and $\bm \he$ by $\bm \he^G$. In particular, the following $G$-valued Stratonovich's SDE
    
    $$
\circ \dd g_t = \SN \bm Y_i(g_t) \circ \d W^i_t + \bm Y_0(g_t) \d t, \quad g_0 = g
$$
is equivalent to the following Itô's one:

$$
\dd^{\bm \nabla^G} g_t = \SN \bm Y_i(g_t) \d W^i_t + \bm Y_0(g_t)\d t + \frac{1}{2}\SN \bm \nabla^G_{\bm Y_i}\bm Y_i(g_t) \d t, \quad g_0 = g.
$$
\end{rem}

\begin{defi}
    Let $(\varphi_t)_{t \geq 0}$ be a semimartingale on $\mathscr{D}$ where $\varphi_0 = \varphi$ and $A \circ \varphi \in T_\varphi \mathscr{D}$. The stochastic parallel transport along $(\varphi_t)_{t \geq 0}$ is defined to be:
    
    $$
    \tau_{0,t} : A \circ \varphi \mapsto \sslash_{0,t}(A \circ \varphi),
    $$
    where for all $x \in M$, $\sslash_{0,t}(x)$ is the stochastic parallel transport along $(\varphi_s(x))_{0 \leq s \leq t}$. 
\end{defi}

As in the deterministic case, one can check that $(\tau_{0,t})_{t \geq 0}$ is an isometric isomorphism between $T_{\varphi} \mathscr{D}^s$ and $T_{\varphi_t}\mathscr{D}^s$ for all $s>m/2$ and, consequently, an isometric isomorphism between $T_\varphi \mathscr{D}$ and $T_{\varphi_t} \mathscr{D}$.

	\begin{defi}
	A solution to a Stratonovich's SDE of the following form:
	
	$$
	\circ \dd \Phi_t = \SN \bm A_i(\Phi_t) \circ \d W^i_t + \bm A_0(\Phi_t)dt, \quad \Phi_0(\varphi) \in \mathscr{D},
	$$
	is said to be an \textit{equivariant diffusion} if for all $g \in G$, $\Phi_t(\varphi \cdot g) = \Phi_t(\varphi) \cdot g$ for all $t \geq 0$.
	\end{defi}

		\begin{defi}
    A $T \mathscr{D}$-valued stochastic process $(\bm U_t)_{t \geq 0}$ along a semimartingale $(\varphi_t)_{t \geq 0}$ is said to be a \textit{semimartingale} if for all $s$, $(\tau_{0,t}^{-1} \bm U_t)_{t \geq 0}$ is a semimartingale on $T_{\varphi_0}\mathscr{D}^s$ where $(\tau_{0,t})_{t \geq 0}$ denote the stochastic parallel transport along $(\varphi_t)_{t \geq 0}$.
\end{defi}

\begin{prop} \label{Prop dérivée covariante champ vecteur lisse D}
    Let $\bm B \in \Gamma^1_S(T\mathscr{D})$ and $(\varphi_t)_{t \geq 0}$ be solution to \eqref{EDS strato D}. Then, $(\bm B(\varphi_t))_{t \geq 0}$ is a semimartingale and:
    
    $$
    \circ \d \tau_{0,t}^{-1}\bm B(\varphi_t) = \SN  \tau_{0,t}^{-1}\left(\bm \nabla_{\bm A_i} \bm B \right)(\varphi_t) \circ \d W^i_t +  \tau_{0,t}^{-1}\left( \bm \nabla_{\bm A_0} \bm B \right)(\varphi_t) \d t.
    $$
    
\end{prop}

\begin{proof}
The proof is essentially the same as the one of Proposition \ref{Prop 3.7}, with the differential calculus replaced by its Stratonovich counterpart.
\end{proof}

\subsection{Stochastic differential equations on $\P$} \label{Subsection 4.2}

We say that a stochastic process $(\mu_t)_{t \geq 0}$ adapted to $(\mathcal{F}_t)_{t \geq 0}$ is a solution to the following Itô's SDE on $\P$ 

\begin{equation} \label{EDSIW}
\dw^{\grad}\mu_t = \SN \bar{Z}_i(\mu_t) \d W^i_t + \bar{Z}_0(\mu_t) \d t, \quad \mu_0 = \mu
\end{equation}
if for all $f \in \mathcal{C}^\infty(M)$, $(\mu_t)_t$ satisfies

$$
F_f(\mu_t) -F_f(\mu) = \SN \int_0^t \L_{\bar{Z}_i} F_f(\mu_s) \d W^i_s + \int_0^t \L_{\bar{Z}_0} F_f(\mu_s) \d s + \frac{1}{2} \SN \int_0^t \overline{\he }F_f(\mu_s)(\bar{Z}_i, \bar{Z}_i) \d s.
$$
In the same way, we say that $(\mu_t)_{t \geq 0}$ is a solution to the Stratonovich's SDE on $\P$

\begin{equation*}
\circ \dw\mu_t = \SN \bar{Z}_i(\mu_t) \circ \d W^i_t + \bar{Z}_0(\mu_t) \d t, \quad \mu_0 = \mu
\end{equation*}
if for all $f \in \mathcal{C}^\infty(M)$, $(\mu_t)_{t \geq 0}$ satisfies

$$
F_f(\mu_t) -F_f(\mu) =\SN \int_0^t \L_{\bar{Z}_i} F_f(\mu_s) \d W^i_s + \int_0^t \L_{\bar{Z}_0} F_f(\mu_s) \d s + \frac{1}{2} \SN \int_0^t \L_{\bar{Z}_i}^2 F_f(\mu_s) \d s.
$$

\begin{prop} \label{Equivalence Ito Strato Wasserstein}
    Let $\bar{Z}_0, \bar{Z}_1, \dots, \bar{Z}_N \in \Gamma^2(T\P)$. The Stratonovich's SDE
    
    $$
    \circ \dw\mu_t = \SN \bar{Z}_i(\mu_t) \circ \d W^i_t + \bar{Z}_0(\mu_t) \d t, \quad \mu_0 = \mu
$$
is equivalent to the following Itô's one:

$$
\dw^{\grad} \mu_t = \SN \bar{Z}_i(\mu_t) \d W^i_t + \bar{Z}_0(\mu_t) \d t + \frac{1}{2} \SN \grad_{\bar{Z}_i}\bar{Z}_i(\mu_t) \d t, \quad \mu_0 = \mu.
$$
\end{prop}

\begin{proof}
This is, as in the diffeomorphisms case, a direct consequence of the definition of the Hessian $\bar{\he}$.
\end{proof}

\begin{defi}
	A continuous $\P$-valued process $(\mu_t)_{t \geq 0}$ is said to be a \textit{semimartingale} if for all $f \in \C^\infty(M)$, the process $(F_f(\mu_t))_{t \geq 0}$ is a real semimartingale.
\end{defi}

In particular, solutions to SDEs on Wasserstein space are semimartingale on $\P$. We now state the main theorem of this section, whose proof will be given in Subsection \ref{Subsection 4.4}.

\begin{thm}\label{Existence Unicité}
	Let $\bar{Z}_0, \bar{Z}_1, \dots, \bar{Z}_N \in \Gamma^1_S(T\P)$ defined by $\bar{Z}_i = \nabla \phi_i(\cdot)$ for $i=0, 1, \dots, N$. The SDE
	
	\begin{equation}\label{Brave EDS sur Wasserstein}
	\dw^{\grad} M_t(\mu) = \SN \bar{Z}_i(M_t(\mu)) \d W^i_t + \bar{Z}_0(M_t(\mu)) \d t, \quad M_0(\mu) = \mu,
	\end{equation}
	admits a solution $(M_t(\mu))_{t \geq 0}$. Moreover, if $\mu \in \P_\infty$, for all $t \geq 0$, $M_t(\mu) \in \P_\infty$.
\end{thm}

\begin{rem}[SPDE aspect of $\P_\infty$-valued SDEs]
	Let $\mu \in \P_\infty$ and $(\mu_t)_{t \geq 0} = (M_t(\mu))_{t \geq 0}$ be solution to
	
	$$
\circ \d\mu_t = \SN \bar{Z}_i(\mu_t) \circ \d W^i_t + \bar{Z}_0(\mu_t) \d t, \quad \mu_0 = \mu.
$$ 
Then, the family of density $(\rho_t)_{t \geq 0}$ of $(\mu_t)_{t \geq 0}$ relative to $\vol$ is solution to the following SPDE :
	
	$$
	\d\rho_t = - \SN \mathrm{div}(\rho_t Z_i(\mu_t)) \d W^i_t - \mathrm{div}(\rho_t Z_0(\mu_t))\d t + \frac{1}{2} \SN \mathrm{div}\left(\mathrm{div} (\rho_t Z_i(\mu_t)) Z_i(\mu_t) \right)\d t,
	$$
	for a proof of a similar result see \cite[Proposition $3.3$]{ding2023stochasticdifferentialequationsstochastic}.
\end{rem} 

Let $(\mu_t)_{t \geq 0}$ be a solution to

\begin{equation} \label{EDSW}
\circ \dw\mu_t = \SN \bar{Z}_i(\mu_t) \circ \d W^i_t + \bar{Z}_0(\mu_t) \d t, \quad \mu_0 = \mu \in \P_\infty,
\end{equation}
with $\bar{Z}_i \in \Gamma^2_S(T\P)$ for $i=0,1,\dots,N$. We emphasize that, from this point on and throughout the remainder of the subsection, the initial condition will always be assumed to lie in $\P_\infty$.

\begin{defi} \label{Defi relevé horizontal}
    Let $(\mu_t)_{t \geq 0}$ be solution to \eqref{EDSW}. We define the \textit{horizontal lift} of $(\mu_t)_{t \geq 0}$ in $\varphi \in \mathscr{D}_\mu$ to be the unique solution to the following $\mathscr{D}$-valued SDE:
    
    \begin{equation} \label{E lift horizontal dim infinie}
    \circ \dd \Phi_t(\varphi) = \SN \bm Z_i(\Phi_t(\varphi)) \circ \d W^i_t + \bm Z_0(\Phi_t(\varphi)) \d t, \quad \Phi_0(\varphi) = \varphi,
    \end{equation}
    where $\bm Z_i(\psi) = \mathfrak{h}_\psi(\bar{Z}_i(\mu))$ is the horizontal lift of $\bar{Z}_i$ for all $i=0,1, \dots, N$.
\end{defi}

\begin{prop} \label{Prop equivariance horizontal lift}
    The horizontal lift of $(\mu_t)_{t \geq 0}$ is an equivariant diffusion.
\end{prop}

\begin{proof}
    Since $\bm Z_i$ is $G$-right invariant for all $i=0,1,\dots,N$, we have:
    
    $$
    \circ \dd (R_g(\Phi_t(\varphi))) = \bm T R_g( \circ \dd \Phi_t(\varphi)) = \SN \bm Z_i(\Phi_t(\varphi) \cdot g) \circ \d W^i_t + \bm Z_0(\Phi_t(\varphi) \cdot g) \d t.
    $$
    The conclusion follows by uniqueness of the solution to \eqref{E lift horizontal dim infinie}.
\end{proof}

\begin{defi} \label{Def Transport parallele sto sur P}
Let $(\mu_t)_{t \geq 0}$ be a diffusion on $\P_\infty$ solution \eqref{EDSW} and let $(\Phi_t)_{t \geq 0}$ be the horizontal lift of $(\mu_t)_{t \geq 0}$. A $T \P$-valued stochastic process $(\bar{U}_t)_{t \geq 0}$ along $(\mu_t)_{t \geq 0}$ is said to be a \textit{semimartingale} if for all $\varphi \in \mathscr{D}_\mu$, $(\mathfrak{h}_{\Phi_t(\varphi)}(\bar{U}_t))_{t \geq 0}$ is a semimartingale on $T \mathscr{D}$.
\end{defi}
    
\begin{prop}
    Let $\bar{A} \in \Gamma^2_S(T\P)$ and $(\mu_t)_{t \geq 0}$ be solution to \eqref{EDSW}. Then, $(\bar{A}(\mu_t))_{t \geq 0}$ is a semimartingale.
    
\end{prop}

\begin{proof}
    This is a straightforward consequence of Proposition \ref{Prop dérivée covariante champ vecteur lisse D}.
\end{proof}

\subsection{Preliminaries}

To prove Theorem \ref{Existence Unicité}, we use the isometric embedding theorem of Nash in order to reduce ourselves to the case of a Riemannian submanifold of the Euclidean space.

\begin{thm}[Isometric embedding, Nash \cite{Nash1956TheIP}]
        Let $M$ be a smooth Riemannian manifold. Then, there exists $d \in \mathbb{N}$ and an isometric embedding $f : M \rightarrow \mathbb{R}^d$. For $x \in M$ and $u,v \in T_xM$ we have:
        
        \begin{equation*}
            \langle u,v \rangle_x = \langle D_xf(u), D_xf(v) \rangle^{\mathbb{R}^d},
        \end{equation*}
        where $\langle \cdot, \cdot \rangle^{\mathbb{R}^d}$ denotes the usual scalar product on $\mathbb{R}^d$.
        
    \end{thm}
    
        In the following, we consider $M$ to be a closed Riemannian sub-manifold $\mathbb{R}^d$ with dimension $\mathrm{dim}(M) = m$ for the sake of simplicity. Let $|\cdot|$ denotes the Euclidean norm on $\mathbb{R}^d$. This subsection is mainly devoted to the proof of the fact that we can extend the vector fields $Z_i(\mu)(\cdot)$ on $\mathbb{R}^d$ in a way that preserves the regularity in the parameter $\mu$.
    
        \begin{prop}[Uniform extension lemma for vector fields] \label{Uniform extension lemma for vector fields}
    		Let $A \in \C^{\infty,k}_w(M \times \P, TM)$ such that for all $\mu \in \P$, $A(\cdot,\mu) \in \Gamma(TM)$. Then, there exists an extension $ \tilde{A} \in \C^{\infty,k}_w(\mathbb{R}^d \times \P, \mathbb{R}^d)$ with compact support.
    \end{prop}
    
    \begin{proof}
     Let $\varepsilon> 0$ and let $M_\varepsilon$ be a tubular neighborhood of $M$ such that $P_\varepsilon$ is smooth (see Proposition \ref{Existence voisinage tubulaire}) . Moreover, let $\chi \in \C^\infty(\mathbb{R})$ be a bump function such that $\chi(0) = 1$ with support $[-\varepsilon/2, \varepsilon/2]$. Then, the vector field defined as $\tilde{A}(x,\mu) = \chi(d(x,M)) A(P_\varepsilon(x),\mu)$ when $x \in M_\varepsilon$ and $0$ elsewhere satisfies the conditions of the proposition.
    \end{proof}
    
     In the following, let $\varepsilon > 0$ and let $\tilde{Z}_i$ denotes a $M_\varepsilon$-supported $\C^{\infty,1}$ extension of $Z_i$ for all $i=0,1,\dots, N$ with $M_\varepsilon$ a neighborhood of $M$ as in Proposition \ref{Existence voisinage tubulaire}.  We emphasize the fact that even if we extend the vector fields on $\mathbb{R}^d$, $\mu$ is still in $\P = \mathbb{P}(M)$.

        \begin{prop}\label{ConstanteLipschitz}
        For $i=0,1 \dots, N$ there exists $K_i > 0$ such that for all $(x,\mu), (y,\nu) \in \mathbb{R}^d \times \mathscr{P}$:
        
        \begin{equation*}
            |\tilde{Z}_i(x,\mu) - \tilde{Z}_i(y,\nu)|^2 \leq K_i(|x-y|^2 + W_2^2(\mu,\nu)).
        \end{equation*}
    \end{prop}

    \begin{proof}
    Using the fact that for all $i=0, \dots, N$, $\tilde{Z}_i \in \C^{\infty,1}(\mathbb{R}^d \times \P)$ and the same arguments as in Corollary \ref{PropLipschitzGeneralVectorField} we get that $\tilde{Z}_i(x,\cdot)$ is Lipschitz for all $x \in \tilde{M}$ with Lipschitz constant $L_i$ independent of $x$.
        
        \begin{equation*}
            |\tilde{Z}_i(x,\mu) - \tilde{Z}_i(y,\nu)| \leq |\tilde{Z}_i(x,\mu) - \tilde{Z}_i(y,\mu)| + | \tilde{Z}_i(y,\mu) - \tilde{Z}_i(y,\nu)| \leq C_i(|x-y| + W_2(\mu,\nu)),
        \end{equation*}
        where $C_i = \max \left( \sup \{ |D \tilde{Z}_i(z,m)| \, : \, z \in \bar{M}_\varepsilon, \, m \in \mathscr{P} \}, L_i \right ) $. Note that this quantity is well defined by compactness of both $\bar{M}_\varepsilon$ and $\P$ and by continuity of $D \tilde{Z}_i$. 
        
        By using the basic inequality $ (a+b)^2 \leq 2(a^2 + b^2)$ we get
        
        \begin{equation*}
            C_i^2(|x-y| + W_2(\mu,\nu))^2 \leq 2C_i^2 (|x-y|^2 + W_2^2(\mu,\nu)).
        \end{equation*}
    \end{proof}
    For a Riemannian submanifold $M$ of $\mathbb{R}^d$, let $\tilde{A}, \tilde{B}$ be smooth extensions of vector fields $A,B \in \Gamma(TM)$. The \textit{second fundamental form} in $\tilde{A}, \tilde{B}$ is 
    
    \begin{equation*}
        \mathrm{I\!I}(\tilde{A},\tilde{B}) = P_{TM^\perp}\left( \nabla^{\mathbb{R}^d}_{\tilde{A}}\tilde{B} \right).
    \end{equation*}
    This operator is in particular a symmetric tensor and its value on $\Gamma(TM)$ does not depend on the extensions chosen.
We define $\tilde{Z}_0'$ to be:
    
    \begin{equation} \label{Définition tilde Z0}
        \tilde{Z}_0'(\mu) = \tilde{Z}_0(\mu) - \frac{1}{2} \sum_{i=1}^N \mathrm{I\!I}(\tilde{Z}_i(\mu),\tilde{Z}_i(\mu)).
    \end{equation}
    In particular, $ \mathrm{I\!I}(\tilde{Z}_i,\tilde{Z}_i)$ is of class $\C^{\infty,1}$ by definition and Proposition \ref{Uniform extension lemma for vector fields}.
    Until the end of the section we will set $K_0, K > 0$ such that:
    
    \begin{align} 
        &| \tilde{Z}_0'(x,\mu) - \tilde{Z}_0'(y,\nu)|^2  \leq K_0 (|x-y|^2 + W_2^2(\mu,\nu)) \label{C1}, \\ 
        & 2 \left| \langle \tilde{Z}_0'(x,\mu) - \tilde{Z}_0'(y,\nu), x-y \rangle \right| + \underset{i=1}{\overset{N}{\sum}} | \tilde{Z}_i(x,\mu) - \tilde{Z}_i(y,\nu)|^2 \leq K(|x-y|^2 + W_2^2(\mu,\nu)). \label{C2}
    \end{align}
    
    \begin{rem}
        To get the second inequality, notice that
        
        \begin{equation*}
            2 \left| \langle \tilde{Z}_0'(x,\mu) - \tilde{Z}_0'(y,\nu), x-y \rangle \right| \leq |\tilde{Z}_0'(x,\mu) - \tilde{Z}_0'(y,\nu)|^2 + |x-y|^2
        \end{equation*}
        and we conclude by applying \eqref{C1} and Proposition \ref{ConstanteLipschitz}.
    \end{rem}

\subsection{Proof of Theorem \ref{Existence Unicité}} \label{Subsection 4.4}

	To get a solution to \eqref{Brave EDS sur Wasserstein}, it is known (see for example \cite{carmona2018probabilistic2}, \cite{wang2020}) that it is enough to solve a conditionnal Mckean-Vlasov SDE on $M$. In our case, this is the following:
	
	\begin{equation} \label{Conditionnal SDE}
	\d^\nabla X_t(x) = \SN Z_i( X_t(x), M_t(\mu)) \d W^i_t + Z_0( X_t(x), M_t(\mu)  ) \d t, \quad X_0(x) = x, \quad M_t(\mu) = (X_t)_* \mu.
	\end{equation}

	We present the definition of the solutions of such conditional McKean-Vlasov SDEs as in \cite{wang2020}.
	
	\begin{defi}
	A family of adapted processes $ \{ (t,x) \in \mathbb{R}_+ \times M \rightarrow X_t(x) \in M \} $ is called a solution of \eqref{Conditionnal SDE} if the following conditions hold a.s:
	
	\begin{enumerate}
	\item $X$ is continuous in $t$ and measurable in $x \in M$.
	\item The map $t \mapsto M_t(\mu) =(X_t)_* \mu$ is continuous. 
	\item For all $f \in \C^\infty(M)$,
	
	\begin{align*}
	f(X_t(x))-f(x) &= \SN \int_0^t Z_i(\cdot, M_s(\mu)) f(X_s(x)) \d W^i_s + \int_0^t Z_0(\cdot,M_s(\mu)) f(X_s(x)) \d s \\
	& + \frac{1}{2} \SN \int_0^t \he f(Z_i(X_s(x),M_s(\mu)),Z_i(X_s(x),M_s(\mu))) \d s.
	\end{align*}
	\end{enumerate}
	\end{defi}
	
		\begin{rem}[Itô-Stratonovich conversion for conditional McKean-Vlasov SDEs]
	    The Stratonovich's equivalent of Equation \eqref{Conditionnal SDE} can be obtained as follows. Let $(X_t)_{t \geq 0}$ be solution to \eqref{Conditionnal SDE}, then, for all $f \in \C^\infty(M)$, we have:
	    
	    \begin{align*}
	        f(X_t(x)) - f(x) &= \SN \int_0^t Z_i(\cdot,M_s(\mu)) f(X_s(x)) \d W^i_s + \int_0^t Z_i(\cdot, M_s(\mu))f(X_s(x)) \d s \\
	        & + \frac{1}{2} \SN \int_0^t \he f(Z_i(X_s(x),M_s(\mu)),Z_i(X_s(x),M_s(\mu))) \d s \\
	        & - \frac{1}{2} \SN \left[ Z_i(\cdot,M_\cdot(\mu)) f(X_\cdot(x)), W^i\right]_t.
	    \end{align*}
	    We have
	    
	    \begin{align} \label{E26}
	        Z_i(\cdot,M_t(\mu)) f(X_t(x)) &= \langle Z_i(X_t(x),M_t(\mu)), \nabla f(X_t(x)) \rangle \\
	        & =\int_0^t \circ \d \langle Z_i(X_s(x),M_s(\mu)), \nabla f(X_s(x)) \rangle \nonumber \\
	        & = \int_0^t \langle \nabla_{\circ \d X_s(x)} Z_i(X_s(x), M_s(\mu)), \nabla f(X_s(x)) \rangle + \langle Z_i(X_s(x),M_s(\mu)), \nabla_{\circ \d X_s(x)} \nabla f(X_s(x)) \rangle \nonumber \\
	        & + \int_0^t \langle \nabla \bar{\mathcal{L}}_{\circ \d M_s(\mu)} \phi_i(X_s(x),M_s(\mu)), \nabla f(X_s(x)) \rangle \nonumber.
	    \end{align}
	    Since the sum of the first two terms in the last equality of \eqref{E26} is $\int_0^t \mathcal{L}_{\circ \d X_s(x)} \mathcal{L}_{Z_i(X_s(x), M_s(\mu))} f(X_s(x))$, we obtain that
	    
	    $$
	    \left[ Z_i(\cdot,M_\cdot(\mu)) f(X_\cdot(x)), W^i\right]_t = \int_0^t \mathcal{L}_{Z_i(X_s(x),M_s(\mu))}^2f(X_s(x)) + \mathcal{L}_{\nabla \bar{\mathcal{L}}_{\bar{Z}_i(M_t(\mu))}\phi_i(X_s(x),M_s(\mu))} f(X_s(x)) \d s.
	    $$
	    Thus, by using the definition of $\he f$, Equation \eqref{Conditionnal SDE} is equivalent to
	    
	    \begin{align*}
	        \circ \d X_t(x) &= \SN Z_i( X_t(x), M_t(\mu)) \circ \d W^i_t + \left( Z_0( X_t(x), M_t(\mu)  ) - \frac{1}{2} \SN  \nabla_{Z_i(\cdot,M_s(\mu))}Z_i(X_s(x),M_s(\mu)) \right) \d t \\
	        & - \frac{1}{2} \SN \nabla \bar{\mathcal{L}}_{\bar{Z}_i(M_t(\mu))}\phi_i(X_t(x),M_t(\mu)) \d t , \quad X_0(x) = x, \quad M_t(\mu) = (X_t)_* \mu.
	    \end{align*}
	\end{rem}

        \begin{thm} \label{TheoremeExistenceUniciteMcKean}
    The conditional McKean-Vlasov equation
    
    \begin{equation} \label{SDE Mckean vlasov RN}
        \d^\nabla X_t(x) = \SN Z_i(M_t(\mu), X_t(x)) \d W^i_t + Z_0(M_t(\mu), X_t(x)) \d t, \quad X_0(x) = x, \quad M_t(\mu) = (X_t)_* \mu,
    \end{equation}
    admits a unique solution for initial conditions $x \in M$ and $\mu \in \P$, where $\mu$ can be a random variable on $\P$. Moreover, if $\mu \in \P_\infty$, the solution flow $(X_t)_{t \geq 0}$ is a stochastic flow of diffeomorphisms.
    \end{thm}
    The proof will be done below. Throughout the proof we will use the following results.
    
    \begin{prop} \label{Résultat régularité solutions}
    Let $M$ be a compact Riemannian submanifold of $\mathbb{R}^d$ and let $A_0,A_1, \dots A_n$ be random time-dependent vector fields being almost surely $\C^{0,\infty}(\mathbb{R}_+, M)$ on $M$, such that $A_i(t,\cdot)$ is $\mathcal{F}_t$-measurable for all $t \geq 0$ and all $i=0,1,\dots,N$. Let $\tilde{A}_0, \tilde{A}_1,  \dots, \tilde{A}_N$ be $\mathcal{C}^{0,\infty}(\mathbb{R}_+, \mathbb{R}^d)$ extensions on $\mathbb{R}^d$. There exists a unique solution to
    
    \begin{equation} \label{E30}
        \d X_t = \SN \tilde{A}_i(t,X_t(x)) \d W^i_t + \left( \tilde{A}_0(t,X_t(x)) - \frac{1}{2} \SN \mathrm{I\!I}(\tilde{A}_i(t,X_t(x)),\tilde{A}_i(t,X_t(x))) \right) \d t, \quad X_0(x) =x \in M.
    \end{equation}
    Moreover, $(X_t)_{t \geq 0}$ remains in $M$ and is the unique solution to
    
    \begin{equation} \label{E31}
        \d ^{\nabla} X_t(x) = \SN A_i(t,X_t(x)) \d W^i_t + A_0(t,X_t(x)) \d t.
    \end{equation}
    \end{prop}
    
    \begin{proof}
    	For the existence and uniqueness of the solution to \eqref{E30}, the reader is refered to \cite[Section $3.4$]{kunita1997stochastic}. The fact that $X_t(x)$ stays in $M$ for all $x \in M$ can be found, for example, in \cite[Section $1.2$]{hsu2002stochastic}. The fact that the solution of \eqref{E30} is actually the solution to \eqref{E31} is clear since the second fundamental form term is exactly what we need to remove in order to obtain the Hessian on $M$.
    \end{proof}
    
    \begin{prop} \label{Processus mesure lisse}
    		Let $(X_t)_{t \geq 0}$ be a stochastic flow of diffeomorphisms. For all $\mu \in \P_\infty$, the stochastic process defined by $(X_t)_* \mu$ lies in $\P_\infty$.
    \end{prop}
    
    \begin{proof}
    Note that, since for all $t \geq 0$, $X_t$ is a diffeomorphism, we have for all $f \in \C^\infty(M)$ :
    	
    	$$
    	\int_M f(x) \d\mu_t(x) = \int_M f(X_t(x)) \d\mu(x) = \int_M f(X_t(x)) \rho(x) \d\vol = \int_M f(x) \rho(X_t^{-1}(x)) |\text{Det} (T_x X_t^{-1})| \d\vol(x).
    	$$
    	Therefore, the assertion holds.
    \end{proof}
  
    We now prove Theorem \ref{TheoremeExistenceUniciteMcKean} by adapting the proof of \cite[Theorem $2.1$]{wang2020}.
    
    \begin{proof}[Proof of Theorem \ref{TheoremeExistenceUniciteMcKean}]
        \textbf{Existence:} Let $X^0(x) = x$ and $\mu^0 = \mu$. We define a sequence $(X_t^n(x),M_t^n(\mu))$ by induction. Let $n \in \mathbb{N}$ such that $(M_t^n(\mu))_{t \geq 0}$ is continuous in the time variable, we consider the following stochastic differential system (SDS) in $\mathbb{R}^d$:
        
        \begin{equation*}
            \d X^{n+1}_t(x) = \underset{i=1}{\overset{N}{\sum}}\tilde{Z}_i(X^{n+1}_t(x) ,M^n_t(\mu)  ) \d W_i^t + \tilde{Z}_0'( X^{n+1}_t(x), M^n_t(\mu)) \d t, \quad X_0^{n+1}(x)=x, \, M^n_t(\mu) = (X^n_t)_* \mu,
        \end{equation*}
        with $\tilde{Z}'_0$ defined by \eqref{Définition tilde Z0}. By Proposition \ref{Résultat régularité solutions} there exists a unique solution $(X^{n+1}_t)_{t \geq 0}$ to this SDS. Moreover, for an initial condition $x \in M$, the solution remains in $M$ and is solution to the following $M$-valued SDE:
        
        \begin{equation*} 
        \d^\nabla X^{n+1}_t(x) = \SN Z_i( X_s^{n+1}(x),M_s^n(\mu)) \d W^i_t + Z_0( X_s^{n+1}(x),M_s^n(\mu)) \d t.
        \end{equation*}
        In particular, we have:
        
        \begin{align}\label{Solution Induction}
            f(X^{n+1}_t(x))-f(x) &= \int_0^t Z_i(\cdot, M^n_s(\mu))f(X^{n+1}_s(x)) \d W^i_s +  \int_0^t Z_0(\cdot, M^n_s(\mu))f(X^{n+1}_s(x)) \d s \nonumber \\
    & + \frac{1}{2} \SN \int_0^t \he f(Z_i(X_s^{n+1}(x),M_s^n(\mu)),Z_i(X_s^{n+1}(x),M_s^n(\mu))) \d s, \quad \forall f \in \C^\infty(M).
        \end{align}
        On $\mathbb{R}^d$, it takes the following form:
        
        \begin{equation*} 
        X_t^{n+1}(x) = x + \int_0^t \SN \tilde{Z}_i(X_s^{n+1}(x), M_s^n(\mu)) \d W^i_s + \int_0^t \tilde{Z}_0'(X_s^{n+1}(x), M_s^n(\mu)) \d s.
        \end{equation*}
        Now, if we let $M_t^{n+1}(\mu)$ denotes $(X^{n+1}_t)_* \mu$, using Lemma \ref{LemmeTransportOptimal} we get
        
        \begin{equation*}
            \mathbb{E} \left[W_2^2(M_s^{n+1}(\mu), M_t^{n+1}(\mu)) \right] \leq \mathbb{E} \left[ \int_M |X^{n+1}_s(x) - X^{n+1}_t(x)|^2 \d\mu(x) \right].
        \end{equation*}
        Thus, when $s \to t$, the right-hand side goes to $0$ because of the (a.s.) time continuity of $(X^{n+1}_t)_t$. The proof of the following lemma can be found in Appendix \ref{Appendice preuve du lemme}.
    \begin{lem} \label{Lemme suite cauchy}
        There exists $t_0 > 0$ such that
        
        $$
        \lim_{n,l \to \infty} \mathbb{E} \left[\sup_{t \in [0,t_0]}d_M^2(X_t^{n}(x),X_t^l(x)) \right]^{1/2} = 0.
        $$
        Moreover, $t_0$ depends only on the vector fields $\tilde{Z}_i$ and may be chosen independently of the initial condition $(x,\mu)$.
    \end{lem}
    
    \noindent \textit{End of the proof of Theorem \ref{TheoremeExistenceUniciteMcKean}.} Using the fact that $L^2(\Omega, \C([0,t_0], M))$ is a complete space, there exists a unique limit 
    $$
    (X_t)_{t \in [0,t_0]} = \lim_{n \rightarrow \infty} (X^n_t)_{t \in [0,t_0]}
    $$ 
    which is measurable in $x$. Let $(M_t(\mu))_{0 \leq t \leq t_0} = ((X_t)_* \mu)_{0 \leq t \leq t_0}$, this process inherits the time continuity of $(X_t)_{0 \leq t \leq t_0}$ by lemma \ref{LemmeTransportOptimal}. Moreover, it is the limit of the sequence $(M_t^n(\mu))_{0 \leq t \leq t_0}$ since
    
    $$
    \mathbb{E} \left[\sup_{t \in [0,t_0]} W_2^2\left(M_t^n(\mu),M_t(\mu)\right) \right] \leq \mathbb{E}\left[ \sup_{t \in [0,t_0]}d_M^2(X_t^n(x),X_t(x)) \right] \underset{n \to \infty}{\longrightarrow} 0.
    $$
    By taking the limit in \eqref{Solution Induction}, we get that for all $f\in \C^\infty(M)$, $(X_t)_{0 \leq t \leq t_0} $ satisfies
    
    \begin{align*}
    f(X_t(x))-f(x) &= \int_0^t Z_i(\cdot, M_s(\mu))f(X_s(x)) \d W^i_s +  \int_0^t Z_0(\cdot, M_s(\mu))f(X_s(x)) \d s \\
    & + \frac{1}{2} \SN \int_0^t \he f(Z_i(X_s(x),M_s(\mu)),Z_i(X_s(x),M_s(\mu))) \d s.
    \end{align*}
    
    Now, to extend $(X_t)_{0 \leq t \leq t_0}$ to the interval [$0,2t_0]$ we can choose initial conditions $\bar{x} = X_{t_0}(x)$ and $\bar{\mu} = M_{t_0}(\mu)$. Since $(W^i_t - W^i_{t_0})_{t \geq t_0}$ is again a Brownian motion, and $t_0$ is independent of the initial conditions, we can construct as above a solution $(X_t(\bar{x}), M_t(\bar{\mu}))$ for $t_0 \leq t \leq 2t_0$ and it is easy to see that the process defined by
	
	$$
	X_t(x):= X_t(\bar{x}), \quad M_t(\mu) := M_t(\bar{\mu}), \quad \forall t \in [t_0,2t_0]
	$$
	solves the condtionnal McKean-Vlasov SDE \eqref{SDE Mckean vlasov RN}. We can iterate this argument to get the solution on any interval $[0,T]$.
	
	\textbf{Uniqueness :} Let $X,Y$ be two solutions of the conditional McKean-Vlasov SDE \eqref{SDE Mckean vlasov RN}. Then, $X$ and $Y$ are also solutions to the extended McKean-Vlasov SDE on $\mathbb{R}^d$. Using Itô's formula on the quantity $|X_t(x) - Y_t(x)|^2$, taking the expectation, using Lipschitz hypothesis and Grönwall's lemma we get the uniqueness by standard arguments of stochastic calculus, see for example \cite[Theorem $2.1$]{wang2020}. 
	
	\textbf{Diffeomorphism property.} Now, let $\mu \in \P_\infty$. For all $i=0,1, \dots, N$ we denote by $\bm Z_i(\psi)$ the horizontal lift $\mathfrak{h}_{\psi}(\bar{Z}_i)$. We consider the following SDE on $\mathscr{D}$:
	
	\begin{equation}\label{Equation dans le cas des difféos}
	    \circ \dd\varphi_t = \SN \bm Z_i(\varphi_t) \circ \d W^i_t + \bm Z_0'(\varphi_t) \d t, \quad \varphi_0 = \mathrm{id}.
	\end{equation}
	where $\bm Z_0' = \bm Z_0 - \frac{1}{2}\SN \bm \nabla_{\bm Z_i} \bm Z_i$. In particular, by using the Itô-Stratonovich conversion (Proposition \ref{Equivalence Ito Strato difféos}), the above Stratonovich's SDE is equivalent to the following Itô's one:
	
	$$
	\dd^{\bm \nabla}\varphi_t = \SN \bm Z_i(\varphi_t) \d W^i_t + \bm Z_0(\varphi_t) \d t, \quad \varphi_0 = \mathrm{id}.
	$$
	Since the driving vector fields belong to $\Gamma^1_S(T\mathscr{D})$, Theorem \ref{Existence unicité Difféos} ensure that there exists a unique solution to this SDE on $\mathscr{D}$. Let $E_{x_0} : \psi \mapsto \psi(x_0)$ denote the evaluation map for $x_0 \in M$. This map is smooth (see Proposition \ref{P Evaluation map lisse}) and one can easily check that $T_\varphi E_{x_0}(\bm Z_i(\varphi)) = Z_i(\varphi) \circ \varphi(x_0)$ for all $\psi \in \mathscr{D}$. By applying the chain rule, we obtain that
	
	\begin{equation*}
	    \circ \d \varphi_t(x_0)=\circ \d E_{x_0}(\varphi_t) = T_{\varphi_t}E_{x_0}(\circ \d \varphi_t) = \SN Z_i(p(\varphi_t)) \circ \varphi_t(x_0) \circ \d W^i_t + Z_0'(p(\varphi_t)) \circ \varphi_t(x_0) \d t.
	\end{equation*}
	Thus, $(\varphi_t)_{t \geq 0}$ is the solution flow to \eqref{Conditionnal SDE}. By using the uniqueness of the solution, $(X_t)_{t \geq 0}$ is a stochastic flow of diffeomorphisms. 
    \end{proof}
    
    Now that the proof of Theorem \ref{TheoremeExistenceUniciteMcKean} is complete, we can prove Theorem \ref{Existence Unicité}.
    
    \begin{proof}[Proof of Theorem \ref{Existence Unicité}]  We have to show that the $\P$-valued process $(M_t(\mu))_{t \geq 0}$ obtained in \ref{TheoremeExistenceUniciteMcKean} is solution to:
    
    \begin{equation}\label{EDS dans le cas mesure}
        \dw^{\grad} M_t(\mu) = \SN \bar{Z}_i(M_t(\mu)) \d W^i_t + \bar{Z}_0(M_t(\mu)) \d t, \quad M_0(\mu) = \mu.
    \end{equation}
	It is the case since, by Theorem \ref{TheoremeExistenceUniciteMcKean},
	
	\begin{align*}
	F_f(M_t(\mu)) & = \int_M f(X_t(x)) \d\mu(x) \\
	&=  \int_M f(x) \d\mu(x) + \int_M \SN \int_0^t Z_i(\cdot, M_s(\mu)) f(X_s(x)) \d W^i_s \d\mu(x) + \int_M \int_0^t Z_0(\cdot,M_s(\mu)) f(X_s(x)) \d s \d\mu(x) \\
	& + \int_M \frac{1}{2} \SN \int_0^t \he f(Z_i(X_s(x),M_s(\mu)),Z_i(X_s(x),M_s(\mu))) \d s \d\mu(x) \\
	& = \int_M f(x) \d\mu(x) +  \SN \int_0^t\int_M Z_i(\cdot, M_s(\mu)) f(X_s(x)) \d \mu(x) \d W^i_s + \int_0^t \int_M Z_0(\cdot,M_s(\mu)) f(X_s(x)) \d \mu(x) \d s \\
	& +  \frac{1}{2} \SN \int_0^t \int_M \he f(Z_i(X_s(x),M_s(\mu)),Z_i(X_s(x),M_s(\mu))) \d\mu(x) \d s  \\
	& = F_f(\mu) + \int_0^t \L_{\bar{Z}_i} F_f(M_s(\mu)) \d W^i_s + \int_0^t \L_{\bar{Z}_0} F_f(M_s(\mu)) \d s + \frac{1}{2} \SN \int_0^t \bar{\he }F_f(\bar{Z}_i(M_s(\mu)),\bar{Z}_i(M_s(\mu))) \d s,
	\end{align*}
	where the last equality comes from \eqref{Equation hessienne constant} and the fact that $M_s(\mu) = (X_s)_* \mu$. If $\mu \in \P_\infty$, the fact that $(M_t(\mu))_{t \geq 0}$ lies in $\P_\infty$ follows from Proposition \ref{Processus mesure lisse}. 
    \end{proof}
    
    \begin{rem}
        As noticed in the proof of Theorem \ref{TheoremeExistenceUniciteMcKean}, it appears that the unique solution to the horizontal lift equation is the flow $(X_t \circ \varphi)_{t \geq 0}$ where $(X_t)_{t \geq 0}$ is solution to the following conditional McKean-Vlasov equation:
    
    \begin{equation*}
    \circ \d X_t(x) = \SN Z_i(X_t(x),\mu_t) \circ \d W^i_t + Z_0(X_t(x), \mu_t) \d t, \quad X_0(x) = x, \quad \mu_t = (X_t)_* \mu.
    \end{equation*}
    \end{rem}
        
    	\subsection{Covariant SDEs} \label{Subsection 4.5}
    	
    	In this subsection we introduce a way to obtain semimartingales on $T\mathscr{D}$. We start by stating some results about Covariant SDEs on Hilbert vector bundles due to Elworthy, see \cite{Elworthy1982}.
    	
    	\vspace{10 pt}

    \noindent \textbf{Covariant SDEs on Hilbert vector bundles.} Let $\mathscr{M}$ be a Hilbert manifold, $\pi_1 : \mathscr{F}^1 \rightarrow \mathscr{M}$ and $\pi_2 : \mathscr{F}^2 \rightarrow \mathscr{M}$ be Hilbert vector bundles over $\mathscr{M}$. Let also $\mathscr{K}$ be a Hilbert space. 

    \begin{defi}
        Let $\Psi : \mathscr{M} \rightarrow \mathscr{M}$. A map $J : \mathscr{F}^1 \rightarrow \mathscr{F}^2$ is said to belong to $\mathrm{Hom}_\Psi(\mathscr{F}^1,\mathscr{F}^2)$ if the diagram
        \begin{center}
\begin{tikzpicture}[>=stealth]
  \node (E1) at (0,1) {$\mathscr{F}^1$};
  \node (E2) at (4,1) {$\mathscr{F}^2$};
  \node (E3) at (0,0)   {$\mathscr{M}$};
  \node (E4) at (4,0)   {$\mathscr{M}$};

  \draw[->, >=latex] (E1) -- (E2) node[midway,above]{$J$}; 
  \draw[->, >=latex] (E1) -- (E3) node[midway,left]{$\pi_1$};
  \draw[->, >=latex] (E3) -- (E4) node[midway,below]{$\Psi$}; 
  \draw[->, >=latex] (E2) -- (E4) node[midway,right]{$\pi_2$};
\end{tikzpicture}
\end{center}
is commutative and for all $x \in \mathscr{M}$, the induced map $J_x : \mathscr{F}^1_x \rightarrow \mathscr{F}^2_{\Psi(x)}$ is a continuous linear map. Let $k \in \mathbb{N}$, $J$ is said to belong to $\mathrm{Hom}_{\Psi}^k(\mathscr{F}^1,\mathscr{F}^2)$ if $\Psi \in \C^k(\mathscr{M},\mathscr{M})$ and for all $x \in \mathscr{M}$ there exists trivializing maps

$$
\theta : \pi_1^{-1}(U) \longrightarrow U \times E
$$
and

$$
\theta' : \pi_2^{-1}(U') \longrightarrow U' \times E'
$$
at $x_0$ and $\Psi(x_0)$ respectively, with $E, E'$ some Hilbert spaces, such that $\Psi(U) \subset U'$, and such that the map of $U$ into $\mathcal{L}_c(E,E')$ given by

$$
x \longmapsto \theta' \circ J_x \circ \theta^{-1}
$$
is a $\C^k$ map.
    \end{defi} 
    
    \begin{exemple}
    By construction (see \cite[Theorem 9.1]{EbinMarsden1970} or \cite{Eliasson1967}), for any vector field $\bm A \in \Gamma^k(T \mathscr{D}^s)$, $\bm \nabla^s \bm A : T\mathscr{D}^s \rightarrow T \mathscr{D}^s$ belongs to $\mathrm{Hom}_{\mathrm{id}}(T\mathscr{D}^s, T\mathscr{D}^s)$.
    \end{exemple}
    
    Let $(X_t)_{t \geq 0}$ be solution to 
	
	\begin{equation*} 
	    \circ \d X_t = \SN A_i(X_t) \circ \d W^i_t + A_0(X_t) \d t, \quad X_0 = x \in \mathscr{M},
	\end{equation*}
	with $A_0, A_1, \dots,  A_N \in \Gamma^2(T\mathscr{M})$. Let $\pi : \mathscr{F} \rightarrow \mathscr{M}$ be a smooth vector bundle endowed with a connection $ \nabla$ and a unique stochastic parallel transport $(\tau_{0,t})_{t \geq 0}$ along $(X_t)_{t \geq 0}$. 
	Let $J_0, J_1, \dots, J_N \in \mathrm{Hom}_{\mathrm{id}}^0(\mathscr{F},\mathscr{F})$ be continuous vector bundle morphisms. An adapted $\mathscr{F}$-valued stochastic process $(U_t)_{t \geq 0}$ along $(X_t)_{t \geq 0}$ is said to be solution to the following covariant SDE:
	
	\begin{equation} \label{EDS covariante générale}
	    \circ D_t U_t = \sum_{i=1}^{N} J_i(U_t) \circ \d W^i_t + \bm J_0(U_t) \d t, \quad U_0 = U,
	\end{equation}
	if $(\tau_{0,t}^{-1} U_t)_{t \geq 0}$ is solution to the following SDE on $\mathscr{F}_\varphi$:
	
	$$
	\circ \d \tau_{0,t}^{-1} U_t = \sum_{i=1}^{N} \tau_{0,t}^{-1} J_i(U_t) \circ \d W^i_t + \tau_{0,t}^{-1} J_0 (U_t) \d t, \quad U_0 = U.
	$$
	
	\begin{thm}[\cite{Elworthy1982}, Theorem 13.A] \label{Existence unicité Elworthy}
	Assume that $ J_i \in \mathrm{Hom}_{\mathrm{id}}^1(\mathscr{F},\mathscr{F})$ for all $i=0,1,\dots,N$. Then, there exists a unique continuous solution to \eqref{EDS covariante générale}.
	\end{thm}
	In \cite[Proposition 13B]{Elworthy1982}, the author proved that if a stochastic process $(U_t)_{t \geq 0}$ is solution to \eqref{EDS covariante générale}, then for all $\C^2$ linear form $\Phi : \mathscr{F} \rightarrow \mathscr{K}$ with $\mathscr{K}$ a Hilbert space, the following holds:
	
	\begin{equation*}
	    \circ \d \Phi(U_t) = (\nabla \Phi(\circ \d X_t))(U_t) + \Phi(\circ D_t A_t).
	\end{equation*}
	As a straightforward consequence, the following holds:
	
	\begin{prop} \label{Vector bundle chain rule}
	    Let $(U_t)_{t \geq 0}$ be solution to \eqref{EDS covariante générale} and $L : \mathscr{F} \rightarrow \mathscr{F}$ a $\C^2$ vector bundle morphism. Then, $( L(U_t))_{t \geq 0}$ is a continuous semimartingale solution to:
	    
	    \begin{equation*}
	        \circ D_t L(U_t) = (\nabla L(\circ \d X_t))(U_t) + L(\circ  D_t U_t).
	    \end{equation*}
	\end{prop}
	\vspace{10 pt}
    
    	\noindent \textbf{Covariant SDEs on $T\mathscr{D}$.} Let $s > m/2$, $j(s) \geq s$. One can easily check that the injection $\iota : \mathscr{D}^{j(s)} \rightarrow \mathscr{D}^s$ is a smooth map. We denote by $\mathscr{F}^s = \iota^\ast T\mathscr{D}^s = \sqcup_{\varphi \in \mathscr{D}^{j(s)}}T_\varphi \mathscr{D}^s$ the pullback vector bundle over $\mathscr{D}^{j(s)}$. We denote by $\bm \nabla$ the pullback connection $\iota^*(\bm \nabla)$ on $\mathscr{F}^s$ and for any continuous semimartingale $(\varphi_t)_{t \geq 0}$ on $\mathscr{D}^{j(s)}$, $\tau_{0,t}$ is the unique stochastic parallel transport on $\mathscr{F}^s$. Furthermore, for $\bm A \in \Gamma^k(T\mathscr{D}^{j(s)})$, we will denote by $\bm A$ the $\C^k$ section $\iota_*(\bm A) \in \Gamma^k(\mathscr{F}^s)$. Let
    	
    	\begin{equation*}
    	    \mathscr{J} := \left \{ j : \mathbb{N}^* \rightarrow \mathbb{N}^* \text{ increasing function} \, : \, j(s) \geq s, \, \forall s >m/2 \right\}.
    	\end{equation*}
    	Let $j \in \mathscr{J}$, then $T\mathscr{D}$ is endowed with the ILH manifold structure given by:
    	
    	$$
    	T\mathscr{D} = \bigcap_{s>m/2} \mathscr{F}^s.
    	$$
    	Note that for any $j \in \mathscr{J}$, and any $s \geq s' >m/2$, the injection $\mathscr{F}^s \rightarrow \mathscr{F}^{s'}$ is a smooth map. Indeed, in a neighborhood of $\varphi \in \mathscr{D}^s$, this injection is just the natural injection:
    	
    	$$
    	T_\varphi \mathscr{D}^{j(s)} \times T_\varphi \mathscr{D}^s \longrightarrow T_\varphi \mathscr{D}^{j(s')} \times T_\varphi \mathscr{D}^{s'}.
    	$$
    	In what follows, even when $j(s) \in \mathscr{J}$ is not mentioned explicitly, $\mathscr{F}^s$ is understood to be a vector bundle over $\mathscr{D}^{j(s)}$ with its natural projection.
	
	\begin{defi} \label{Defi vector bundle morphism}
	    Let $\Psi \in \C^k_S(\mathscr{D},\mathscr{D})$. A map $\bm J : T\mathscr{D} \rightarrow T\mathscr{D}$ is said to belong to $\mathrm{Hom}^k_{\Psi}(T\mathscr{D}, T\mathscr{D})$ if there exists $s_0 \in \mathbb{N}$ and $j\in \mathscr{J}$ such that for all $s>s_0$ there exists a $\C^k$ extension $\bm J^s \in \mathrm{Hom}_{\Psi}^k(\mathscr{F}^s, \mathscr{F}^s)$. 
	\end{defi}
	
	\begin{exemple}
	Let $k \in \mathbb{N}^*$, $\Psi \in \C^k_S(\mathscr{D},\mathscr{D})$. The differential $\bm T\Psi$ belong to $\mathrm{Hom}_{\Psi}^{k-1}(T\mathscr{D},T\mathscr{D})$.
	\end{exemple}
	Let $(\varphi_t)_{t \geq 0}$ be the unique solution to the following SDE:
	
	\begin{equation*} 
	    \circ \dd \varphi_t = \SN \bm Z_i(\varphi_t) \circ \d W^i_t + \bm Z_0(\varphi_t) \d t, \quad \varphi_0 = \varphi,
	\end{equation*}
	with $\bm Z_0, \bm Z_1, \dots,  \bm Z_N \in \Gamma^2_S(T\mathscr{D})$. Let $\bm J_0, \bm J_1, \dots, \bm J_N \in \mathrm{Hom}_{\Psi}^0(T\mathscr{D},T\mathscr{D})$. A $T\mathscr{D}$-valued stochastic process $(\bm A_t)_{t \geq 0}$ along $(\varphi_t)_{t \geq 0}$ is said to be solution to the following covariant SDE:
	
	\begin{equation} \label{EDS covariante générale ILH}
	    \circ \bm D_t \bm A_t = \SN \bm J_i(\bm A_t) \circ \d W^i_t + \bm J_0(\bm A_t) \d t, \quad \bm A_0 = A \in T_\varphi \mathscr{D},
	\end{equation}
	if there exists $s_0 \in \mathbb{N}$ such that for all $s > s_0$, $(\bm A_t)_{t \geq 0}$ is solution to
	
	\begin{equation*}
	    \circ \bm D_t \bm A_t = \SN \bm J^s_i(\bm A_t) \circ \d W^i_t + \bm J^s_0(\bm A_t) \d t, \quad \bm A_0 = A \in T_\varphi \mathscr{D},
	\end{equation*}
	with $\bm J_i^s \in \mathrm{Hom}_{\mathrm{id}}^0(\mathscr{F}^s, \mathscr{F}^s)$ a continuous extension of $\bm J_i$ as in Definition \ref{Defi vector bundle morphism} for all $i=0,1\dots,N$. The following results are consequences of their Hilbert vector bundle counterpart:
	
	\begin{prop} \label{P Existence unicité ILH covariante ILH}
	    Assume that $\bm J_i\in \mathrm{Hom}^1_{\mathrm{id}}(T\mathscr{D},T\mathscr{D})$ for all $i=0,1,\dots,N$. Then, there exists a unique continuous solution to \eqref{EDS covariante générale ILH}. 
	\end{prop}
	
	\begin{prop} \label{Vector bundle chain rule ILH}
	    Let $(\bm A_t)_{t \geq 0}$ be solution to \eqref{EDS covariante générale ILH}. Then, for all $\C^2$ linear form $\Phi : T\mathscr{D} \rightarrow \mathscr{K}$ with $\mathscr{K}$ a Hilbert space, the following holds:
	
	\begin{equation*}
	    \circ \dd \Phi(\bm A_t) = (\bm \nabla \Phi(\circ \dd \varphi_t))(\bm A_t) + \Phi_{\varphi_t}(\circ \bm D_t \bm A_t).
	\end{equation*}
	As a consequence, for $\bm L\in \mathrm{Hom}^2_{\Psi}(T\mathscr{D},T\mathscr{D})$ with $\Psi \in \C^2_S(\mathscr{D},\mathscr{D})$, $(\bm L_{\varphi_t}(\bm A_t))_{t \geq 0}$ is a semimartingale solution to:
	    
	    \begin{equation*}
	        \circ \bm D_t \bm L(\bm A_t) = (\bm \nabla\bm L_{\varphi_t}(\circ \dd \varphi_t))(\bm A_t) + \bm L_{\varphi_t}(\circ \bm D_t \bm A_t).
	    \end{equation*}
	\end{prop} 
	
	\vspace{10 pt}

	\noindent \textbf{Covariant SDEs on $T\P_\infty$.} Let $\mu \in \P_\infty$ and $(\mu_t)_{t \geq 0}$ be solution to
    
	\begin{equation*}
	    \circ \d\mu_t = \SN \bar{Z}_i(\mu_t) \circ \d W^i_t + \bar{Z}_0(\mu_t) \d t, \quad \mu_0 =\mu,
	\end{equation*}   
	where for all $i=0,1, \dots, N$, $\bar{Z}_i \in \Gamma^2_S(T\P)$. Let $(\Phi_t)_{t \geq 0}$ be the horizontal lift of $(\mu_t)_{t \geq 0}$, namely the solution to
	
	\begin{equation*}
	\circ \dd\Phi_t(\varphi) = \SN \bm Z_i(\Phi_t(\varphi)) \circ \d W^i_t + \bm Z_0(\Phi_t(\varphi)) \d t, \quad \Phi_0(\varphi) = \varphi \in \mathscr{D}_\mu,
	\end{equation*}
	where $\bm Z_i \in \Gamma^3_S(T\mathscr{D})$ is the horizontal lift of $\bar{Z}_i$ for all $i=0,1, \dots, N$.
	
	\begin{defi}
	    A map $\bar{J} : T \P_\infty \rightarrow T \P_\infty$ such that for all $\mu \in \P_\infty$, $\bar{J}_\mu \in \mathcal{L}_c(T_\mu \P_\infty, T_\mu \P_\infty)$ is said to belong to $\mathrm{Hom}_{\mathrm{id}}^k(T\P_\infty, T \P_\infty)$ if there exists $\bm J \in \mathrm{Hom}_{\mathrm{id}}^k(T\mathscr{D},T\mathscr{D})$ such that the diagram:
	    
	    \begin{center}
\begin{tikzpicture}[>=stealth]
  \node (E1) at (0,1) {$T\mathscr{D}$};
  \node (E2) at (4,1) {$T\mathscr{D}$};
  \node (E3) at (0,0)   {$T\P_\infty$};
  \node (E4) at (4,0)   {$T\P_\infty$};

  \draw[->, >=latex] (E1) -- (E2) node[midway,above]{$\bm J$}; 
  \draw[->, >=latex] (E1) -- (E3) node[midway,left]{$Tp$};
  \draw[->, >=latex] (E3) -- (E4) node[midway,below]{$\bar{J}$}; 
  \draw[->, >=latex] (E2) -- (E4) node[midway,right]{$Tp$};
\end{tikzpicture}
\end{center}
commutes.
	\end{defi}
	We will forget about the subscript $\mu$ when there is no ambiguity about the point considered.
	
	Let $\bar{J}_0, \bar{J}_1, \dots, \bar{J}_N \in \mathrm{Hom}_{\mathrm{id}}^0(T\P_\infty,T\P_\infty)$. A $T\P_\infty$-valued stochastic process $(\bar{A}_t)_{t \geq 0}$ along $(\mu_t)_{t \geq 0}$ is said to be solution to 
	
	$$
	\circ \bar{D}_t \bar{A}_t = \SN \bar{J}_i(\bar{A}_t) \circ \d W^i_t + \bar{J}_0(\bar{A}_t) \d t, \quad \bar{A}_0 = A \in \T_\mu \P_\infty,
	$$
	if for all $\varphi \in \mathscr{D}_\mu$, $(\bm A_t)_{t \geq 0} = (\mathfrak{h}_{\Phi_t(\varphi)}(\bar{A}_t))_{t \geq 0}$ is solution to
	
	\begin{equation*}
	    \circ \bm D_t \bm A_t = \SN \bm J_i(\bm A_t) \circ \d W^i_t + \bm J_0(\bm A_t) \d t, \quad \bm A_0 = \mathfrak{h}_{\varphi}(A) \in T_\varphi \mathscr{D},
	\end{equation*}
	with $\bm J_i \in \mathrm{Hom}^0_{\mathrm{id}}(T\mathscr{D}, T\mathscr{D})$ such that
	
	$$
	\bm T p \circ \bm J_i = \bar{J}_i \circ \bm T p, \quad i=0,1,\dots,N.
	$$
	\begin{rem}
	    For finite dimensional manifolds this definition is equivalent to the usual one. 
	\end{rem}

\section{Stochastic parallel transport on $\P_\infty$} \label{Section 5}

 This section is devoted to the proof of the existence and uniqueness of the stochastic parallel transport along diffusions on $\P_\infty$. Since the main idea of the proof is to adapt Theorem \ref{Transport parallel en dim finie} in our infinite-dimensional setting. The reader is referred to Subsections \ref{Subsection 2.2} and \ref{Subsection 2.4} for a presentation of the proof in the finite-dimensional case. 

We want to construct a horizontal transport $(\mathcal{T}_{0,t}^h)_{t \geq 0}$ which plays the same role as $(\tau_{0,t}^h)_{t \geq 0}$ in Subsection \ref{Subsection 2.2}. The starting point of the proof lies in the fact that, in the finite dimensional case, $(\tau_{0,t}^h)_{t \geq 0}$ is characterized by Equation \eqref{Horizontal parallel transport finite dimension}, as mentioned in Remark \ref{Remarque caractérosation du transport}. 

We begin by providing a precise definition of the stochastic parallel transport on $T\P_\infty$ together with some preliminary material that will be used to establish existence and uniqueness for the infinite-dimensional analogue of \eqref{Horizontal parallel transport finite dimension}. Then, we show that the horizontal and vertical projections are smooth vector bundle morphisms in Theorem \ref{Dérivées covariantes projection horizontale et verticale}. Let $(\mu_t)_{t \geq 0}$ be a solution to a $\P_\infty$-valued SDE. 

Given a $\P_\infty$-valued stochastic process $(\mu_t)_{t \geq 0}$ solving a SDE, we associate to it its horizontal lift $(\Phi_t(\varphi))_{t \geq 0}$ with initial datum $\varphi \in \mathscr{D}_\mu$ (see Definition \ref{Defi relevé horizontal}). The equation satisfied by $(\mathcal{T}_{0,t}^h)_{t \geq 0}$ is derived from Proposition \ref{Prop connexion levi civita releve} and solved in Proposition \ref{Existence et unicité transport parallele horizontal}. In the same proposition, we also show that if the initial condition is horizontal, the solution remains horizontal. In Theorem \ref{P Equivariance transport horizontal D} we show that the horizontal transport, $(\mathcal{T}_{0,t}^h(\varphi, \cdot))_{t \geq 0}$, is $G$-right invariant. In Theorem  \ref{Existence unicité transport parallèle}, we show that the process given by $\bm T p(\mathcal{T}_{0,t}^h)$ is the stochastic parallel transport along $(\mu_t)_{t \geq 0}$. 

Additionally, we construct a lift of the stochastic parallel transport on $\P_\infty$ which is equivariant on the $G \rtimes \mathfrak{g}$-principal bundle $T \mathscr{D}$ in Proposition \ref{P Transport equivariant}.

\subsection{Definitions and properties} \label{Subsection 5.1}
	In this subsection we introduce the definition of a stochastic parallel transport. Then, we introduce the normal tensor and the Riemann curvature operator on $\mathscr{D}$ and present some of their properties. The reader is referred to Subsections \ref{Subsection 4.1}, \ref{Subsection 4.2} and \ref{Subsection 4.5} for definitions and properties about semimartingales on $T\mathscr{D}$ or $T \P_\infty$. The end of the subsection is dedicated to the study of the regularity properties of the normal tensor.
	
\begin{defi} \label{Definition transport parallele}
	Let $\bar{Z}_0, \bar{Z}_1, \dots, \bar{Z}_N \in \Gamma^2_S(T\P)$ and $(\mu_t)_{t \geq 0}$ be a solution to 
	
	$$
	\circ \dw\mu_t = \SN \bar{Z}_i(\mu_t) \circ \d W^i_t + \bar{Z}_0(\mu_t) \d t, \quad \mu_0 =\mu \in \P_\infty.
	$$    
	A $\T\mathcal{\P}$-valued semimartingale $(\bar{A}_t)_{t \geq 0}$ along $(\mu_t)_{t \geq 0}$ is said to be a \textit{stochastic parallel transport} if it is solution to the following covariant SDE:
	
	\begin{equation} \label{Equation du transport para}
	\circ \bar{D}_t \bar{A}_t =  0.
	\end{equation}
\end{defi}

	\begin{defi}
	Let $s>m/2$. The \textit{Riemann curvature operator} on $\mathscr{D}^s$ evaluated in $\varphi \in \mathscr{D}^s$ is defined as
	
	\begin{align*}
	    \bm R^s: (T_{\varphi} \mathscr{D}^s)^3 & \longrightarrow T_{\varphi} \mathscr{D}^s \\
	    (A \circ \varphi, B \circ \varphi, C \circ \varphi) &\longmapsto \bm R^s(A \circ \varphi, B \circ \varphi) C \circ \varphi =  \left( R(A, B) C \right) \circ \varphi,
	\end{align*}
	where $R$ is the Riemann curvature tensor on $M$. The Riemann curvature tensor $\bm R$ on $\mathscr{D}$ is defined as the limit object of the $\bm R^s$.
	\end{defi}

	\begin{rem}[Comparison with the usual definition of the Riemann curvature operator.]
	    On a Riemannian manifold, the Riemann curvature operator is defined by a formula involving covariant derivatives of vectors fields. It can then be shown that this operator is indeed a tensor and thus admits a pointwise definition. Actually, one can check that $\bm R^s$ writes
	    
	    $$
	    \bm R^s(\bm A(\varphi), \bm B(\varphi)) \bm C(\varphi)  = \bm \nabla_{\bm A} \bm \nabla_{\bm B} \bm C(\varphi) -\bm \nabla_{\bm B} \bm \nabla_{\bm A} \bm C(\varphi) - \bm \nabla_{\left[\bm A, \bm B \right]} \bm C(\varphi),
	    $$
	    for $\bm A, \bm B, \bm C \in \Gamma^2(T\mathscr{D}^s)$. Hence, because of the symmetries, $\bm R$ is tensorial. In particular, we recover the usual definition of the curvature operator on a Riemannian manifold.
	    
	    Also, note that by construction, since $\bm \nabla^s$ is a smooth connection on $T\mathscr{D}^s$, $\bm R$ is a smooth section of the bundle of trilinear forms from $T \mathscr{D}^s$ into itself $L^3(T\mathscr{D}^s, T\mathscr{D}^s)$, see \cite[Chapter IX]{Lang1999} for instance. As a straightforward corollary, for $\bm A, \bm B \in \Gamma^k(T\mathscr{D}^s)$, the map $\bm R^s (\cdot, \bm A) \bm B : T\mathscr{D}^s \rightarrow T\mathscr{D}^s$ belongs to $\mathrm{Hom}^k_{\mathrm{id}}(T\mathscr{D}^s, T\mathscr{D}^s)$.
	\end{rem}

	\noindent \textbf{Normal tensor.} We introduce the \textit{normal tensor} (almost) in the same way as in \cite[Section 5]{gigli2011second}. It plays a role analogous to the O'Neill tensor appearing in Proposition \ref{Connexion Levi-Civita Horizontale dim finie} in the context of the diffeomorphisms group. In particular, it will be the main tool for the proof of the existence of the stochastic parallel transport. In this subsection, and unless explicitly stated otherwise, s denotes an integer satisfying
	
	$$
	s>m/2+3.
	$$
    
    \begin{defi}
		 Let $\varphi \in \mathscr{D}^{s+1}$. We define the \textit{normal operator} $\mathcal{N}_\varphi$ to be:
		
		\begin{align*}
		\mathcal{N}_\varphi : T_\varphi \mathscr{D}^{s+1} \times T_\varphi \mathscr{D}^s & \longrightarrow \V_\varphi \mathscr{D}^s \\
		(U \circ \varphi, A \circ \varphi) & \longmapsto \mathcal{N}_{\varphi} (U \circ \varphi, A \circ \varphi) = P_{\V_{\varphi}}\left(  ((\nabla U)^T \cdot A) \circ \varphi  \right),
		\end{align*}
		where $(\nabla U)^T$ is the adjoint of the linear operator $\nabla U : A \mapsto (\nabla U) \cdot A = \nabla_A U$ for the Riemannian metric on $M$.
	\end{defi}
	
	The reader is referred to \eqref{k symmetric product} for the definition of the symmetric product $\mathrm{SP}^k$. 
	\begin{prop}
		Let $j(s) > s$. For all $\varphi \in \mathscr{D}^{j(s)}$, the operator $\mathcal{N}_\varphi$ extends in a unique way to an antisymmetric bilinear operator:
		$$
		\mathcal{N}_{\varphi} : \mathrm{SP}^2\left( T_\varphi \mathscr{D}^{j(s)} \times T_\varphi \mathscr{D}^{s} \right) \rightarrow \V_\varphi \mathscr{D}^{s}.
		$$
	\end{prop}
	
	\begin{proof}
	The bilinearity is clear by definition. Let $A,B \in \Gamma^{(s)}(TM)$ and $\check{A}, \check{B}$ the associated right invariant vector fields. For all $C \circ \varphi \in T_\varphi \mathscr{D}^{s}$ we have the following equality:
	
	$$
	\mathcal{L}^\mathscr{D}_{C \circ \varphi} \langle \check{A}(\varphi), \check{B}(\varphi) \rangle_\varphi = \langle \bm \nabla_{C \circ \varphi} \check{A}(\varphi), \check{B}(\varphi) \rangle_\varphi + \langle \check{A}(\varphi), \bm \nabla_{C \circ \varphi} \check{B}(\varphi) \rangle_\varphi.
	$$
	It follows:
	
	\begin{align*}
	    &\langle \bm \nabla_{C \circ \varphi} \check{A}(\varphi), \check{B}(\varphi) \rangle_\varphi + \langle \check{A}(\varphi), \bm \nabla_{C \circ \varphi} \check{B}(\varphi) \rangle_\varphi \\
	    & \quad = \int_M \langle \nabla_{C \circ \varphi(x)}A\circ \varphi(x), B\circ \varphi(x) \rangle_{\varphi(x)} + \langle A\circ \varphi(x), \nabla_{C \circ \varphi(x)}B\circ \varphi(x) \rangle_{\varphi(x)}  \d \vol(x) \\
	    & \quad = \int_M \langle C \circ \varphi(x), ((\nabla A)^T \cdot B + (\nabla B)^T \cdot A) \circ \varphi(x) \rangle_{\varphi(x)} \d \vol(x) \\
	    & \quad = \int_M \langle C \circ \varphi(x), \nabla \langle A,B \rangle \circ \varphi(x) \rangle_{\varphi(x)} \d \vol(x),
	\end{align*}
	which lead us to the following equality:
	
	$$
	\bm \nabla \langle \check{A}(\varphi), \check{B} (\varphi) \rangle_\varphi = \left( \nabla \langle A, B \rangle \right) \circ \varphi =\left((\nabla A)^T \cdot B + (\nabla B)^T \cdot A \right) \circ \varphi.
	$$
	Since the vertical projection of the left hand term vanishes, we obtain the antisymmetry property:
	
	$$
	\mathcal{N}_{\varphi} (B,A) = -\mathcal{N}_{\varphi} (A,B).
	$$
	We naturally extend $\mathcal{N}_\varphi$ by antisymmetry to $\mathcal{N}_\varphi : \mathrm{SP}^2\left( T_\varphi \mathscr{D}^{s} \times T_\varphi \mathscr{D}^{s-1} \right) \rightarrow T_\varphi \mathscr{D}^{s-1}$.
	
	Since $\iota_* : T_\varphi \mathscr{D}^{j(s)} \rightarrow T_\varphi \mathscr{D}^s$ is a linear continuous map for $\varphi \in \mathscr{D}^{j(s)}$, $\mathcal{N}_\varphi(\iota_*(\cdot), \cdot)$ is an antisymmetric extension to $\mathrm{SP}^2\left( T_\varphi \mathscr{D}^{j(s)} \times T_\varphi \mathscr{D}^{s} \right)$.
	\end{proof} 
	
	\begin{prop} \label{Tenseur normal dans le cas horizontal}
	Let $\bm Z \in \Gamma^1(\mathscr{H}\mathscr{D}^{s})$ be a horizontal vector field and $A \circ \varphi \in T_\varphi \mathscr{D}^{s}$. Then, for all $\varphi \in \mathscr{D}^{s}$, the following holds:
	
	\begin{equation} \label{E 3.9}
	P_{\V_{\varphi}} \left( \bm \nabla_{A \circ \varphi} \bm Z(\varphi) \right) = \mathcal{N}_\varphi(\bm Z(\varphi), A \circ \varphi).
	\end{equation}
	As a consequence, the map
	
	\begin{align*}
	    \mathrm{SP}^2(\Gamma^1(\H\mathscr{D}^{s}) \times \Gamma(T\mathscr{D}^s)) &\longrightarrow \Gamma(\V\mathscr{D}^s) \\
	    \left(\bm Z, \bm A \right) & \longmapsto P_{\V}\left( \bm \nabla_{\bm A} \bm Z \right)
	\end{align*}
	is a tensor. In particular, for $\bm Z_1, \bm Z_2 \in \Gamma^1(\mathscr{H}\mathscr{D}^s)$, the following holds:
	
	\begin{equation} \label{E 3.10}
	P_{\V_{\varphi}} \left( \bm \nabla_{\bm Z_1} \bm Z_2(\varphi) \right) = \mathcal{N}_\varphi(\bm Z_1(\varphi), \bm Z_2(\varphi)) = -\mathcal{N}_\varphi(\bm Z_2(\varphi), \bm Z_1(\varphi)) = - P_{\V_{\varphi}} \left( \bm \nabla_{\bm Z_2} \bm Z_1(\varphi) \right).
	\end{equation}
	\end{prop}
	
	\begin{proof}
	Let $\bm Z \in \Gamma^1_S(\mathscr{H}\mathscr{D}^s)$ such that $ \bm Z(\varphi) = \nabla \phi(\varphi) \circ \varphi$ for all $\varphi \in \mathscr{D}^s$. We have
		
		$$
		\bm \nabla_{A \circ \varphi} \bm Z (\varphi) =  \nabla_{A \circ \varphi} \nabla \phi (\varphi) \circ \varphi + \left(\nabla  \mathcal{L}^\mathscr{D}_{A \circ \varphi} \phi(\varphi) \right) \circ \varphi.
		$$
		Since the second term on the right hand side is horizontal, we get
		
		$$
		P_{\V_{\varphi}} \left(\bm \nabla_{A \circ \varphi} \bm Z \right)(\varphi) = P_{\V_{\varphi}} \left(  \nabla_{A \circ \varphi} \nabla \phi(\varphi)  \circ \varphi \right) = P_{\V_{\varphi}} \left(  (\nabla \nabla \phi(\varphi)) \cdot A)  \circ \varphi \right).
		$$
		We conclude by using the fact that the Hessian on $M$ is symmetric. Note that \eqref{E 3.10} is a direct consequence of \eqref{E 3.9} and of the antisymmetry of the normal tensor.
		
	\end{proof}
	
	In the following definition we introduce the linear counterpart of the normal tensor together with its adjoint, following \cite[Definition 5.11]{gigli2011second}.

	\begin{defi}
		Let $\varphi \in \mathscr{D}^{s}$, $U \circ \varphi \in T_{\varphi} \mathscr{D}^{s+1}$. We define the linear operator $\mathcal{O}_{U \circ \varphi} : T_\varphi \mathscr{D}^{s} \rightarrow \V_{\varphi}\mathscr{D}^{s}$ to be:
		
		$$
		\mathcal{O}_{U \circ \varphi}(A\circ \varphi) = \mathcal{N}_{\varphi} (U \circ \varphi,A \circ \varphi), \quad \forall A \circ \varphi \in T_\varphi\mathscr{D}^s.
		$$
		We also define $\mathcal{O}_{U \circ \varphi}^\ast$ to be the adjoint of $\mathcal{O}_{U \circ \varphi}$ for the scalar product $\langle \cdot, \cdot \rangle_{\varphi}$.
	\end{defi}

	\begin{rem} \label{Adjoint de l'opérateur O}
	Let $U \circ \varphi \in T_\varphi \mathscr{D}^{s}$, $A \circ \varphi \in T_\varphi \mathscr{D}^s$. The following identity holds:
	
	\begin{equation*}
	    \mathcal{O}^\top_{U \circ \varphi}(A\circ \varphi) = \mathcal{O}^\top_{U \circ \varphi} \left( P_{\V_\varphi}\left(A \circ \varphi\right) \right).
	\end{equation*}
	Indeed, let $B \circ \varphi \in T_\varphi \mathscr{D}^s$, we have:
	
	\begin{align*}
	    \langle \mathcal{O}^\top_{U \circ \varphi}(A\circ \varphi), B \circ \varphi \rangle_\varphi &= \langle A \circ \varphi, \mathcal{O}_{U \circ \varphi}( B \circ \varphi) \rangle_\varphi \\
	    &= \langle P_{\V_\varphi}(A \circ \varphi), \mathcal{O}_{U \circ \varphi}( B \circ \varphi) \rangle_\varphi \\
	    & = \langle \mathcal{O}^\top_{U \circ \varphi}(P_{\V_\varphi}(A\circ \varphi)), B \circ \varphi \rangle_\varphi.
	\end{align*}
	From this remark, it becomes clear that for a $G$-right invariant vector field $\bm Z \in \Gamma^1(\H \mathscr{D}^s)$, the following property holds:
	
	$$
	\mathcal{O}^\top_{\bm Z} = \bm \nabla \bm Z \circ P_\V.
	$$
	\end{rem}
	The notion of normal tensor extends naturally for $\bm Z \in \Gamma^1_S(T\mathscr{D})$ to
    
    $$
    \mathcal{O}_{\bm Z} : T \mathscr{D} \longrightarrow \V \mathscr{D}. 
    $$ 

The following result will be the cornerstone of the proof of the existence of the stochastic parallel transport. This is the infinite-dimensional analogous of Proposition \ref{Connexion Levi-Civita Horizontale dim finie}.
    \begin{prop} \label{Prop connexion levi civita releve}
     Let $\varphi \in \mathscr{D}$, $\bar{Z}_1, \bar{Z}_2 \in \Gamma^1_S (T\P)$, and let $\bm Z_1,\bm Z_2$ denote their horizontal lift. The following identity holds:
     
     \begin{equation*}
    \bm \nabla_{\bm Z_2(\varphi)} \bm Z_1(\varphi) = \h_{\varphi} \left( \grad_{\bar{Z}_2(\mu)} \bar{Z}_1(\mu) \right) - \mathcal{O}_{\bm Z_1}(\bm Z_2)(\varphi).
    \end{equation*}
     
    \end{prop} 
    
    \begin{proof}
    The proof is identical to that of the finite-dimensional case (see Proposition \ref{Connexion Levi-Civita Horizontale dim finie}). More precisely, we apply Koszul's formula \eqref{Koszul D} to obtain the infinite-dimensional counterpart of \eqref{E23} and the conclusion follows by Proposition \ref{Tenseur normal dans le cas horizontal}.
    \end{proof}
    
    Let $s>m/2$, $j(s) \geq s$ and $\iota : \mathscr{D}^{j(s)} \rightarrow \mathscr{D}^s$ be the smooth injection. For the sake of readability, we shall frequently employ the following notations: for $\bm J \in \mathrm{Hom}(T \mathscr{D}^s, T\mathscr{D}^s)$, we will denote by $\bm J^*$ its pullback $\iota^*(\bm J)$. Until the end of the subsection, $\bm \nabla$ (resp. $\bm R$) will denote the Levi-Civita connection $\bm \nabla^s$ (resp. the Riemann tensor $\bm R^s$) on $T \mathscr{D}^s$. We denote by $\bm \nabla^{*}$ the pullback connection $\iota^*\left(\bm \nabla^s \right)$.

	    \begin{prop} \label{P derivee covariante pullback}
    Let $\bm A \in \Gamma^k(T\mathscr{D}^s)$. Then, $(\bm \nabla \bm A)^* \in \mathrm{Hom}_{\mathrm{id}}^{k-1}(\mathscr{F}^s, \mathscr{F}^s)$ and the following holds:
    
    $$
    \bm \nabla^{*} \left(\left( \bm \nabla\bm A \right)^*\right)( B \circ \varphi) =  \bm \nabla^2 \bm A (\cdot, B \circ \varphi) - \bm R(\cdot, B \circ \varphi) \bm A , \quad \forall \varphi \in \mathscr{D}^{j(s)}, \, \forall B \circ \varphi \in T_\varphi \mathscr{D}^s.
    $$
\end{prop}

\begin{proof}
    This is a straightforward corollary of Proposition \ref{pullback fonctorielle}.
\end{proof}
    
    An easy consequence of the preceding result is that for $\bm Z \in \Gamma^1(T\mathscr{D}^s)$, the tensor
    
    $$
    \mathcal{O}_{\bm Z} =  P^*_{\V} \circ ( \bm \nabla \bm Z)^*
    $$ 
    belongs to $\mathrm{Hom}_{\mathrm{id}}(\mathscr{F}^s, \mathscr{F}^s)$ for any $j \in \mathscr{J}$. We now turn to the study of the regularity properties of this morphism.

\subsection{Regularity of the projections}

In \cite[Appendix A, Lemma 2]{EbinMarsden1970}, the authors show, by an application of the method introduced by Ebin in \cite[Theorem 7.1]{Ebin1970}, that the datum of a linear differential operator $D$ of order $k$ between two vector bundles $\zeta, \eta$ induces a smooth vector bundle morphism $\bm D : H^s(M,\zeta) \rightarrow H^{s-k}(M,\eta)$ defined by the conjugation

$$
\bm D_\varphi : \tilde{R}_\varphi \circ D \circ \tilde{R}_{\varphi^{-1}},
$$
where $\tilde{R}_\varphi$ stands for $f \mapsto f \circ \varphi$. This result is noteworthy, since the operations $\varphi \mapsto \tilde{R}_\varphi$ and $\varphi \mapsto \tilde{R}_{\varphi^{-1}}$ are not smooth. In this subsection we show that this conjugation also have a smoothing effect on some families of linear differential operators parameterized by $\varphi$. By applying the method used in \cite[Appendix A]{EbinMarsden1970}, we will then show that the morphism $P_\H^* : \mathscr{F}^s \rightarrow \mathscr{F}^s$ is smooth. Until the end of the subsection we will consider $s>m/2+3$ and $j(s) = s+1$. We now state the following lemma, whose proof relies on standard arguments from differential calculus.

\begin{lem} \label{Lemme calcul diff}
    Let $f \in H^s(M,\mathbb{R})$ and $\varphi \in \mathscr{D}^{j(s)}$. Then, for all $A \in \Gamma^{(s)}(\varphi^*TM)$,
    
    \begin{equation} \label{Première formule}
        (T f( A \circ \varphi^{-1})) \circ \varphi = T_{\varphi(\cdot)}f(A) =  T (f \circ \varphi) ( (T\varphi)^{-1}(A))
    \end{equation}
    and
    
    \begin{equation*}
        \mathrm{div}(A \circ \varphi^{-1}) \circ \varphi = \mathrm{div}((T\varphi)^{-1}(A)) + \frac{1}{\mathrm{det}(T\varphi)} T \mathrm{det}((T\varphi))((T\varphi)^{-1}(A)).
    \end{equation*}
\end{lem}

\begin{proof}
The first point is a straightforward application of the chain-rule. For the second point, applying \eqref{Première formule} yields

\begin{align*}
        -\int_M f(x) \mathrm{div}(A \circ \varphi^{-1})(x) \d \vol(x) &= \int_M T_x f(A \circ \varphi^{-1}) \, \d \vol(x) \\
        & = \int_M T_{\varphi(x)}f(A \circ \varphi^{-1}) \cdot \mathrm{det}(T\varphi)(x) \, \d \vol(x) \\
        & = \int_M T_x (f\circ \varphi)( (T\varphi)^{-1}(A)) \cdot \mathrm{det}(T\varphi)(x) \d \vol(x) \\
        & = -\int_M f \circ \varphi(x) \mathrm{div}( \mathrm{det}(T\varphi) (T\varphi)^{-1}(A))(x) \d \vol(x) \\
        & = - \int_M f(x) \frac{1}{\mathrm{det}(T\varphi)(\varphi^{-1}(x))} \mathrm{div}(\mathrm{det}(T\varphi) (T\varphi)^{-1}(A))(\varphi^{-1}(x)).
    \end{align*}
    Thus, 
    
    $$
    \div(A \circ \varphi^{-1}) \circ \varphi = \frac{1}{\mathrm{det}(T\varphi)} \mathrm{div}(\mathrm{det}(T\varphi) (T\varphi)^{-1}(A)),
    $$
    and the conclusion follows from the Leibniz rule for the divergence.
\end{proof}
	 
	 \begin{lem}
	     The map $ \bm \div : \mathscr{F}^s \rightarrow \mathscr{D}^{j(s)} \times H^{s-1}(M,\mathbb{R})$ defined by
	 
	 $$
	 \bm \div_{\varphi} = \left( \varphi, \tilde{R}_\varphi \circ \div_{p(\varphi)} \circ \tilde{R}_{\varphi^{-1}}(\cdot)\right).
	 $$
	 belongs to $\mathrm{Hom}^\infty_{\mathrm{id}}(\mathscr{F}^s, \mathscr{D}^{j(s)} \times H^{s-1}(M,\mathbb{R}))$. The same result holds for $\bm \Grad : \mathscr{D}^{j(s)} \times H^{s+1}(M,\mathbb{R}) \rightarrow \mathscr{F}^s$ defined by
	 
	 $$
	 \bm \Grad_\varphi = \left( \varphi, \tilde{R}_\varphi \circ \nabla \circ \tilde{R}_{\varphi^{-1}}(\cdot) \right).
	 $$
	 \end{lem}
	 
	 \begin{proof}
	 Let $\mathcal{U}_{\varphi_0}$ be a chart around $\varphi_0 \in \mathscr{D}^{j(s)}$. After local trivialization, $\bm \div_\varphi \in \mathcal{L}_c(T_{\varphi_0}\mathscr{D}^s, H^{s-1}(M,\mathbb{R}))$ for all $\varphi \in \mathcal{U}$. Therefore, to show that $\varphi \mapsto \bm \div_\varphi$ is smooth it is sufficient to show that for all $A \in T_{\varphi_0}\mathscr{D}^s$, $\varphi \in \mathcal{U}_{\varphi_0} \mapsto \bm \div_\varphi(A)$ is smooth and its derivatives are bounded by constants depending only on the norm of $A$.
	 Let $A \in T_{\varphi_0}\mathscr{D}^s$. By applying Lemma \ref{Lemme calcul diff} with $f = \log \mathrm{det}(T\varphi^{-1})$, we obtain in local coordinates:
	 
	 $$
	 \bm \div_{\varphi}(A) = \mathrm{div}\left( (T\varphi)^{-1}(A) \right).
	 $$
	 Thus, $\bm \div_{p(\varphi)}(A)$ is a rational combination of first and second order derivatives of $\varphi$ multiplied with $A$ and its first order derivatives. The smoothness of $\varphi \mapsto \bm \div_\varphi(A)$ follows from \cite[Lemma 3.2]{Ebin1970}. Moreover, since $A$ and its first order derivative appear in a linear manner, the derivatives of $\varphi \mapsto \bm \div_\varphi(A)$ are bounded by constants depending only on the norm of $A$. 
	 
	 The smoothness of $\bm \Grad$ can be shown in the same way, see \cite[Appendix A, Lemma 2]{EbinMarsden1970} for instance.
	 \end{proof}
	 
	 \begin{lem}
	     The sets $\bm \Grad(\mathscr{D}^{j(s)} \times H^{s+1}(M,\mathbb{R}))$, $\ker(\bm \div)$, $\bm \div(\mathscr{F}^s)$ are smooth subbundles of $\mathscr{F}^s$, $\mathscr{F}^s$ and $\mathscr{D}^{j(s)} \times H^{s-1}(M,\mathbb{R})$ respectively.
	 \end{lem}
	 
	 \begin{proof}
	    The result for $\bm \div$ is a straightforward consequence of \cite[Appendix A, Lemma 1]{EbinMarsden1970} with the following exact sequence:
	    
	    $$
	    \mathscr{F}^s \overset{\bm \div}{\longrightarrow} \mathscr{D}^{j(s)} \times H^{s-1}(M,\mathbb{R}) \overset{\bm \Int}{\longrightarrow} \mathscr{D}^{j(s)} \times \mathbb{R},
	    $$
	    where $\bm \Int : \mathscr{D}^{j(s)} \times H^{s}(M,\mathbb{R}) \rightarrow \mathscr{D}^{j(s)} \times \mathbb{R}$ denote the following smooth bundle morphism:
	     
	     $$
	     \bm \Int_\varphi(f) = \left( \varphi, \int_M f \, \d \vol \right), \quad \forall f \in H^s(M,\mathbb{R}).
	     $$
	     The fact that $\bm \div(\mathscr{F}^s) = \ker(\bm \Int)$ follows from Stokes' theorem for the direct inclusion and by Hodge decomposition for the converse one.
	     The result for $\bm \Grad$ can be shown by duality and by applying \cite[Appendix A, Lemma 4]{EbinMarsden1970}.
	 \end{proof}

\begin{thm}\label{Dérivées covariantes projection horizontale et verticale} 
        The projections $P_\H^*$ and $P_\V^*$ are both elements of $\mathrm{Hom}^\infty_{\mathrm{id}}(\mathscr{F}^s, \mathscr{F}^s)$. Moreover, for $\varphi \in \mathscr{D}^{j(s)}$ and $ Z\circ \varphi \in \H_\varphi\mathscr{D}^{j(s)}$, the following holds:
	    
	    \begin{align}
	        &(\bm \nabla^{*} P^*_{\H_\varphi}) \cdot (Z \circ \varphi) = -\mathcal{O}_{ \iota_*( Z \circ \varphi)}\left(P_{\H_{\varphi}}(\cdot)\right) - P_{\H_{\varphi}}\left(\mathcal{O}^\top_{\iota_*(Z \circ \varphi)}(\cdot) \right), \label{E 56} \\
	        & (\bm \nabla^{*} P^*_{\V_\varphi}) \cdot (Z \circ \varphi) = \mathcal{O}_{\iota_*(Z \circ \varphi)}\left(P_{\H_{\varphi}}(\cdot)\right)  + P_{\H_{\varphi}}\left(\mathcal{O}^\top_{\iota_*(Z \circ \varphi)}(\cdot) \right) .\label{E 57}
	    \end{align}
\end{thm}

\begin{proof}
By Hodge decomposition, the horizontal projection $P_{\H}^* : \mathscr{F}^s \rightarrow \mathscr{F}^s$ can be written in the following form:

$$
P_{\H_\varphi}^* = \bm  \Upsilon_\varphi \circ \bm \div_\varphi,
$$
with $\bm \Upsilon : \mathscr{F}^s \rightarrow \mathscr{F}^s$ defined by:

$$
\bm \Upsilon_\varphi = \tilde{R}_\varphi \circ \nabla \circ \Delta^{-1}_{p(\varphi)} \circ \tilde{R}_{\varphi^{-1}}.
$$
It follows from Hodge decomposition that $\bm \div : \bm \Grad(\mathscr{D}^{j(s)} \times H^{s+1}(M,\mathbb{R})) \rightarrow \bm \div_{p(\varphi)}(\mathscr{F}^s)$ is a bijection. Since $\bm \div$ is a smooth vector bundle morphism, it follows from the inverse mapping theorem that $\Upsilon$ is also a smooth vector bundle morphism. As a consequence, $P_\H^*$ belongs to $\mathrm{Hom}_{\mathrm{id}}^\infty(\mathscr{F}^s, \mathscr{F}^s)$, and therefore, the same holds for $P_\V^* = \mathrm{Id} - P_\H^*$.

The computations of the covariant derivatives are performed in a more general setting in \cite[Proposition 5.18]{gigli2011second} for instance.
\end{proof}

In the following corollary and until the end of the paper, $\bm \nabla$ and $\bm R$ stand respectively for the Levi-Civita connection and the Riemann curvature tensor on $\mathscr{D}$.

\begin{coro}\label{Dérivée covariante Qt}
    Let $k \geq 1$, $\bm Z \in \Gamma^k_S(T \mathscr{D})$ be a $G$-right invariant stratified vector field on $\mathscr{D}$. Then, $\mathcal{O}_{\bm Z}$ and $\mathcal{O}^\top_{\bm Z}$ belong to $\mathrm{Hom}_{\mathrm{id}}^{k-1}(T\mathscr{D}, T\mathscr{D})$. Moreover, 
	     
	     \begin{align*}
	          (\bm\nabla \mathcal{O}_{\bm Z(\varphi)}) \cdot ( Z_2 \circ \varphi) & = \mathcal{O}_{\bm \nabla_{Z_2 \circ \varphi} \bm Z(\varphi)}\left(\cdot  \right)  -  P_{\V_{\varphi}} \left( \bm R( \cdot, Z_2 \circ \varphi ) \bm Z(\varphi) \right)\nonumber  \\
     &+ P_{\H_{\varphi}}\left( \mathcal{O}_{ Z_2\circ \varphi}^\top(\mathcal{O}_{\bm Z(\varphi)}(\cdot) \right) - \mathcal{O}_{ Z_2 \circ \varphi}(  \mathcal{O}_{\bm Z(\varphi)}(\cdot))
	     \end{align*}
	     and
	     
	     \begin{align*}
	          (\bm\nabla \mathcal{O}^\top_{\bm Z(\varphi)}) \cdot ( Z_2 \circ \varphi) & =  \mathcal{O}^\top_{\bm \nabla_{  Z_2 \circ \varphi} \bm Z(\varphi)}\left(\cdot   \right)  -    \bm R( P_{\V_{\varphi}}(\cdot),  Z_2 \circ \varphi ) \bm Z(\varphi) \nonumber  \\
     &+  \mathcal{O}_{Z_2 \circ \varphi}^\top(\mathcal{O}_{\bm Z(\varphi)}(P_{\H_{\varphi}}(\cdot)))  - \mathcal{O}^\top_{Z_2 \circ \varphi}(  \mathcal{O}^\top_{\bm Z(\varphi)}(\cdot)),
	     \end{align*}
	     for any $ Z_2 \circ \varphi \in \H_\varphi \mathscr{D}$.
\end{coro}

\begin{proof}
The fact that $\mathcal{O}_{\bm Z}$ (resp. $\mathcal{O}^\top_{\bm Z}$) belongs to
	 $ \mathrm{Hom}_{\mathrm{id}}^{k-1}(T\mathscr{D},T\mathscr{D})$ is a straightforward consequence of Proposition \ref{P derivee covariante pullback} and Theorem \ref{Dérivées covariantes projection horizontale et verticale} together with the fact that $\mathcal{O}_{\bm Z} = P_{\V} \circ \bm \nabla \bm Z$ (resp. $\mathcal{O}^\top_{\bm Z} =   \bm \nabla \bm Z  \circ P_{\V}$).

	 The computations of the covariant derivative are done in \cite[Theorem 5.21]{gigli2011second} or can be performed by using the chain-rule together with Proposition \ref{P derivee covariante pullback} and Theorem \ref{Dérivées covariantes projection horizontale et verticale}.
\end{proof}

\begin{rem}
    The $G$-right invariance assumption is here to ensure that $\bm O^\top_{\bm Z} = \bm \nabla \bm Z \circ P_{\V}$.
\end{rem}

	\subsection{Existence and uniqueness of the stochastic parallel transport on $\P_\infty$} \label{Subsection 5.2}
	
	Let $\mu \in \P_\infty$ and $(\mu_t)_{t \geq 0}$ be solution to
    
	\begin{equation} \label{E 517}
	    \circ \dw\mu_t = \SN \bar{Z}_i(\mu_t) \circ \d W^i_t + \bar{Z}_0(\mu_t) \d t, \quad \mu_0 =\mu,
	\end{equation}   
	where for all $i=0,1, \dots, N$, $\bar{Z}_i \in \Gamma^3_S(T\P)$. Let $(\Phi_t)_{t \geq 0}$ be the horizontal lift of $(\mu_t)_{t \geq 0}$, namely the solution to
	
	\begin{equation} \label{EDS du lift horizontal}
	\circ \dd\Phi_t(\varphi) = \SN \bm Z_i(\Phi_t(\varphi)) \circ \d W^i_t + \bm Z_0(\Phi_t(\varphi)) \d t, \quad \Phi_0(\varphi) = \varphi \in \mathscr{D}_\mu,
	\end{equation}
	where $\bm Z_i \in \Gamma^2_S(T\mathscr{D})$ is the horizontal lift of $\bar{Z}_i$ for all $i=0,1, \dots, N$. We will also denote by $(\varphi_t)_{t \geq 0}$ the process $(\Phi_t(\varphi))_{t \geq 0}$.
	
	To prove the existence of the stochastic parallel transport on $\P_\infty$, we aim to replicate the argument of Subsection \ref{Subsection 2.2} in the context of our infinite dimensional Riemannian submersion $p : \mathscr{D} \rightarrow \P_\infty$. To this end, the infinite-dimensional equivalent of \eqref{Horizontal parallel transport finite dimension} is needed. The comparison between Proposition \ref{Connexion Levi-Civita Horizontale dim finie} and Proposition \ref{Prop connexion levi civita releve} leads us to the idea that the normal tensor $\mathcal{O}$ plays the same role as the O'Neill tensor. Consequently, the infinite dimensional counterpart of \eqref{Horizontal parallel transport finite dimension} is given by:
	
	\begin{equation} \label{equation transport parallele horizontal sur les difféos} 
	    \circ \bm D_t \mathcal{T}_{0,t}^h(U \circ \varphi) = -\mathcal{O}_{\circ \dd \varphi_t}(\mathcal{T}_{0,t}^h(U \circ \varphi)), \quad \mathcal{T}_{0,0}^h(U \circ \varphi) = U\circ \varphi.
	\end{equation}

	    \begin{prop} \label{Existence et unicité transport parallele horizontal}
    		The $T\mathscr{D}$-valued covariant SDE
    		
        \begin{equation*} 
	        \circ \bm D_t \mathcal{T}_{0,t}^h(U \circ \varphi) = -\SN \mathcal{O}_{\bm Z_i(\varphi_t)}(\mathcal{T}_{0,t}^h(U \circ \varphi)) \circ \d W^i_t -\mathcal{O}_{\bm Z_0(\varphi_t)}(\mathcal{T}_{0,t}^h(U \circ \varphi)) \d t, \quad \mathcal{T}_{0,0}^h(U \circ \varphi) = U\circ \varphi \in T_\varphi \mathscr{D}.
	    \end{equation*}
    		admits a unique solution. Moreover, if $U \circ \varphi \in \H_\varphi \mathscr{D}$, $\mathcal{T}_{0,t}^h(U \circ \varphi)$ remains horizontal for all $t \geq 0$.
    \end{prop}

    \begin{proof}
    By Corollary \ref{Dérivée covariante Qt}, $\mathcal{O}_{\bm Z_i} \in \mathrm{Hom}_{\mathrm{id}}^1(T\mathscr{D},T\mathscr{D})$. The existence and uniqueness follow from Proposition \ref{P Existence unicité ILH covariante ILH}. Let us show that for a horizontal initial condition, the horizontal transport is actually a horizontal process. Let $U \circ \varphi \in \H_{\varphi} \mathscr{D} $ be a horizontal vector, $\bm U_t = \mathcal{T}_{0,t}^h(U \circ \varphi)$ and $\bm U_t^\V = P_{\V_{\varphi_t}}(\bm U_t)$. By Proposition \ref{Dérivées covariantes projection horizontale et verticale} $P_\V$ is a smooth bundle morphism, thus, Proposition \ref{Vector bundle chain rule ILH} yields: 
	
	$$
	\circ \bm D_t \bm U_t^\V = P_{\H_{\varphi_t}}(\mathcal{O}^\top_{\circ \dd \varphi_t}(\bm U_t^\V)) - \mathcal{O}_{\circ \dd \varphi_t}(\bm U_t^\V), \quad \bm U_0^\V = 0.
	$$
	Since $U_0^\V = 0$, $(\bm U_t^\V)_t = (0_{T_{\varphi_t}\mathscr{D}})_t$ is a solution. The conclusion follows from uniqueness of the solution to this covariant SDE.

    \end{proof}

     The following proposition is the infinite-dimensional analogous of Proposition \ref{P Equivariance transport paralelle horizontal dim finie}.

	\begin{prop} \label{P Equivariance transport horizontal D}
		Let $\varphi \in \mathscr{D}$ and $(\mathcal{T}_{0,t}^h(\varphi, \cdot))_{t \geq 0}$ be the unique solution to 
	
	\begin{equation}  \label{E Pour montrer l'unicité}
	    \circ \bm D_t \mathcal{T}_{0,t}^h(\varphi,A \circ \varphi) = -\mathcal{O}_{\circ \dd\Phi_t(\varphi)}(\mathcal{T}_{0,t}^h(\varphi,A \circ \varphi) ), \quad \mathcal{T}_{0,0}^h(A \circ \varphi) = A \circ \varphi, \quad \forall A \circ \varphi \in \H_\varphi \mathscr{D},
	\end{equation}
	along $(\Phi_t(\varphi))_{t \geq 0}$ given by Proposition \ref{Existence et unicité transport parallele horizontal}. Then, $(\mathcal{T}_{0,t}^h)_{t \geq 0}$ is $G$-right invariant, namely:
	
	\begin{equation*}
	\bm T R_g \left( \mathcal{T}_{0,t}^h(\varphi, A \circ \varphi ) \right) = \mathcal{T}_{0,t}^h (\varphi \cdot g, A \circ (\varphi \cdot g) ), \quad \forall A \circ \varphi \in \H_\varphi \mathscr{D}, \, \forall g \in G.
	\end{equation*}

	\end{prop}
	
	\begin{proof}
		Let $\varphi \in \mathscr{D}$ and $A \circ \varphi \in \bar{\H}_\varphi \mathscr{D}$. First, notice that by Proposition \ref{Vector bundle chain rule ILH} and the first point of Proposition \ref{Invariance à droite connexion levi-civita difféos} we have:

	\begin{equation*}
	\circ \bm D_t \bm T R_g ( \mathcal{T}_{0,t}^h(\varphi, A \circ \varphi) ) = \bm T R_g ( \circ \bm D_t \mathcal{T}_{0,t}^h(\varphi, A \circ \varphi)).
	\end{equation*}
	Then, since
	
	$$
	\circ \bm D_t \mathcal{T}_{0,t}^h(\varphi, A \circ \varphi) = -\mathcal{O}_{\circ \dd \Phi_t(\varphi)}( \mathcal{T}_{0,t}^h(\varphi, A \circ \varphi) ),
	$$
	and since, for all horizontal vector field $(\bm U_t)_{t \geq 0}$ along $(\Phi_t(\varphi))_{t \geq 0}$ we have:
	
	\begin{align}
	\bm T R_g (\mathcal{O}_{\circ \dd \Phi_t(\varphi)}( \mathcal{T}_{0,t}^h(\varphi, A \circ \varphi) )) \nonumber & = \bm T R_g\left( P_{\V_{\Phi_t(\varphi)}} \left(  \bm \nabla_{\bm U_t} \circ \dd \Phi_t(\varphi) \right) \right) \nonumber \\
	    & = P_{\V_{\Phi_t(\varphi \cdot g)}} \left( \bm T R_g \left(  \bm \nabla_{\bm U_t} \circ \dd \Phi_t(\varphi) \right) \right) \label{E57} \\
	    & = P_{\V_{\Phi_t(\varphi \cdot g)}} \left(    \bm \nabla_{\bm T R_g(\bm U_t)} \bm T R_g(\circ \dd \Phi_t(\varphi)) \right)\label{E58} \\
	    & = \mathcal{O}_{\circ \dd \Phi_t(\varphi \cdot g)} \left( \bm T R_g(\bm U_t)  \right) \nonumber
	\end{align}
	where \eqref{E57} is due to the equivariance of $(\Phi_t)_{t \geq 0}$ and to the fact that $\bm T R_g(P_{\V_\psi}) = P_{\V_{ \psi \cdot g}}(\bm T R_g)$ for all $\psi \in \mathscr{D}$ (see Proposition \ref{P Propriétés invariances à gauche et à droite des fibrés}) and \eqref{E58} is due to Proposition \ref{Invariance à droite connexion levi-civita difféos}. Since $(\mathcal{T}_{0,t}^h(\varphi, A \circ \varphi))_{t \geq 0}$ is horizontal by Proposition \ref{Existence et unicité transport parallele horizontal}, we obtained that:
	
	\begin{equation*}
	\circ \bm D_t \bm T R_g ( \mathcal{T}_{0,t}^h(\varphi, A \circ \varphi) ) = -\mathcal{O}_{\circ \dd \Phi_t(\varphi \cdot g)}\left(\bm T R_g(\mathcal{T}_{0,t}^h(\varphi, A \circ \varphi)) \right).
	\end{equation*}
	The conclusion follows by uniqueness of the solution to \eqref{E Pour montrer l'unicité}.
	\end{proof}
    
    Now we can prove the main theorem of this section (the infinite-dimensional analogous of Theorem \ref{Transport parallel en dim finie}), namely, the existence and uniqueness of the stochastic parallel transport along a diffusion on $\P_\infty$. 
    \begin{thm} \label{Existence unicité transport parallèle}
    Let $(\mu_t)_{t \geq 0}$ be solution to \eqref{E 517} and $(\Phi_t)_{t \geq 0}$ be the horizontal lift of $(\mu_t)_{t \geq 0}$, namely the unique solution to \eqref{EDS du lift horizontal}. Let $(\mathcal{T}_{0,t}^h(\varphi, \cdot))_{t \geq 0}$ be the unique solution to 
	
	\begin{equation*}
	    \circ \bm D_t \mathcal{T}_{0,t}^h(\varphi,U \circ \varphi) = -\mathcal{O}_{\circ \dd\Phi_t(\varphi)}(\mathcal{T}_{0,t}^h(\varphi,U\circ \varphi)), \quad \mathcal{T}_{0,0}^h(U\circ \varphi) = U\circ \varphi, \quad U\circ \varphi \in \H_\varphi \mathscr{D},
	\end{equation*}
	along $(\Phi_t(\varphi))_{t \geq 0}$ given by Proposition \ref{Existence et unicité transport parallele horizontal}. Then, the process $(\mathcal{T}_{0,t})_{t \geq 0}$ defined by
    
    \begin{equation} \label{E Transport parallèle P}
    \mathcal{T}_{0,t}(v) =\bm T_{\varphi_t}p \left( \mathcal{T}_{0,t}^h(\varphi,\h_\varphi(v))  \right), \quad v \in T_\mu \P_\infty, \quad \varphi \in \mathscr{D}_\mu
    \end{equation}
    does not depend on the element of the fiber $\varphi$ and is the stochastic parallel transport along $(\mu_t)_{t\geq 0}$. 
    
    \end{thm} 
    
    \begin{proof}
	Let $(\Phi_t(\varphi))_t$ be solution to \eqref{EDS du lift horizontal} with $\varphi \in \mathscr{D}$. To prove that the quantity defined by \eqref{E Transport parallèle P} is well defined, notice that, by Proposition \ref{P Equivariance transport horizontal D}, we have
	
	$$
	\mathcal{T}_{0,t}^h(\varphi \cdot g, A \circ (\varphi \cdot g)) = \bm T R_g \left(\mathcal{T}_{0,t}^h(\varphi , A \circ \varphi) \right).
	$$
	In particular, since $\mathfrak{h}_{\varphi \cdot g}(v) = \mathfrak{h}_\varphi(v) \circ g$, we have
	
	$$
	\bm T_{\Phi_t(\varphi \cdot g)}p \left( \mathcal{T}_{0,t}^h(\varphi \cdot g,\h_{\varphi \cdot g}(v))  \right) = \bm T_{\Phi_t(\varphi \cdot g)}p \left(\bm T R_g \left(\mathcal{T}_{0,t}^h(\varphi , \mathfrak{h}_{\varphi}(v) \right) \right) = \bm  T_{\Phi_t(\varphi)}p(\mathcal{T}_{0,t}^h(\varphi, \mathfrak{h}_\varphi(v)).
	$$
	Thus, the process $(\mathcal{T}_{0,t}(v))_{t \geq 0}$ defined by \eqref{E Transport parallèle P} does not depend on the element of the fiber $\varphi \in \mathscr{D}_\mu$ we chose.
	The stochastic parallel transport property of $(\mathcal{T}_{0,t})_{t \geq 0}$ follows directly from its construction.

    \end{proof} 
    
    \begin{rem}[Isometry property of the stochastic parallel transport]\label{Remarque unicité et isométrie du transport para}
    The stochastic parallel transport along a diffusion $(\mu_t)_{t \geq 0}$ is an isometry. Indeed, let $(\bar{A}_t)_{t \geq 0}$, $(\bar{B}_t)_{t \geq 0}$ be two stochastic parallel transports along a semimartingale $(\mu_t)_{t \geq 0}$ starting from $\mu \in \P_\infty$, and let $(\bm A_t)_{t \geq 0}$, $(\bm B_t)_{t \geq 0}$ be their horizontal lifts along $(\Phi_t(\varphi))_{t \geq 0}$ with $\varphi \in \mathscr{D}_\mu$. Then, the following holds:
    
    $$
    \circ \d \langle \bar{A}_t, \bar{B}_t \rangle_{\mu_t} = \circ \d \langle \bm A_t, \bm B_t \rangle_{\Phi_t(\varphi)} =  \left \langle \circ \bm D_t \bm A_t, \bm B_t \right \rangle_{\Phi_t(\varphi)} + \left \langle \bm A_t, \circ \bm D_t \bm B_t \right \rangle_{\Phi_t(\varphi)} =  0.
    $$
\end{rem}

    \begin{rem}[Parallel transport along flows of purely atomic measures] In \cite[Example 4.16]{gigli2011second}, the author observes that the limit of a (deterministic) parallel transport may fail to be a parallel transport. This is essentially due to the fact that, for a purely atomic measure $\nu$, the tangent space takes the form
    
    \begin{equation} \label{Tangent space purely atomic}
        \T_\nu \P = L^2(M, TM, \nu).
    \end{equation}
    More precisely, let $(\varphi_t)_{t \geq 0}$ be a smooth path in $\mathscr{D}$, $\mu \in \P_\infty$ and $(\mu^n)_n$ a sequence of purely atomic measures converging to $\mu$. Let also $\mu_t := (\varphi_t)_* \mu$ and $\mu^n_t := (\varphi_t)_* \mu^n$ for all $n \geq 0$. For all $v \in \T_\mu \P$, we can always find a sequence $(v_n)_{n \geq 0}$ such that $v_n \in \T_{\mu_n} \P$ for all $n \geq 0$ strongly converging (see \cite[Definition 1.8]{gigli2011second}) to $v$. Because of \eqref{Tangent space purely atomic}, it is easy to see that the parallel transport of $v_n$ along $(\mu^n_t)_{t \geq 0}$ is given by $\mathcal{T}_{0,t}(v_n) = \sslash_{0,t} v_n \circ \varphi_t^{-1}$, where $\sslash_{0,t}$ is the parallel transport along $(\varphi_t)_{t \geq 0}$. Consequently, when letting $n \to + \infty$, it can be easily seen that $\sslash_{0,t}(v_n)$ strongly converges to $\sslash_{0,t}(v)$ but not to $\mathcal{T}_{0,t}(v)$. The same result also holds in the stochastic case.
        
    \end{rem}

    \vspace{10pt}
    
    \noindent \textbf{Equivariant lift of the stochastic parallel transport.} We finish this section by stating the infinite-dimensional equivalent of Proposition \ref{P Processus equivariant TD}. The \textit{Lie algebra} $\mathfrak{g}$ of $G$ is defined to be
    
    \begin{equation*}
        \mathfrak{g} = \V_{\mathrm{id}} \mathscr{D}.
    \end{equation*}
    
    \begin{defi} \label{D Défi adjoint}
	Let $g \in G$, the \textit{adjoint action} of $g$ is the map $\mathrm{Ad}(g) : \mathfrak{g} \rightarrow \mathfrak{g}$ is the differential of the inner automorphism $ \varphi \mapsto g \cdot \varphi \cdot g^{-1}$ in $\varphi = \mathrm{id}$.
    \end{defi}
    
    Let $\varpi$ denote the connection form on $\mathscr{D}$, namely, $\varpi(A \circ \varphi) = \bm T L_{\varphi^{-1}}(P_{\V_\varphi}(A \circ \varphi))$ for all $A \circ \varphi \in T_\varphi \mathscr{D}$. One can easily check that
    
    \begin{equation*}
	    \varpi(\bm T R_g(A \circ \varphi)) = \mathrm{Ad}(g^{-1}) \varpi(A \circ \varphi).
	\end{equation*}
    
	As in the finite-dimensional case (see Remark \ref{R Equivariance of the horizontal stochastic parallel transport dim finie}), $\bm T p|_{\H \mathscr{D}} : \H\mathscr{D} \rightarrow T \P_\infty$ can be seen as a $G$-principal bundle where the action of $G$ on $\H \mathscr{D}$ is given $\bm T R_g$ for $g \in G$. In particular, Proposition \ref{P Equivariance transport horizontal D} can be interpreted as follows: the process $(\mathcal{T}_{0,t}^h)_{t \geq 0}$ is equivariant on the principal bundle $\H\mathscr{D}$.
	Note that we can also consider the group $G \rtimes \mathfrak{g}$ and its action on $T \mathscr{D}$ (the infinite-dimensional equivalent of \eqref{E Action de groupe}) given by
	
	\begin{equation*}
	\left( \varphi, A \circ \varphi \right) \cdot \left( g, \mathcal{Y} \right) = \left( \varphi \cdot g, A \circ (\varphi \cdot g) + \bm T R_g \bm T L_{\varphi}  \left( \mathrm{Ad}_g(\mathcal{Y}) \right) \right),
	\end{equation*}
	for $\varphi \in \mathscr{D}$, $A \circ \varphi \in T_\varphi \mathscr{D}$, $g \in G$ and $\mathcal{Y} \in \mathfrak{g}$. This action provides the tangent bundle $T \mathscr{D}$ with a natural principal bundle structure above $T \P_\infty$. We can, as in Proposition \ref{P Processus equivariant TD}, construct a lift of the horizontal transport which is equivariant on the principal bundle $(T\mathscr{D}, G \rtimes \mathfrak{g})$.

	\begin{prop}[Equivariant lift of the stochastic parallel transport] \label{P Transport equivariant}
	    Let $(\Phi_t)_{t \geq 0}$ be the horizontal lift of a $\P_\infty$-valued diffusion $(\mu_t)_{t \geq 0}$. For all $\varphi \in \mathscr{D}_{\mu_0}$ we define $\tilde{\mathcal{T}}_{0,t}(\varphi, \cdot)$ to be the stochastic flow of linear maps $\tilde{\mathcal{T}}_{0,t}(\varphi, \cdot) : T_{\Phi_0(\varphi)} \mathscr{D} \rightarrow T_{\Phi_t(\varphi)} \mathscr{D}$ such that
	    
	    \[
\begin{cases}
	        & \tilde{\mathcal{T}}_{0,t}(u,Z \circ \varphi) = \mathcal{T}_{0,t}^h(\varphi,Z \circ \varphi) \text{ and } \tilde{\mathcal{T}}_{0,t}(\varphi,Y \circ \varphi) =\bm T L_{\Phi_t(\varphi)}(\varpi(Y \circ \varphi)), \quad \forall Z \circ \varphi \in \H_\varphi \mathscr{D}, \, \forall Y \circ \varphi \in \V_\varphi \mathscr{D}, \\
	        & \tilde{\mathcal{T}}_{0,0}(u, \cdot)= \mathrm{id}_{T_{\Phi_0(\varphi)} \mathscr{D}}. \nonumber
	    \end{cases} \]
	    This process is equivariant on $T \mathscr{D}$ seen as a $G \rtimes \mathfrak{g}$ principal bundle, namely,
	    
	    \begin{equation*}
	        \tilde{\mathcal{T}}_{0,t}(\varphi,A\circ \varphi) \cdot (g, \mathcal{Y}) = \tilde{\mathcal{T}}_{0,t}( (\varphi, A \circ \varphi) \cdot (g,\mathcal{Y})), \quad \forall A \circ \varphi \in \T \mathscr{D}, \, \forall (g, \mathcal{Y}) \in G \rtimes \bar{\mathfrak{g}}.
	    \end{equation*}
	    Moreover, $(\tilde{\mathcal{T}}_{0,t})_{t \geq 0}$ is a lift of $(\bar{\mathcal{T}}_{0,t})_{t \geq 0}$, i.e. for any $Z \in T_{\mu_0} \P_\infty$:
        
        \begin{equation*} 
             \bm T p (\tilde{\mathcal{T}}_{0,t}(\varphi,A \circ \varphi))= \mathcal{T}_{0,t}(Z) , \quad  \forall A \circ \varphi \in \bm T p^{-1}(Z).
        \end{equation*}
	    
	\end{prop}
	
	\begin{proof}
	The proof can be performed in the exact same way as in the finite-dimensional case, see the proof of Proposition \ref{P Processus equivariant TD}.
	\end{proof}

\section{Decomposition of equivariant diffusions on $\mathscr{D}$} \label{Section 6}

    In this section we consider equivariant diffusions on $\mathscr{D}$, namely diffusions $(\Phi_t)_{t \geq 0}$ satisfying the following property:
    
    $$
    \Phi_t(\varphi \cdot g) = \Phi_t(\varphi) \cdot g, \quad \forall \varphi \in \mathscr{D}, \, \forall g \in G, \, \forall t \geq 0.
    $$
	
	A diffusion $(\Phi_t)_{t \geq 0}$ is said to be \textit{horizontal} (resp. \textit{vertical}) if $\circ \dd \Phi_t$ is horizontal (resp. vertical) for all $t \geq 0$. In particular, if $(\Phi_t)_{t \geq 0}$ is vertical, it lies in the same fiber as the initial condition. Note that the choice to write the considered SDE in Stratonovich's form is not insignificant since a Stratonovich's derivative of a process holds, at least formally, in the tangent space, which is not the case of the Itô's derivative because of the Hessian term.
	
	Since $\mathscr{D}$ has a structure of $G$-principal bundle over $\P_\infty$, a path $(\varphi_t)_{t \geq 0}$ writes in a unique way as
	
	\begin{equation} \label{Decompo fibré principal D}
	    \varphi_t = h_t \cdot g_t,
	\end{equation}
	where $(h_t)_{t \geq 0}$ is the horizontal lift of the projection $(p(\varphi_t))_{t \geq 0}$ and $(g_t)_{t \geq 0}$ is a path in $G$ with $g_0=id$.
	
	We show that, in the case of an equivariant diffusion $(\Phi_t)_{t \geq 0}$, the decomposition \eqref{Decompo fibré principal D} enjoys good properties. More precisely, the process $(h_t)_{t \geq 0}$ is an equivariant horizontal diffusion and the horizontal lift of a diffusion on $\P_\infty$ and $(g_t)_{t \geq 0}$ is the right exponential of a stochastic process in the Lie Algebra $\mathfrak{g}$ of $G$.
We prove this result for $(\Phi_t)_{t \geq 0}$ being solution to an ODE on $\mathscr{D}$, the proof being exactly the same when one comes to consider Stratonovich's SDEs. In the beginning of this section we will restrict ourselves to the case where $\mathscr{D}$ is a principal bundle over $\P_\infty$. In particular, for all $\varphi \in \mathscr{D}$ we denote by $\iota_{\varphi} : G \rightarrow \mathscr{D}$ the left multiplication $\iota_{\varphi}(g) = \varphi \cdot g$ for all $g \in G$. There also exist a connection form (given by $\varpi(A \circ \varphi) = \bm T L_{\varphi^{-1}}(P_{\V_\varphi}(A \circ \varphi))$ for $A \circ \varphi \in T_\varphi \mathscr{D}$). Recall that, in this case, for all $\varphi \in \mathscr{D}$, $g \in G$ and $\bm Y$ right invariant vector fields we have

$$
\varpi(\bm T R_g(Y \circ \varphi)) = \mathrm{Ad}(g^{-1})\left(\varpi(Y\circ \varphi) \right),
$$
where $\mathrm{Ad}$ is defined in Definition \ref{D Défi adjoint}.

The choice to see $\mathscr{D}$ merely as a principal bundle over $\P_\infty$ emphasize the fact that these results are not specific to the group of diffeomorphisms over $\P_\infty$ setting, but can be readily extended to any setting where a structure of principal bundle is involved.

At the end of the section, we will come back to the case of $\mathscr{D}$ seen as a Riemannian submersion onto $\P_\infty$ to study a specific case of Itô's equivariant equations. As in the previous section, the finite-dimensional case is treated in Subsections \ref{subsection 2.3} and \ref{Subsection 2.4}.

    The following proposition is the infinite-dimensional equivalent of Proposition \ref{P D1}.
    \begin{prop}
    Let $\bm A \in \Gamma^1_S(T\mathscr{D})$ with $\bm A = \bm Z \oplus \bm Y \in \Gamma^1_S(\H \mathscr{D}) \oplus \Gamma^1_S(\V \mathscr{D})$ be a right invariant vector field for the action of $G$ on $\mathscr{D}$. Let $(\Phi_t)_{t \geq 0}$ be the flow of maps solution to
    
    \begin{equation} \label{EDO equivariante D}
    \dot{\Phi}_t(\varphi) = \bm A (\Phi_t(\varphi)), \quad \Phi_0(\varphi) \in \mathscr{D},
    \end{equation}
    with $\Phi_0$ an equivariant map and let $M_t = p(\Phi_t)$ for all $t \geq 0$. Then, $(M_t)_{t\geq 0}$ is solution to the following autonomous equation
    
    \begin{equation*}
        \dot{M}_t(\mu) = \bar{Z}(M_t(\mu)), \quad M_0(\mu) = p(\Phi_0(\varphi)), \, \forall \varphi \in \mathscr{D}_\mu,
    \end{equation*}
    with $\bar{Z} = \bm T p(\bm Z)$. Moreover, the flow $(\Phi_t)_{t \geq 0}$ is equivariant and the decomposition \eqref{Decompo fibré principal D} takes the form $\Phi_t = h_t \cdot g_t$, where $(h_t)_{t \geq 0}$ is the horizontal lift of $(M_t)_{t \geq 0}$ and $(g_t)_{t \geq 0}$ is vertical such that
    
    \begin{equation*}
    \dot{h}_t(\varphi) = \h_{h_t(\varphi)}\left( \dot{M}_t(p(\varphi)) \right) = \bm Z(h_t(\varphi)) , \quad h_0(\varphi) = \Phi_0(\varphi)
    \end{equation*}
    and 
    
    \begin{equation} \label{E g_t}
    \dot{g}_t(\varphi) = \bm T R_{g_t(\varphi)} \varpi(\bm Y(h_t(\varphi))), \quad g_0(\varphi) = \mathrm{id}.
    \end{equation}
    In particular, $(h_t)_{t \geq 0}$ is a horizontal equivariant autonomous flow and $(g_t)_{t \geq 0}$ satisfies $g_t(\varphi \cdot g) = g^{-1} g_t(\varphi) \cdot g$ for all $g \in G$ and is a right exponential of a path in the Lie algebra $\mathfrak{g}$ of $G$, namely:
    
    \begin{equation} \label{E exponentielle droite deterministe D}
    g_t(\varphi) = \mathcal{E}^R\left( \int_0^\cdot \varpi(\bm Y(h_s(\varphi))) \d s \right)_t.
    \end{equation}
    \end{prop}
    
    \begin{proof}
    Let $f \in \C^\infty(M)$. Since $\bm A$ is right invariant, this is also the case for $\bm Z$ and $\bm Y$. In particular, $\bm Z(\psi)$ writes as $Z(p(\psi)) \circ \psi$. Then, since $\bm \nabla \bm F_f(\varphi) = \nabla f \circ \varphi$, we obtain by orthogonality:
    
    \begin{align*}
    F_f(M_t(\mu))- F_f(\mu) &= \int_0^t \langle \bm \nabla \bm F_f( \Phi_t(\varphi)), \bm A(\Phi_t(\varphi)) \rangle_{\Phi_t(\varphi)} \\
    &= \int_0^t \langle \bm \nabla \bm F_f( \Phi_t(\varphi)), \bm Z(\Phi_t(\varphi)) \rangle_{\Phi_t(\varphi)} \\
    & = \int_0^t \int_M \langle \nabla f \circ \Phi_s(\varphi), Z(M_s(p(\varphi)) \circ \Phi_s(\varphi) \rangle \d \vol \, \d s \\
    & = \int_0^t \int_M \langle \nabla f(x), Z(M_s(p(\varphi)))(x) \rangle \d M_s(p(\varphi)) \, \d s
    \end{align*}
    In particular, $(M_t)_{t \geq 0}$ is solution to the following autonomous equation:
	
	$$
	\dot{M}_t(p(\varphi)) = \bar{Z} (M_t(p(\varphi))), \quad M_0(p(\varphi)) = p(\varphi).
	$$    
	To see that $(\Phi_t)_{t \geq 0}$ is equivariant, note that, by $G$-right invariance of $\bm A$, we have:
	
	$$
	\frac{\d}{\d t} \left( \Phi_t(\varphi) \cdot g \right) = \bm T R_g( \dot{\Phi}_t(\varphi)) = A(\Phi_t(\varphi) \cdot g).
	$$
	By uniqueness of the solution to \eqref{EDO equivariante D}, we obtain that $\Phi_t(\varphi \cdot g) = \Phi_t(\varphi) \cdot g$.
	
    Let $(h_t)_{t \geq 0}$ be the unique solution to
    
    $$
    \dot{h}_t(\varphi) = \h_{h_t(\varphi)}\left( \dot{M}_t(p(\varphi)) \right) = \bm Z(h_t(\varphi)) , \quad h_0(\varphi) = \Phi_0(\varphi).
    $$   
    Since $(M_t)_{t \geq 0}$ is autonomous, it is clear that $(h_t)_{t \geq 0}$ is an equivariant autonomous horizontal flow. Moreover, since $\Phi_t$ and $h_t$ are on the same fibers and that $G$ acts freely and transitively on the fibers, there exists a stochastic process $(g_t)_{t \geq 0}$ such that

$$
\Phi_t = h_t \cdot g_t.
$$    
	Observe that $(h_t)_{t \geq 0}$ is equivariant as a horizontal lift of a flow on $\P_\infty$ (see Proposition \ref{Prop equivariance horizontal lift}). Then, we have
    
    $$
	\Phi_t(\varphi \cdot g) = h_t(\varphi \cdot g) \cdot g_t(\varphi \cdot g) \text{ and } \Phi_t(\varphi \cdot g) = \Phi_t(\varphi) \cdot g = h_t(\varphi) \cdot g_t(\varphi) \cdot g.
    $$
    Hence, since $G$ acts freely on the fibers of $\mathscr{D}$ we get $g_t(\varphi \cdot g) = g^{-1} \cdot g_t(\varphi) \cdot g$ for all $t \geq 0$.
    Since $\bm A$ is right invariant for the action of $G$, it is also the case for its horizontal and vertical parts. Taking the time derivative yields
    
    $$
    \dot{\Phi}_t = \bm T R_{g_t}(\dot{h}_t) + T\iota_{h_t}(\dot{g}_t).
    $$
    Identifying the horizontal and vertical parts we obtain that $\bm T R_{g_t}(\dot{h}_t)= \bm Z (\Phi_t)$ and
    
    $$
    T\iota_{h_t}\dot{g}_t = \bm Y(\Phi_t) = T\iota_{\Phi_t} \varpi(\bm Y( \Phi_t)).
    $$
    Using the properties of the connection form, this yields:
    
    $$
    T\iota_{h_t}\dot{g}_t = T\iota_{h_t \cdot g_t} \varpi (\bm T R_{g_t}(\bm Y(h_t))) = T\iota_{h_t} \bm T L_{g_t} \mathrm{Ad}(g_t^{-1}) \varpi (\bm Y(h_t)) = T\iota_{h_t} \bm T R_{g_t} \varpi( \bm Y(h_t))  .
    $$
    Formula \eqref{E exponentielle droite deterministe D} is a direct consequence of \eqref{E g_t}.
    
    \end{proof}
    
    Following the exact same arguments in the case of Stratonovich's SDEs we get the following proposition, which is the infinite-dimensional equivalent of Proposition \ref{P D2}.
    
\begin{prop} \label{Proposition diffusion equivariantes D}
	Let $\bm A_0, \bm A_1, \dots, \bm A_N \in \Gamma^2_S(T\mathscr{D})$ be right invariant vector fields with $\bm A_i = \bm Z_i + \bm Y_i \in \Gamma^2_S(\H \mathscr{D}) \oplus \Gamma^2_S(\V \mathscr{D})$ for all $i=0,1, \dots, N$. Let $(\Phi_t)_{t \geq 0}$ be the diffusion on $\mathscr{D}$ solution to
	
	$$
	\circ \dd \Phi_t = \SN \bm A_i(\Phi_t) \circ \d W^i_t + \bm A_0(\Phi_t) \d t, \quad \Phi_0(\varphi) \in \mathscr{D}
	$$
	with $\Phi_0$ an equivariant map and let $M_t = p(\Phi_t)$ for all $t \geq 0$. Then, $(M_t)_{t \geq 0}$ is a diffusion solution to the following Stratonovich's SDE
	
	\begin{equation*}
	\circ \dw M_t = \SN \bar{Z}_i(M_t) \circ \d W^i_t + \bar{Z}_0(M_t) \d t, \quad M_0(\mu) =  p(\Phi_0(\varphi)), \, \forall \varphi \in \mathscr{D}_\mu
	\end{equation*}
	 Moreover, the stochastic flow $(\Phi_t)_{t \geq 0}$ is equivariant and the decomposition \eqref{Decompo fibré principal D} takes the form $\Phi_t = h_t \cdot g_t$ where $(h_t)_{t \geq 0}$ is the horizontal lift of $(M_t)_{t \geq 0}$ and $(g_t)_{t \geq 0}$ is a vertical process such that for all $\varphi \in \mathscr{D}$:
	
	\begin{equation*}
	\circ \dd h_t(\varphi) = \h_{h_t(\varphi)}(\circ \d M_t(p(\varphi))) = \bm Z(h_t(\varphi)), \quad h_0(\varphi) = \Phi_0(\varphi),
	\end{equation*}
	and
	
	\begin{equation*}
	\circ \dd g_t(\varphi) = \SN \bm T R_{g_t(\varphi)} \varpi( \bm Y_i(h_t(\varphi)))  \circ \d W^i_t + \bm T R_{g_t(\varphi)} \varpi( \bm Y_0(h_t(\varphi))) \d t, \quad g_0(\varphi) = \mathrm{id}.
	\end{equation*}
	In particular, $(h_t)_{t \geq 0}$ is a horizontal equivariant diffusion and $(g_t)_{t \geq 0}$ satisfies $g_t(\varphi \cdot g) = g^{-1} g_t(\varphi) \cdot g$ for all $g \in G$ and is a right stochastic exponential of a diffusion in the Lie algebra $\mathfrak{g}$ of $G$, namely:
	
	\begin{equation*}
	g_t(\varphi) = \mathcal{E}^R \left( \SN \int_0^\cdot \varpi \bm Y_i(h_s(\varphi))) \circ \d W^i_s + \varpi (\bm Y_0(h_s(\varphi))) \d s \right)_t.
	\end{equation*}

\end{prop}

We now return to the case where $\mathscr{D}$ is endowed with the Riemannian structure we already presented in Section $4$. Each fiber is seen as a submanifold of $\mathscr{D}$ fitted with the induced Riemannian structure. The following propositions are respectively the infinite-dimensional analogous of Proposition \ref{P D 4} and Proposition \ref{P D3}.

    \begin{prop} 
    Let $(\Phi_t)_{t \geq 0}$ be a diffusion on $\mathscr{D}$ such that
	
	\begin{equation*}
	\dd^{\nabla}\Phi_t = \SN \bm Z_i(\Phi_t) \d W^i_t + \bm Z_0(\Phi_t) \d t, \quad \Phi_0(\varphi) \in \mathscr{D},
	\end{equation*}
	where $\bm Z_i \in \Gamma^2_S(\H\mathscr{D})$ is a horizontal right invariant vector field for all $i=0,1, \dots, N$ with $\Phi_0$ an equivariant map and let $(M_t)_{t \geq 0} = (p(\Phi_t))_{t \geq 0}$. Then, $(M_t)_{t \geq 0}$ is a diffusion solution to the following Itô's SDE
    
    \begin{equation*} 
        \dw^{\grad} M_t(\mu) = \SN \bar{Z}_i(M_t(\mu)) \d W^i_t + \bar{Z}_0(M_t(\mu)) \d t, \quad M_0(\mu) = p(\Phi_0(\varphi)), \, \varphi \in \mathscr{D}_\mu,
    \end{equation*}    
    where $\bar{Z}_i =\bm T p( \bm Z_i)$ for all $i=0,1,\dots,N$. Moreover, the stochastic flow $(\Phi_t)_{t \geq 0}$ is equivariant and $(\Phi_t)_{t \geq 0}$ is horizontal. In particular, decomposition \eqref{decomposition} writes $\Phi_t = h_t$.
    \end{prop}
    
    \begin{proof}
    By using the Itô-Stratonovich equivalence on $\mathscr{D}$ (Proposition \ref{Equivalence Ito Strato difféos}), we have:
    
    \begin{equation*}
    \circ \dd \Phi_t = \SN \bm Z_i( \Phi_t) \circ \d W^i_t + \bm Z_0(\Phi_t) \d t - \frac{1}{2} \SN \bm \nabla_{\bm Z_i} \bm Z_i (\Phi_t) \d t, \quad \Phi_0(\varphi) \in \mathscr{D}.
    \end{equation*}
    Since $\bm Z_i$ is $G$-right invariant for $i=0,1, \dots, N$, it is also the case for $\bm \nabla_{\bm Z_i} \bm Z_i$ by Proposition \ref{Invariance à droite connexion levi-civita difféos}. Thus, we can apply Proposition \ref{Proposition diffusion equivariantes D}. Since $\varpi(\H \mathscr{D}) = \{0\}$ and $P_{\V}(\bm \nabla_{\bm Z_i} \bm Z_i) = \mathcal{N}_\varphi(\bm Z_i, \bm Z_i) = 0$ the conclusion follows from Proposition \ref{Tenseur normal dans le cas horizontal}.
    \end{proof} 

\begin{prop} \label{P diffusion ito verticale}
	Let $(\Phi_t)_{t \geq 0}$ be a diffusion on $\mathscr{D}$ such that
	
	\begin{equation*}
	\dd^{\bm \nabla}\Phi_t = \SN \bm Y_i(\Phi_t) \d W^i_t + \bm Y_0(\Phi_t) \d t, \quad \Phi_0(\varphi) \in \mathscr{D},
	\end{equation*}
	where $\Phi_0$ is an equivariant map and $\bm Y_i \in \Gamma^2_S(\V\mathscr{D})$ is a vertical $G$-right invariant vector field for all $i=0,1, \dots, N$. Then, $(\Phi_t)_{t \geq 0}$ is an equivariant diffusion. Moreover, the decomposition \eqref{Decompo fibré principal D} is given by a finite variation process $(h_t)_{t \geq 0}$ solution to:
	
	\begin{equation*}
	\dd h_t = -\frac{1}{2} \SN \mathrm{I\!I}^{h_t \cdot G}(\bm Y_i, \bm Y_i)(h_t) \d t, \quad h_0(\varphi) = \Phi_0(\varphi)
	\end{equation*}
	and
	
\begin{align} \label{E g en ito}
    \dd^{\bm \nabla^G} g_t &= \SN \bm T R_{g_t}\varpi\left( \bm Y_i(h_t) \right) \circ \d W^i_t + \bm T R_{g_t}\varpi\left(\bm Y_0(h_t)\right) \d t  \\
    & + \frac{1}{2} \SN \left( \bm T R_{g_t} \left( \bm \nabla^G_{\varpi(\bm Y_i(h_t))} \varpi(\bm Y_i(h_t)) -  \varpi \left( \bm \nabla^{G\cdot h_t}_{\bm Y_i} \bm Y_i (h_t) \right) \right) \right) \d t, \quad g_0(\varphi) = \mathrm{id}.\nonumber
\end{align}

\end{prop}

\begin{proof}
Using the Itô-Stratonovich equivalence on $\mathscr{D}$ (Proposition \ref{Equivalence Ito Strato difféos}), it follows that

$$
\circ \dd\Phi_t = \SN \bm Y_i(\Phi_t) \circ \d W^i_t + \bm Y_0(\Phi_t) \d t - \frac{1}{2} \SN \bm \nabla_{\bm Y_i} \bm Y_i(\Phi_t) \d t.
$$
Since for all $i=0,1,\dots,N$, $\bm Y_i$ is right invariant, this is also the case for $\bm \nabla_{\bm Y_i} \bm Y_i$  and $\mathrm{I\!I}^G(\bm Y_i, \bm Y_i)$ by Proposition \ref{Invariance à droite connexion levi-civita difféos}. Thus, we can apply Proposition \ref{Proposition diffusion equivariantes D}. Using the fact that

$$
P_{\H_{\Phi_t}}(\circ \dd \Phi_t) = - \frac{1}{2} \SN \mathrm{I\!I}^{ h_t \cdot G}(\bm Y_i, \bm Y_i) (\Phi_t) \dd t = \bm T R_{g_t} (\circ \d_t h_t),
$$
we deduce the equation for $(h_t)_{t \geq 0}$. The equation of $(g_t)_{t \geq 0}$ is given by

\begin{align*}
    \circ \dd g_t & = \SN \bm T R_{g_t}\varpi\left( \bm Y_i(h_t) \right) \circ \d W^i_t + \bm T R_{g_t}\varpi\left(\bm Y_0(h_t)\right) \d t - \frac{1}{2} \SN \bm T R_{g_t} \left( \varpi \left( P_{\V_{g_t}} \left( \bm \nabla_{\bm Y_i} \bm Y_i (h_t) \right) \right) \right) \\
    & =\SN \bm T R_{g_t}\varpi\left( \bm Y_i(h_t) \right) \circ \d W^i_t + \bm T R_{g_t}\varpi\left(\bm Y_0(h_t)\right) \d t - \frac{1}{2} \SN \bm T R_{g_t} \left( \varpi \left( \bm \nabla^{G\cdot h_t}_{\bm Y_i} \bm Y_i (h_t) \right) \right).
\end{align*}
Applying Itô-Stratonovich conversion on $G$ (see Remark \ref{Equivalence Ito Strato sur G}), we obtain that the above Stratonovich's SDE is equivalent to the following Itô's one:

\begin{align*}
    \dd^{\bm \nabla^G} g_t &= \SN \bm T R_{g_t}\varpi\left( \bm Y_i(h_t) \right) \circ \d W^i_t + \bm T R_{g_t}\varpi\left(\bm Y_0(h_t)\right) \d t \\
    & + \frac{1}{2} \SN \left( \bm \nabla^G_{\bm T R_{\cdot}\varpi(\bm Y_i(h_t))} \bm T R_{\cdot} \varpi(\bm Y_i(h_t))(g_t) - \bm T R_{g_t} \left( \varpi \left( \bm \nabla^{G\cdot h_t}_{\bm Y_i} \bm Y_i (h_t) \right) \right) \right) \d t.
\end{align*}
We conclude by applying Proposition [\ref{Invariance à droite connexion levi-civita difféos}, (1)] to $\bm \nabla^G$.
\end{proof}

\begin{rem}
    Note that Proposition \ref{P diffusion ito verticale} is slightly different from Proposition \ref{P D3} since $(g_t)_{t \geq 0}$ is presented as a solution of a Itô's SDE rather than a Stratonovich's one. Actually, this is due to the fact that $G$ is both the group acting on $\mathscr{D}$ and the fiber $\mathscr{D}_\vol$. Consequently, we can endow $G$ with a canonical Riemannian structure, which is the one inherited of $\mathscr{D}$. In particular, the term 
    
    $$
    \left( \bm \nabla^G_{\varpi(\bm Y_i(h_t))} \varpi(\bm Y_i(h_t)) -  \varpi \left( \bm \nabla^{G\cdot h_t}_{\bm Y_i} \bm Y_i (h_t) \right) \right),
    $$
    in \eqref{E g en ito} can be interpreted as a difference of volume between the fibers $\mathscr{D}_\vol$ and $\mathscr{D}_{p(h_t)}$.
\end{rem}

\begin{rem}
	As in the finite dimensional case, the process $(\Phi_t)_t$ is vertical if the second fundamental form vanishes. This is the case if and only if the fibers are totally geodesic (see Remark \ref{Remarque fibres totalement géodésiques dim finie}), which is not the case in general. It is due to the fact that the Riemannian exponential depends on the geometry of the variety induced by the metric and is not, in general, the flow of the considered vector field. In \cite{EbinMarsden1970}, the authors show that geodesic in fibers $\mathscr{D}_\mu$ are solutions to incompressible Euler equations. The questions of regularity and existence of solutions for all time $t \geq 0$ to these equations for general manifolds are still open.
\end{rem}

\begin{rem}
	Such decompositions could also be interesting in the context of stochastic partial differential equations since, as said in \cite{EbinMarsden1970}, the group $G$ is the appropriate configuration space for the hydrodynamics of an incompressible fluid, while SDEs with values on $\P_\infty$ give rise to mass-preserving SPDEs by considering the associated densities.
\end{rem}

\appendix 

\section{} 

\subsection{Hodge decomposition} \label{Annexe Hodge}

Let $(M, \langle \bullet, \bullet \rangle)$ be an $m$-dimensional Riemannian manifold, $s>m/2$ and $\mu \in \P$ with a $\C^2$ positive density. Let $\Omega^k(M) = \Lambda^k(T^*M)$ and $\Gamma^{(s)}_\mu(\Omega^k(M))$ (resp. $\Gamma_\mu^{(s)}(TM)$) the vector space of sections of $\Omega^k(M)$ (resp. of $TM$) with regularity $H^s(\mu)$ with $s \in \mathbb{N}$.

\begin{prop}
     The exterior derivative $\d$ admits a unique adjoint $\d_\mu^*$ in $L^2(M,T^*M,\mu)$. Moreover, any differential form $\Gamma^{(s)}_\mu(\Omega^k(M))$ with $k \geq 1$ decomposes in a unique way as 
    
    \begin{equation} \label{E décompo hodge}
        \alpha = \d \omega_1 + \d_\mu^\ast \omega_2 + \omega_3,
    \end{equation}
    where $\omega_1 \in \Gamma^{(s+1)}_\mu(\Omega^{k-1}(M))$, $\omega_2 \in \Gamma^{(s+1)}_\mu(\Omega^{k+1}(M))$ and $\omega_3 \in \Gamma^{(s)}_\mu(\Omega^k(M))$ is a harmonic form.
\end{prop}

\begin{proof}
Let $\rho = \d \mu / \d \vol$ and $\xi \in \C^2(M)$ such that $\xi^{2+m} = \rho$. The volume measure of the Riemannian manifold $(M, \xi^2 \langle \bullet, \bullet \rangle)$ is given by $\xi^m \d \vol$. Consequently, $\d_\mu^*$ can be constructed in the usual way as the adjoint of $\d$ on the Riemannian manifold $(M, \xi^2\langle \bullet, \bullet \rangle)$. In particular, the Hodge decomposition on a Riemannian manifold (see \cite{Wells1980} for instance) gives us \eqref{E décompo hodge}.
\end{proof}

Let $\sharp$ denote the musical isomorphism, we define the operator $\mathrm{div}_\mu : \Gamma^{(s+1)}(\Omega^{k+1}(M)) \rightarrow \Gamma^{(s)}(\Omega^k(M))$ for $k \geq 0$ as follows:

$$
\sharp \left( \d_\mu^* \alpha \right) = \mathrm{div}_\mu \left( \sharp \alpha \right), \quad \forall \alpha \in \Gamma^{(s+1)}(\Omega^{k+1}(M)).
$$
In particular, one can check that for $A \in \Gamma^{(s)}(M)$, $\mathrm{div}_\mu$ satisfies:

$$
\mathrm{div}_\mu(A) = \mathrm{div}(A) + \langle \nabla \log \rho, A \rangle,
$$
where $\rho =\d \mu / \d \vol$.

\begin{coro} \label{Coro décompo hodge}
    Let $A \in \Gamma^{(s)}_\mu(TM)$ be a $H^s(\mu)$ vector field. Then, $A$ decomposes in a unique way as
    
    \begin{equation*}
        A = \nabla \phi(\mu) + Y(\mu), 
    \end{equation*}
    where $\phi(\mu) \in H^{s+1}(M,\mathbb{R},\mu)$ and $Y(\mu) \in \Gamma^{(s)}_\mu(TM)$ with $\mathrm{div}_\mu(Y) = 0$. This decomposition is orthogonal in $L^2(M,TM,\mu)$.
\end{coro}

\begin{proof}
Let $A \in \Gamma^{(s)}_\mu(TM)$, by applying \eqref{E décompo hodge} to $\sharp^{-1}(A)$ we obtain:

$$
\sharp^{-1}(A) = \d \phi + \d_\mu^* \omega_2 + \omega_3,
$$
where $\phi \in \Gamma^{(s+1)}_\mu(TM)$, $\omega_2 \in \Gamma^{(s+1)}_\mu(\Omega^2(M))$ and $\omega_3 \in \Gamma^{(s)}_\mu(\Omega^1(M))$ is a harmonic form. Hence,

$$
A = \nabla f + \mathrm{div}_\mu ( \sharp \omega_2) + \sharp \omega_3.
$$
Since the space of harmonic forms is included in $\ker (\mathrm{d}_\mu^*)$ (see \cite[Lemma 9.1.1]{Petersen16} for instance) and $(\mathrm{d}_\mu^*)^2 = 0$, the vector field $Y = \mathrm{div}_\mu ( \sharp \omega_2) + \sharp \omega_3$ satisfies $\mathrm{div}_\mu(Y) = 0$.
\end{proof}

The projection on the vector space of $\mu$-divergence free vector fields is also known as the Leray projection and writes

$$
P(\mu) : \mathrm{Id} - \nabla \Delta^{-1}_{\mu}(\mathrm{div}_\mu(\cdot)).
$$
This is well known that $P(\mu)$ is a pseudo-differential operator of order $0$. As a consequence, $P(\mu)$ extends in a unique way as a continuous linear operator from $\Gamma^{(s)}_\mu(TM)$ to $\Gamma^{(s)}_\mu(TM)$ for any $s \in \mathbb{Z}$. The same result holds directly for $P(\mu)^\perp$. 

\subsection{Lemmas related to Section \ref{Section 4}} \label{Appendice preuve du lemme}

\begin{lem}\label{LemmeTransportOptimal}
    Let $(U_t)_{t \geq 0}$ be a family of measurable maps from $M$ to $M$. Then,
    
    \begin{equation*}
        W_2^2(\mu_t,\mu_s) \leq \int_M d_M(U_s(x),U_t(x))^2 \d\mu(x),
    \end{equation*}
    where $\mu_t = (U_t)_* \mu$.
    \end{lem} 
    
    \begin{prop} \label{Existence voisinage tubulaire}
        Let $M$ be a compact submanifold $M$ of $\mathbb{R}^d$. Then, there exists $\varepsilon> 0$ and a tubular neighborhood $M_\varepsilon := \{ x \in \mathbb{R}^d \, : \, d(x,M) < \varepsilon \}$ such that the projection $P_\varepsilon : M_\varepsilon \rightarrow M$ which minimize the distance, i.e. $| P(x) - x| = \inf_{y \in M}| y -x|$ for all $x \in M_\varepsilon$, is smooth.
    \end{prop}
    
    \begin{proof}
    This is a classical result in differential geometry, see for example \cite{Lee1997}.
    \end{proof}

    \begin{lem} \label{Equivalence distances}
    		Let $M$ be a closed Riemannian submanifold of $\mathbb{R}^d$. There exists a constant $\kappa>0$ such that
    		
    		$$
    		|x-y| \leq d_M(x,y) \leq \kappa |x-y|, \quad \forall x,y \in M.
    		$$
    \end{lem}
    
    \begin{proof}
    		The left inequality is clear. For the second one, let $\varepsilon > 0$ and $M_\varepsilon := \{ x \in \mathbb{R}^d \, : \, d(x,M) < \varepsilon \}$ such that $P_\varepsilon$ is smooth (see Proposition \ref{Existence voisinage tubulaire}).  If $x,y \in M$ are such that $|x-y| \geq \varepsilon$, then the quantity 
    		
    	$$
    	\kappa_1 := \sup \left\{ \frac{d_M(x,y)}{|x-y|} \, : \, |x-y| \geq \varepsilon  \right\}
    	$$
    	is finite by continuity and compactness. Suppose that $|x-y| < \varepsilon$. Let $\gamma_t(x,y) := (1-t)x + ty$, it is clear that
    	
    	$$
    	d_M(x,y) \leq \int_0^1 \left|\frac{\d}{\d t} P_\varepsilon(\gamma_t(x,y)) \right| \d t.
    	$$
    	But,
    	
    	$$
    	\left| \frac{\d}{\d t}P_\varepsilon (\gamma_t(x,y)) \right| = \left| {TP_\varepsilon}_{ \gamma_t(x,y)}(y-x) \right| \leq |TP_\varepsilon|_\infty |x-y|.
    	$$
    	By denoting $\kappa_2 = |TP_\varepsilon|_\infty$ we obtain the inequality by choosing $\kappa= \max(\kappa_1,\kappa_2)$.
    \end{proof}

We include a proof of Lemma \ref{Lemme suite cauchy} in this appendix.

    \begin{lem}[= Lemma \ref{Lemme suite cauchy}]
        There exists $t_0 > 0$ such that
        
        $$
        \lim_{n,l \to \infty} \mathbb{E} \left[\sup_{t \in [0,t_0]}d_M^2(X_t^{n}(x),X_t^l(x)) \right]^{1/2} = 0.
        $$
        Moreover, $t_0$ depends only on the vector fields $\tilde{Z}_i$ and may be chosen independently of the initial condition $(x,\mu)$.
    \end{lem}

    \begin{proof}[Proof of Lemma \ref{Lemme suite cauchy}]
    Let $n \geq 0$. Let $X_t^n$ denote $X_t^n(x)$ and $M_t^n$ denote $M_t^n(\mu)$.
    Using Itô's formula and inequalities \eqref{C1}, \eqref{C2} we get
    
    \begin{align*}
    \d |X_t^{n+1}&-X_t^n|^2 = 2 \langle \d (X_t^{n+1} -X_t^n),X_t^{n+1} -X_t^n\rangle + |\d (X_t^{n+1} -X_t^n)|^2 \\
    & \leq K \left(|X_t^{n+1}-X_t^n|^2 + W_2^2(M^n_t, M^{n-1}_t) \right) \d t + 2 \SN \left\langle \tilde{Z}_i(X^{n+1}_t, M^n_t) - \tilde{Z}_i(X^{n}_t, M^{n-1}_t), X_t^{n+1} -X_t^n \right \rangle \d W^i_t.
    \end{align*}
    Now, by using Burkholder-Davis-Gundy inequality, Cauchy-Schwarz inequality,  \eqref{C2} and the fact that $2ab \leq a^2 + b^2$, we obtain
    
    \begin{align} \label{Grosse inégalité}
    &\mathbb{E} \left[ \sup_{t \in [0,t_0]} \left| \int_0^t 2 \SN \left\langle \tilde{Z}_i(X^{n+1}_s, M^n_s) - \tilde{Z}_i(X^{n}_s, M^{n-1}_s), X_s^{n+1} -X_s^n \right \rangle \d W^i_s \right| \right] \nonumber \\
    & \leq c_1 \mathbb{E} \left [ \left(\int_0^{t_0}  \SN \left\langle \tilde{Z}_i(X^{n+1}_s, M^n_s) - \tilde{Z}_i(X^{n}_s, M^{n-1}_s), X_s^{n+1} -X_s^n \right \rangle^2 \d t \right)^{\frac{1}{2}}  \right]\nonumber \\
    & \leq c_1 \mathbb{E} \left [ \left(\int_0^{t_0}  \SN \left| \tilde{Z}_i(X^{n+1}_s, M^n_s) - \tilde{Z}_i(X^{n}_s, M^{n-1}_s)\right|^2 \, \left| X_s^{n+1} -X_s^n \right |^2 \d t \right)^{\frac{1}{2}}  \right] \nonumber \\
    & \leq \mathbb{E} \left [ \left(  \sup_{t \in [0,t_0]} \left| X_t^{n+1} -X_t^n \right |^2 \int_0^{t_0} c_1K \left(|X_s^{n+1}-X_s^n|^2 + W_2^2(M^n_s, M^{n-1}_s) \right) \d s \right)^{\frac{1}{2}}  \right] \nonumber \\
    & \leq \frac{1}{2} \mathbb{E} \left [  \sup_{t \in [0,t_0]} \left| X_t^{n+1} -X_t^n \right |^2 \right] + \frac{c_1}{2} \mathbb{E}\left[ \int_0^{t_0} K \left(|X_s^{n+1}-X_s^n|^2 + W_2^2(M^n_s, M^{n-1}_s) \right) \d s \right], 
    \end{align}
    where $c_1$ is the constant in the Burkholder-Davis-Gundy inequality.
    Thus, using Lemmas \ref{LemmeTransportOptimal} and \ref{Equivalence distances}  there exists $\kappa >0$ such that:
    
    $$
    W_2^2(M^n_s, M^{n-1}_s) \leq \int_M d_M^2(X^n_s(x), X^{n-1}_s(x)) \d\mu(x) \leq\kappa \int_{\mathbb{R}^d} |X^n_s(x) - X^{n-1}_s(x)|^2 \d\mu(x).
    $$
    Hence, by implementing this inequality in \eqref{Grosse inégalité}, we get
    
    \begin{align*}
    \mathbb{E}\left[\sup_{t \in [0,t_0]} |X_t^{n+1}-X_t^n|^2 \right] &\leq c_2 \mathbb{E}\left[ \int_0^{t_0} K \left(|X_s^{n+1}-X_s^n|^2 + \kappa \int_{\mathbb{R}^d} |X_s^n- X^{n-1}_s|^2 d\mu \right) \d s \right] \\
    & \leq  \int_0^{t_0} c_2K \mathbb{E}\left[ \sup_{r \in [0,s]}|X_r^{n+1}-X_r^n|^2\right]\d s + c_2K \kappa \int_0^{t_0} \mathbb{E} \left[ \int_{\mathbb{R}^d} |X_s^n- X^{n-1}_s|^2 d\mu \right] \d s,
    \end{align*}
    where $c_2 = c_1 + 1$. By using Grönwall's lemma, we obtain
    
    \begin{align} 
    \mathbb{E}\left[\sup_{t \in [0,t_0]} |X_t^{n+1}-X_t^n|^2 \right] &\leq  c_2K \kappa e^{t_0c_2K} \int_0^{t_0} \mathbb{E} \left[ \int_{\mathbb{R}^d} |X_s^n- X^{n-1}_s|^2 d\mu \right] \d s \nonumber \\
     &\leq c(t_0) \sup_{t \in [0,t_0]} \mathbb{E}\left[ \int_{\mathbb{R}^d} |X_t^n- X^{n-1}_t|^2 \d\mu  \right], \label{GronwallChpsDeVecteursGeneraux}
    \end{align}
    where $c(t_0) = t_0c_2K \kappa e^{t_0c_2K}$. Notice that $c(t_0)$ depends only on $t_0$ because the constants in BDG inequality are independent of the martingales, is continuous in $t_0$, is increasing and $c(0) =0$. Note that
    
    \begin{equation} \label{E A6}
        \sup_{t \in [0,t_0]}\mathbb{E}\left[ |X_t^{n+1}-X_t^n|^2 \right]  \leq \mathbb{E}\left[\sup_{t \in [0,t_0]} |X_t^{n+1}-X_t^n|^2 \right].
    \end{equation}
    By implementing \eqref{E A6} in \eqref{GronwallChpsDeVecteursGeneraux} and integrating with respect to $\mu$, we obtain
    
    $$
    \sup_{t \in [0,t_0]}\mathbb{E}\left[ \int_{\mathbb{R}^d} |X_t^{n+1}-X_t^n|^2  \d\mu \right] \leq c(t_0) \sup_{t \in [0,t_0]} \mathbb{E}\left[ \int_{\mathbb{R}^d} |X_t^n- X^{n-1}_t|^2 \d\mu \right].
    $$
    Iterating this inequality $n$-times and implementing it in \eqref{GronwallChpsDeVecteursGeneraux} leads us to
	
	$$
	\mathbb{E}\left[\sup_{t \in [0,t_0]} |X_t^{n+1}-X_t^n|^2 \right]  \leq c(t_0)^n \sup_{t \in [0,t_0]} \mathbb{E}\left[ \int_{\mathbb{R}^d} |X_t^1- x|^2 \d\mu \right].
	$$
	By choosing $t_0 > 0$ such that $c(t_0) < 1$ (we can always do it because of the properties of the function c mentioned above), we obtain that the series with general term
	
	$$
	\mathbb{E}\left[ \sup_{t \in [0,t_0]} \int_{\mathbb{R}^d} |X_t^{n+1}-X_t^n|^2  \d\mu \right]^{1/2}
	$$
	is convergent. The conclusion follows from Lemma \ref{Equivalence distances}.

    \end{proof}

\subsection{General results about the diffeomorphism group}

We recall that we consider $\mathscr{D}$ as the following Inverse Limit Hilbert Lie group:

$$
\mathscr{D} = \cap_{s > m/2} \mathscr{D}^s,
$$
and that, for all $s > m/2$, the tangent space is given by:

$$
T_\varphi \mathscr{D}^s = \left \{ A \circ \varphi \, : \, A \in \Gamma(TM), \, A \circ \varphi \in H^s(M,TM) \right \}.
$$
In particular, the map

\begin{align*}
    \bm \exp_\varphi : T_\varphi \mathscr{D}^s &\longrightarrow \mathscr{D}^s  \\
     A \circ \varphi &\longmapsto \bm \exp_\varphi(A \circ \varphi) = \exp_{\cdot}(A (\cdot)) \circ \varphi,
\end{align*}
defines a smooth chart on some neighbourhood of the zero vector in $T_\varphi \mathscr{D}^s$. Let us define the map

\begin{align*}
E_{x_0} : \mathscr{D}^s &\longrightarrow M \\
\varphi &\longmapsto \varphi(x_0),
\end{align*}
for $s > m/2$.

 We have the following property:

\begin{prop} \label{P Evaluation map lisse}
	Let $x_0 \in M$. The map $E_{x_0} : \mathscr{D}^s \rightarrow M$ is of class $\C^\infty$ for all $s >m/2$. In particular, $E_{x_0} : \mathscr{D} \rightarrow M$ belongs to $\C^\infty(\mathscr{D},M)$.
\end{prop}

\begin{proof}
Let $s > m/2$, $x_0 \in M$ and $\varphi \in \mathscr{D}^s$. Let $(\mathcal{U}_0, \bm \exp_\varphi^{-1})$ be a chart of $\varphi \in \mathscr{D}^s$ with $\tilde{U}_0 = \bm \exp_\varphi^{-1}(\mathcal{U}_0)$ and $(\mathcal{U}_1, \theta)$ be a chart of $x_0 \in M$ with $\tilde{\mathcal{U}}_1= \theta(\mathcal{U}_1)$. The induced map

$$
\tilde{E}_{x_0} : \tilde{\mathcal{U}}_0 \longrightarrow \tilde{\mathcal{U}}_1,
$$
defined to make the following diagram commutative

\begin{center}
\begin{tikzpicture}[>=stealth]
  \node (E1) at (0,1) {$\mathcal{U}_0$};
  \node (E2) at (4,1) {$\mathcal{U}_1$};
  \node (E3) at (0,0)   {$\tilde{\mathcal{U}}_0$};
  \node (E4) at (4,0)   {$\tilde{\mathcal{U}}_1$};

  \draw[->, >=latex] (E1) -- (E2) node[midway,above]{$E_{x_0}$}; 
  \draw[->, >=latex] (E3) -- (E1) node[midway,left]{$\bm \exp_\varphi$};
  \draw[->, >=latex] (E3) -- (E4) node[midway,below]{$\tilde{E}_{x_0}$}; 
  \draw[->, >=latex] (E2) -- (E4) node[midway,right]{$\theta$};
\end{tikzpicture}
\end{center}
writes

$$
\tilde{E}_{x_0}(A \circ \varphi) = \theta\left( \exp_{\varphi(x_0)}(A(\varphi(x_0))) \right) = \theta \left( \exp_{\varphi(x_0)}(E'_{x_0}(A \circ \varphi))  \right),
$$
where 

\begin{align*}
E'_{x_0} : T_\varphi \mathscr{D}^s & \longrightarrow T_{\varphi(x_0)}M \\
A \circ \varphi & \longmapsto A(\varphi(x_0)).
\end{align*}
The map $E'_{x_0}$ is obviously linear and by Sobolev injection, we have:

$$
|E'_{x_0}(A \circ \varphi)|_{\varphi(x_0)} = |A(\varphi(x_0))|_{\varphi(x_0)} \leq \sup_{x \in M}|A(\varphi(x_0))| \leq C | A \circ \varphi |_{H^s}.
$$
Consequently, $E'_{x_0}$ is smooth as a continuous linear map, thus, so is $\tilde{E}_{x_0}$ as a composition of smooth maps. Hence, $E_{x_0}$ is smooth.
\end{proof}

 The following result is due to the functoriality of the pullback.
    
    \begin{prop} \label{pullback fonctorielle}
        Let $\mathscr{M}, \mathscr{N}$ be Hilbert manifolds, $k \geq 1$ and $\nabla$ be a connection on $T\mathscr{N}$. Let $\iota : \mathscr{M} \rightarrow \mathscr{N}$ be a smooth map and $A \in \Gamma^k(T\mathscr{N})$ (resp. $J \in \mathrm{Hom}_{\mathrm{id}}^k(T\mathscr{N},T\mathscr{N})$). Then, $\iota^*(A) \in \Gamma^k(\iota^*(T\mathscr{N}))$ (resp. $\iota^* J \in \mathrm{Hom}^k_{\mathrm{id}}(\iota^*(T\mathscr{N})), \iota^*(T\mathscr{N})$). Moreover, the following holds:
        
        $$
        \nabla_B^*(\iota^*(A)) = \iota^*(\nabla_{\iota_*(B)}A) \quad \text{and} \quad \nabla^*_{B}(\iota^*(J)) = \iota^*(\nabla_{\iota_*(B)}J), \quad \forall B \in \Gamma^0(T\mathscr{M}).
        $$
    \end{prop}

	\bibliographystyle{amsalpha}
	\bibliography{sample}

\end{document}